\documentclass[10pt,english]{amsart}

\usepackage{amsfonts}
\usepackage{mathrsfs}
\usepackage{amsmath}
\usepackage{amsthm}
\usepackage{amssymb}
\usepackage{latexsym}
\usepackage[all]{xy}
\usepackage{verbatim}
\usepackage{graphicx}
\usepackage{amscd,amsbsy}
\usepackage{nccmath}
\usepackage{stmaryrd}
\SetSymbolFont{stmry}{bold}{U}{stmry}{m}{n} 
\usepackage{etoolbox}
\usepackage{enumitem}
\usepackage{dsfont}
\usepackage{bbm}
\usepackage[
    pdfborder={0 0 0},
	colorlinks,
	plainpages,
	citecolor=red!50!black,
    filecolor=Darkgreen,
    linkcolor=blue!40!black,
    urlcolor=cyan!50!black!90]{hyperref}

\usepackage{babel}
\usepackage{csquotes}
\MakeOuterQuote{"}

\newtheorem{theorem}{Theorem}[section]
\newtheorem{proposition}[theorem]{Proposition}
\newtheorem{corollary}[theorem]{Corollary}
\newtheorem{lemma}[theorem]{Lemma}
\newtheorem{definition}[theorem]{Definition}
\newtheorem{remark}[theorem]{Remark}

\newcommand{\nc}{\newcommand}

\nc{\tl}{\tilde}
\nc{\wt}{\widetilde}
\nc{\wh}{\widehat}
\nc{\mc}{\mathcal}
\nc{\mf}{\mathfrak}
\nc{\ms}{\mathsf}
\nc{\mr}{\mathrm}

\newcommand{\al}{\alpha}

\newcommand{\eps}{\epsilon}

\newcommand{\ka}{\kappa}
\newcommand{\la}{\lambda}

\newcommand{\vphi}{\varphi}

\newcommand{\gr}{\mathrm{gr}\,}

\newcommand{\mbf}{\mathbf{f}}
\newcommand{\mbF}{\mathbf{F}}

\newcommand{\mbG}{\mathbf{G}}
\newcommand{\mbH}{\mathbf{H}}

\newcommand{\mbu}{\mathbf{u}}

\newcommand{\mcE}{\mathcal{E}}
\newcommand{\mcF}{\mathcal{F}}

\newcommand{\mcI}{\mathcal{I}}
\newcommand{\mcJ}{\mathcal{J}}
\newcommand{\mcK}{\mathcal{K}}

\newcommand{\mcR}{\mathcal{R}}

\newcommand{\mcT}{\mathcal{T}}

\newcommand{\mcY}{\mathcal{Y}}

\newcommand{\mcZ}{\mathcal{Z}}

\newcommand{\mfa}{\mathfrak{a}}

\newcommand{\mfg}{\mathfrak{g}}
\newcommand{\mfgl}{\mathfrak{g}\mathfrak{l}}

\newcommand{\mfh}{\mathfrak{h}}

\newcommand{\mfS}{\mathfrak{S}}
\newcommand{\mfsl}{\mathfrak{s}\mathfrak{l}}
\newcommand{\mfso}{\mathfrak{s}\mathfrak{o}}
\newcommand{\mfsp}{\mathfrak{s}\mathfrak{p}}

\newcommand{\mfz}{\mathfrak{z}}

\newcommand{\msF}{\mathsf{F}}

\newcommand{\End}{\mathrm{End}}

\newcommand{\Ker}{\mathrm{Ker}}

\newcommand{\mdD}{\mathds{D}}
\newcommand{\mdF}{\mathds{F}}
\newcommand{\mdK}{\mathds{K}}
\newcommand{\mdL}{\mathds{L}}
\newcommand{\mdT}{\mathds{T}}
\newcommand{\mdX}{\mathds{X}}

\newcommand{\iso}{\stackrel{\sim}{\longrightarrow}}

\newcommand{\into}{\hookrightarrow}
\newcommand{\onto}{\twoheadrightarrow}

\newcommand{\C}{\mathbb{C}}
\newcommand{\Z}{\mathbb{Z}}

\newcommand{\ot}{\otimes}
\newcommand{\ol}{\overline}

\nc{\qu}{\quad}
\nc{\qq}{\qquad}

\usepackage[scr=boondoxo]{mathalfa}
\nc{\key}{{\mathscr{k}}}
\nc{\ley}{{\mathscr{l}}}
\nc{\emm}{{\mathscr{m}}}
\nc{\mt}{{\mathscr{t}}}
\nc{\mT}{{\mathscr{T}}}
\nc{\mX}{{\mathscr{X}}}
\nc{\my}{{\mathscr{y}}}
\nc{\mY}{{\mathscr{Y}}}
\nc{\mz}{{\mathscr{z}}}
\newcommand {\Omit}[1]{}

\DeclareMathOperator{\ad}{ad}

\numberwithin{equation}{section}

\renewcommand{\,}{\kern 0.1em} 

\usepackage{color}
\usepackage[usenames,dvipsnames]{xcolor}

\title{The $R$-matrix presentation for the Yangian of a simple Lie algebra}
\author{Curtis Wendlandt}
\address{Department of Mathematical and Statistical Sciences,
University of Alberta, CAB 632, Edmonton, AB T6G 2G1, Canada.}
\email{cwendlan@ualberta.ca}

\begin{document} 
\begin{abstract}
Starting from a finite-dimensional representation of the Yangian $Y(\mfg)$ for a simple Lie algebra $\mfg$ in Drinfeld's original presentation, we construct a Hopf algebra $X_\mcI(\mfg)$, called the extended Yangian, whose defining relations are encoded in a ternary matrix relation built from a specific $R$-matrix $R(u)$. We prove that there is a surjective Hopf algebra morphism $X_\mcI(\mfg)\onto Y(\mfg)$ whose kernel is generated as an ideal by the coefficients of a central matrix $\mcZ(u)$. When the underlying representation is irreducible, we show that this matrix becomes a grouplike central series, thereby making available a proof of a well-known theorem stated by Drinfeld in the 1980's. We then study in detail the algebraic structure of the extended Yangian, and prove several generalizations of results which are known to hold for  Yangians associated to classical Lie algebras in their $R$-matrix presentations.
\end{abstract}
\maketitle
{
\setlength{\parskip}{0ex}
\tableofcontents
}

\section{Introduction}
To any simple Lie algebra $\mfg$ one can associate a Hopf algebra $Y(\mfg)$, called the Yangian of $\mfg$, which is a filtered deformation of the enveloping algebra for the Lie algebra 
$\mfg[z]$ of polynomial maps $\C\to \mfg$. This quantum group originally appeared in disguise in the work of mathematical physicists studying quantum integrable systems and the quantum Yang-Baxter equation (see, for example, \cite{KS1,KS2}). The definition of $Y(\mfg)$ was later formalized in the pioneering paper \cite{Dr1}, where several foundational results were established. Since the 1980's, the study of Yangians has grown into a beautiful theory with applications to several areas, including, for instance, the theory of classical Lie algebras \cite{Mobook, MoGZ, NaSCap, NaCap, NaTaCap}, the study of finite $W$-algebras and their representations \cite{Brown, BrownRep, BK1, BKmem, BRWalg, RtwistWalg, RSWalg}, the theory of classical $W$-algebras and affine vertex algebras \cite{MoMuWalg, MoMuWalg2,MoFFC}, as well as geometric representation theory \cite{MO, FinkRyb, FKPRW, KWWY, KTWWY, Nak, SV1, SV2, Va, YaGu1, YaGu2, YaGu3}.

Yangians admit at least three important presentations: Drinfeld's original (or "$J$") presentation, the $R$-matrix (or $RTT$) realization, and the Drinfeld "new" (or current) presentation \cite{Dr1,Dr2,FRT}. Many applications of $Y(\mfg)$ are specific to $\mfg=\mfsl_N$ and employ the $R$-matrix realization of $Y(\mfsl_N)$, which has a rich history (see the monograph \cite{Mobook}). In this setting, the evaluation morphism $Y(\mfsl_N)\onto U(\mfsl_N)$, which only exists for $\mfg=\mfsl_N$, is particularly simple to describe and this phenomena gives rise to many interesting results. The $R$-matrix presentation of $Y(\mfg)$ has also been studied for orthogonal and symplectic Lie algebras \cite{AACFR,AMR}, and this has led to a more explicit description of the relationship between Yangian characters, classical $W$-algebras, and the centers of vertex algebras at the critical level \cite{MoMuWalg}. It has also served as the catalyst for the study of twisted Yangians of type B-C-D and their representations \cite{GR, smrank, RepsI, RepsII, IMO}. 

The equivalence between the $J$ and $R$-matrix presentations of the Yangian was succinctly explained in \cite[Theorem 6]{Dr1}. The idea is as follows: starting from a finite-dimensional irreducible representation $V$ of the Yangian $Y(\mfg)$ in the $J$-presentation, one can build a Hopf algebra called the extended Yangian (which we denote $X(\mfg)$) defined by a ternary matrix relation called the $RTT$-relation. Using the existence of the universal $R$-matrix for $Y(\mfg)$ (see Theorem \ref{T:DrThm3}), one then constructs a surjective Hopf algebra morphism $\wt \Phi:X(\mfg)\onto Y(\mfg)$. According to Drinfeld, the kernel of this morphism is generated by the coefficients of a grouplike central series $c(u)\in X(\mfg)[\![u^{-1}]\!]$. That is, 
the coefficients of $c(u)$ are central elements and $\Delta(c(u))=c(u)\ot c(u)$, where $\Delta$ denotes the coproduct of $X(\mfg)$. The quotient $X(\mfg)/\Ker \wt \Phi$, which we denote $Y_R(\mfg)$, is the so-called $R$-matrix realization of the Yangian. 

However, this construction has only been explicitly written down and studied in the special cases alluded to above. In these instances, $\mfg=\mfsl_N,\mfso_N$ or $\mfsp_N$ and the underlying representation $V$ of $Y(\mfg)$ is the vector representation $\C^N$. Moreover, the proof of \cite[Theorem 6]{Dr1}, which in principle should explain how to construct the series $c(u)$, has never appeared in the literature in full generality. We note, however, that for the special case where $\mfg=\mfso_N$ or $\mfsp_N$ and  $V=\C^N$ a proof was given in \cite{GRW}. 

This brings us to the original motivation and first main goal of this paper: to make available a detailed proof of \cite[Theorem 6]{Dr1}. In fact, we take a slightly more general approach. After recalling the definition of $Y(\mfg)$ in its $J$-presentation in Section \ref{Sec:YJ} and obtaining the polynomial current algebra version of \cite[Theorem 6]{Dr1} in Section \ref{Sec:r-matrix}, we construct in Section \ref{Sec:RTT} the $RTT$-Yangian $Y_R(\mfg)$ and its extension $X_\mcI(\mfg)$ for any non-trivial finite-dimensional $Y(\mfg)$-module $V$. Here $\mcI$ is an indexing set which keeps track of the dimension of the endomorphism space $\End_{Y(\mfg)}V$, and is omitted as a subscript of $X_\mcI(\mfg)$ when $V$ is irreducible. We will prove in Section \ref{Sec:YR->YJ} that, even when $V$ is not irreducible, the Yangian $Y_R(\mfg)$ is isomorphic to $Y(\mfg)$: see Theorem \ref{T:YR->YJ}. As an immediate corollary to the proof of this result, one obtains Theorem \ref{T:PBW}, which gives a Poincar\'{e}-Birkhoff-Witt theorem for $Y_R(\mfg)$. The actual statement and proof of Theorem 6 in \cite{Dr1} is collected later in Section \ref{Sec:Drin-Vec}, which is solely devoted to the case when $V$ is irreducible: see Theorem \ref{T:Drin}. Our argument also gives a concrete description of the series $c(u)$ (or at least one choice for $c(u)$) in terms of the generating matrix $T(u)$ for $X(\mfg)$ and its image under the square of the antipode: see Corollary \ref{C:z,y}, where $c(u)$ is denoted $z(u)$. As a disclaimer, we note that Part \eqref{Dr:2} of Theorem \ref{T:Drin} does depend on the assumption that $V$ is irreducible. In the general setting, the formal series $c(u)$ is replaced by a matrix $C(u)$ which satisfies similar properties: see Remark \ref{R:Drin}. 

The proof of Theorem \ref{T:YR->YJ} makes use of the so-called $r$-matrix presentation of the current algebra $\mfg[z]$. This presentation is very similar in flavour to the $R$-matrix realization of $Y(\mfg)$, except that the role played by the universal $R$-matrix of $Y(\mfg)$ is instead played by the classical $r$-matrix $\frac{\Omega}{u-v}$ associated to the standard Lie bialgebra structure of $\mfg[z]$. Since no general treatment of this presentation seems to exist in the literature, we have devoted Section \ref{Sec:r-matrix} to its construction and to obtaining an analogous presentation for a certain extension of $\mfg[z]$ which is closely related to $X_\mcI(\mfg)$. As was suggested in the previous paragraph, the equivalence of the standard and $r$-matrix presentations of $\mfg[z]$, which is given in Propositions \ref{P:r-matrix} and \ref{P:g_I[z]}, can be viewed as the classical version of \cite[Theorem 6]{Dr1}. 

Let us now describe the second main goal of this paper. When $\mfg$ is a classical Lie algebra and $V$ is its vector representation, the extended Yangian $X(\mfg)$ is often studied in place of its quotient $Y_R(\mfg)$. Whereas the center of $Y_R(\mfg)$ is trivial, $X(\mfg)$ has a large center which is isomorphic to a polynomial algebra in countably many variables, and which can conveniently be described using certain explicit formal series. When $\mfg=\mfsl_N$, the study of these series and their twisted Yangian analogues has led to applications in studying the centers of $U(\mfgl_N)$, $U(\mfso_N)$ and $U(\mfsp_N)$ (see \cite[Chapter 7]{Mobook}). It is also known that one can describe $Y_R(\mfg)$ (for $\mfg=\mfsl_N,\mfso_N$ and $\mfsp_N$) not only as a quotient of $X(\mfg)$, but also as the subalgebra of $X(\mfg)$ fixed by a certain family of automorphisms (see Subsection \ref{ssec:class}). These considerations naturally lead to the question of whether or not the structure of $X_\mcI(\mfg)$ for general $\mfg$ and $V$ can be described in more detail, and in particular if some of the results which characterize $X_\mcI(\mfg)$ in the aforementioned special cases can be proven in general.

The second goal of this paper, which is considered in Section \ref{Sec:XR}, is to provide an affirmative answer to this question with as much detail as possible.  Our first result in this direction is Theorem \ref{T:XR->CxYR}, which proves that $X_\mcI(\mfg)$ is always isomorphic to the tensor product of a polynomial algebra in countably many variables with the Yangian $Y_R(\mfg)$. Not only does this prove that $Y_R(\mfg)$ can be identified with a subalgebra of $X_\mcI(\mfg)$, but it shows that the center of $X_\mcI(\mfg)$ is a polynomial algebra. In Proposition \ref{P:Z(u)}, explicit algebraically independent generators of the center are identified. Our next main result is a Poincar\'{e}-Birkhoff-Witt type theorem for $X_\mcI(\mfg)$: see Theorem \ref{T:X-PBW}. This result proves that $X_\mcI(\mfg)$ can be viewed as a filtered deformation 
of the enveloping algebra for the current algebra $(\mfg\oplus \mfz_\mcI)[z]$, where $\mfz_\mcI$ is a commutative Lie algebra of dimension $\dim \End_{Y(\mfg)}V$. Additionally, it demonstrates that the enveloping algebra of $\mfg$ is always contained in $X_\mcI(\mfg)$ as a subalgebra. In Subsection \ref{ssec:Y->X}, we prove that the embedding $Y_R(\mfg)\into X_\mcI(\mfg)$ furnished by Theorem \ref{T:XR->CxYR} is a Hopf algebra morphism and study the behaviour of the center of $X_\mcI(\mfg)$ with respect to its Hopf structure: see Proposition \ref{P:Hopf}. The last result relevant to the second main goal of our paper 
is Theorem \ref{T:fixed-pt}, which proves that $Y_R(\mfg)$ can be realized as the subalgebra of $X_\mcI(\mfg)$ consisting of all elements stable under a specific family of automorphisms. In Subsection \ref{ssec:class} of Section \ref{Sec:Drin-Vec}, it is explained in more detail how the results of Section \ref{Sec:XR} generalize results which are known to hold when $\mfg$ is a classical Lie algebra and $V$ is its vector representation \cite{AACFR,AMR,Mobook}. 

We now give a few remarks, the first of which concerns the current presentation $Y_D(\mfg)$ of the Yangian.
In \cite[Theorem 1]{Dr2}, Drinfeld established that the $J$ and  current realizations of the Yangian were isomorphic, and also gave an an explicit formula for an isomorphism  $Y(\mfg)\to Y_D(\mfg)$. A proof of this result was not published at the time, but one was recently made available in \cite[Theorem 2.6]{GRW}, where $Y_D(\mfg)$ was denoted $Y^{\rm cr}(\mfg)$. By composing this map with the morphism of Theorem \ref{T:YR->YJ}, one obtains an isomorphism $Y_R(\mfg)\to Y_D(\mfg)$ for each finite-dimensional non-trivial $Y(\mfg)$-module $V$. We remark that, when $\mfg$ is assumed to be a classical Lie algebra and $V$ its vector representation, such an isomorphism has also been established using the Gauss decomposition of the generating matrix for $X(\mfg)$. For $\mfg=\mfsl_N$, this was accomplished in 
\cite{BKpara}, while for $\mfg=\mfso_N$ and $\mfg=\mfsp_N$ this was achieved in the recent paper \cite{JLM}.  

As a last remark, we note that due to deep parallels between the theories of Yangians and quantum loop algebras \cite{GM,GTL1,GTL2,GTL3}, it is reasonable to expect that the results of this paper could be proven, to some extent, for the quantum loop algebra associated to an arbitrary simple Lie algebra.

\noindent {\it Acknowledgements.} The author gratefully acknowledges the financial support of the Natural Sciences and Engineering Research Council of Canada provided via the Alexander Graham Bell Canada Graduate Scholarship (CGS D). He would also like to thank Nicolas Guay and the anonymous reviewers for several helpful comments. 

\section{Preliminaries}\label{Sec:Pre}

%
%
\subsection{Simple Lie algebras and their polynomial current algebras}

Throughout this paper we assume that $\mfg$ is a finite-dimensional complex simple Lie algebra with symmetric non-degenerate invariant bilinear form $(\cdot,\cdot)$. Following the notation of \cite{Dr1}, we fix an orthonormal basis $\{X_\lambda\}_{\lambda\in \Lambda}$ of 
$\mfg$ with respect to this form, where $\Lambda$ is an indexing set of size $\dim \mfg$. Let $\{\al_{\lambda\nu}^\gamma\}_{\lambda,\nu,\gamma\in \Lambda}$ be the structure constants with respect to this basis:

\begin{equation*}
 [X_\lambda,X_\nu]=\sum_{\gamma\in \Lambda}\al_{\lambda\nu}^\gamma X_\gamma.
\end{equation*}
In particular, $\al_{\lambda\nu}^\gamma=-\al_{\nu\lambda}^\gamma$ and $\al_{\lambda\nu}^\gamma=-\al_{\lambda\gamma}^\nu$ for all $\lambda,\nu,\gamma\in \Lambda$, the second of these equalities being a consequence of the invariance of the bilinear form $(\cdot,\cdot)$. 

Let $\Omega$ and $\omega$ denote the Casimir elements
\begin{equation*}
 \Omega=\sum_{\lambda\in \Lambda} X_\lambda \otimes X_\lambda \in \mfg\ot\mfg \quad \text{ and }\quad \omega=\sum_{\lambda\in \Lambda} X_\lambda^2\in U(\mfg),
\end{equation*}
and let $c_\mfg$ denote the eigenvalue of $\omega$ in the adjoint representation. Here $U(\mfg)$ denotes the enveloping algebra of $\mfg$. More generally, the notation $U(\mfa)$ will be used to denote the enveloping algebra of an arbitrary complex Lie algebra $\mfa$, and $\Delta$ will denote the standard coproduct on $U(\mfa)$.

The polynomial current algebra of a complex Lie algebra $\mfa$ is the Lie algebra which is equal to $\mfa[z]=\mfa\otimes \C[z]$ as a vector space, with Lie bracket given by 
\begin{equation*}
 [X\otimes f(z),Y\otimes g(z)]=[X,Y]_\mfg\otimes f(z)g(z) \quad \text{ for all }\; X,Y\in \mfa \; \text{ and }\; f(z),g(z)\in \C[z].
\end{equation*}
Equivalently, $\mfa[z]$ is the space of polynomial maps $\C\to \mfg$ with Lie bracket given pointwise. 
If $\mfa=\mfg$ is a complex simple Lie algebra, then the enveloping algebra $U(\mfg[z])$ is isomorphic to the unital associative algebra generated by elements $\{X_\lambda z^r\,: \, \lambda \in \Lambda, \, r\geq 0\}$ subject to the defining relations 
\begin{equation}
 [X_\lambda z^r, X_\mu z^s]=\sum_{\gamma\in \Lambda}\al_{\lambda \mu}^\gamma X_\gamma z^{r+s} \; \text{ for all }\; \lambda,\mu\in \Lambda \; \text{ and }\; r,s\geq 0. \label{g[z]}
\end{equation}
The Lie algebra $\mfa[z]$ is graded: we have $\mfa[z]=\bigoplus_{k\geq 0} \mfa z^k$, with $\mfa z^k=\mfa\ot \C z^k$. If $\mfa=\mfg$ is simple, then 
$\mfg[z]$ is generated as a Lie algebra by $\mfg$ and $\mfg z$.

In addition to having the structure of a Lie algebra, $\mfg[z]$ admits the structure of a coboundary Lie bialgebra determined by the  classical 
$r$-matrix 
\begin{equation*}
 r_\mfg=-\!\!\sum_{\lambda\in \Lambda,k\geq 0} X_\lambda v^k \ot X_\lambda u^{-k-1}\in \mfg[v]\wh \ot \mfg(\!(u^{-1})\!).
\end{equation*}
That is, its Lie bialgebra cocommutator $\delta:\mfg[z]\to \mfg[z]\ot \mfg[z]\cong (\mfg\ot \mfg)[v,u]$ is given by
\begin{equation*}
 \delta(f(z))(u,v)=[f(v)\ot 1+1\ot f(u),r_\mfg] \quad \forall \quad f(z)\in \mfg[z]. 
\end{equation*}
 That the right-hand side of the above expression indeed belongs to $(\mfg\ot \mfg)[v,u]$ follows from the observation that $r_\mfg$ may be identified with the element 
\begin{equation*}
 -\frac{\Omega}{u-v}=-\sum_{k\geq 0}\Omega  v^k u^{-k-1}\in (\mfg\ot \mfg)\ot (\C[v])[\![u^{-1}]\!], 
\end{equation*}
 together with the fact that $[\Delta(X),\Omega]=0$ for all $X\in\mfg$. The statement that 
$r_\mfg$ is an $r$-matrix is meant to indicate that it is a solution of the classical Yang-Baxter equation with spectral parameter: see \cite[Subsection 6.3.2]{ES}, as well as 
Subsection 6.2 of \textit{loc. cit.} for a more complete description of the Lie bialgebra structure on $\mfg[z]$.

A deep understanding of the bialgebra $(\mfg[z],\delta)$ will not be needed here, although 
the $r$-matrix $\frac{\Omega}{u-v}$ will play a significant role. We, however, adapt the viewpoint that this element be treated as a rational function in $u-v$ which can 
be expanded as a formal series in $(\mfg\ot\mfg)\ot \C[\![v^{\pm 1},u^{\pm 1}]\!]$ in various ways: see Remark \ref{R:expand}. 

%
%
\subsection{Matrix, formal series, and miscellaneous notation} 
In what follows, all vector spaces and algebras are assumed to be over the complex numbers $\C$, and we will maintain this assumption for the remainder of this paper.

Suppose that $W$ is an arbitrary vector space and that $V$ is a finite-dimensional vector space of dimension $N$ with a fixed basis $\{e_1,\ldots,e_N\}$, and let $\{E_{ij}\}_{1\leq i,j\leq N}$ denote the elementary matrices of $\End V$ with respect to this basis. We will often be working with spaces of the form 
$(\End V)^{\otimes m}\otimes W$, with $m\geq 1$. Given $A=\sum_{i,j=1}^N E_{ij}\otimes a_{ij}\in \End V\otimes W$ and $1\leq k\leq m$, we set
\begin{equation*}
 A_k=\sum_{i,j=1}^N 1^{\ot(k-1)}\ot E_{ij}\ot 1^{\ot(m-k)}\ot a_{ij}\in (\End V)^{\otimes m}\otimes W. 
\end{equation*}
If $W$ is a formal power series ring or if more generally $A=A(u)$ depends on a formal parameter $u$, we will indicate this by 
writing $A_a(u)$ in place of $A_a$ (and rather than $A(u)_a$). 

Similarly, if $\mathscr{A}$ is a unital algebra and $B=\sum_{i=1}^r a_i\ot b_i \in \mathscr{A}\ot \mathscr{A}$  with $1\leq k<l \leq m$ and $m\geq 2$, then we will denote by $B_{kl}$ the element 
\begin{equation*}
 B_{kl}=\sum_{i=1}^r 1^{\ot (k-1)}\ot  a_i \ot  1^{\ot (l-k-1)}\ot b_i \ot 1^{\ot(m-l)}\in \mathscr{A}^{\otimes m}.
\end{equation*}
We instead write $B_{kl}(u)$ if $B=B(u)$ depends on a formal parameter $u$.

In Sections \ref{Sec:RTT} - \ref{Sec:XR} we will consider embeddings of elements $A(u)\in \End V\ot \mathscr{A}[\![u^{-1}]\!]$ into $\End V\ot (\mathscr{A} \ot \mathscr{A})[\![u^{-1}]\!]$. With this in mind, given $A(u)=\sum_{i,j=1}^N E_{ij} \ot a_{ij}(u)\in \End V\ot \mathscr{A} [\![u^{-1}]\!]$ and $1\leq k \leq 2$, we define 
\begin{equation*}
 A_{[k]}(u)=\sum_{i,j=1}^N E_{ij}\ot 1^{\ot(k-1)}\ot a_{ij}(u)\ot 1^{\ot(2-k)}\in \End V\otimes (\mathscr{A} \ot \mathscr{A})[\![u^{-1}]\!].
\end{equation*}
Now suppose that $W_1$ and $W_2$ are arbitrary vector spaces, and let $\phi: W_1\to W_2$ be a linear map. Then, given 
$a(u)=\sum_{r\geq 0} a_r u^{-r}\in W_1[\![u^{-1}]\!]$ and $b(u)=\sum_{r\geq 0} b_r u^{-r}\in W_2[\![u^{-1}]\!]$, we will write $\phi(a(u))=b(u)$  to indicate that 
$\phi(a_r)=b_r$ for all $r\geq 0$. Conversely, we will use expressions of the form $\phi(a(u))=b(u)$ (understood in the same way) to define linear maps, algebra homomorphisms and anti-homomorphisms. 
Similarly, expressions of the form $\phi(A(u))=(\mathrm{id}\ot \phi)A(u)=B(u)$ with $A(u)\in \End V \ot W_1[\![u^{-1}]\!]$ and $B(u)\in \End V \ot W_2[\![u^{-1}]\!]$ will be used to define and interpret transformations $\phi:W_1\to W_2$. 

For any two vector spaces $W_1$ and $W_2$, let 
$\sigma_{W_1,W_2}: W_1\ot W_2 \to W_2\ot W_1$ be the permutation operator defined by $\sigma_{W_1,W_2}(w_1\ot w_2)=w_2\ot w_1$ for all $w_1\in W_1$ and $w_2\in W_2$. In practice we will drop the subscripts and simply write $\sigma=\sigma_{W_1,W_2}$: the underlying vector spaces will always be clear from context. Given $R\in W_1\ot W_2$, we will write $R_{21}$ for the element $\sigma(R)\in W_2\ot W_1$.

Finally, for any unital associative algebra $\mathscr{A}$ we  denote by $\mathrm{Lie}(\mathscr{A})$ the Lie algebra which is equal to $\mathscr{A}$ as a vector space and has Lie bracket equal to the commutator bracket: $[a_1,a_2]=a_1a_2-a_2a_1$ for all $a_1,a_1\in \mathscr{A}$. 

\section{The Yangian of a simple Lie algebra}\label{Sec:YJ}
In this section we recall the definition for the Yangian of $\mfg$ in its $J$-presentation, as well as some of its properties which will play a role in Sections \ref{Sec:YR->YJ} and \ref{Sec:XR}. Aside from Proposition \ref{P:grJ} and a few brief remarks, all of the contents of this section appeared in Drinfeld's seminal paper \cite{Dr1}. 
\begin{definition}[\cite{Dr1}]\label{D:YJ}
 The Yangian $Y(\mfg)$ is the unital associative $\C$-algebra generated by the set of elements $\{X,J(X)\,:\,X\in \mfg\}$ subject to the defining relations
 \begin{align}
  &XY-YX=[X,Y]_\mfg,  \quad J([X,Y])=[J(X),Y], \label{YJ:1}\\
  &J(cX+dY)=cJ(X)+dJ(Y),\label{YJ:2}\\
  &[J(X),[J(Y),Z]]-[X,[J(Y),J(Z)]]=\sum_{\lambda,\mu,\nu\in \Lambda}([X,X_\la],[[Y,X_\mu],[Z,X_\nu]])\{X_\lambda,X_\mu,X_\nu\},\label{YJ:3}\\ 
  &[[J(X),J(Y)],[Z,J(W)]]+[[J(Z),J(W)],[X,J(Y)]]\nonumber\\
  &\qquad =\sum_{\lambda,\mu,\nu\in \Lambda}\left(([X,X_\lambda],[[Y,X_\mu],[[Z,W],X_\nu]])+([Z,X_\lambda],[[W,X_\mu],[[X,Y],X_\nu]])\right)\{X_\lambda,X_\mu,J(X_\nu)\},\label{YJ:4}
 \end{align}
 for all $X,Y,Z,W\in \mfg$ and $c,d\in \C$, where $\{x_1,x_2,x_3\}=\frac{1}{24}\sum_{\pi \in \mathfrak{S}_3}x_{\pi(1)}x_{\pi(2)}x_{\pi(3)}$ for all $x_1,x_2,x_3\in Y(\mfg)$.
\end{definition}
The algebra $Y(\mfg)$ is equipped with an ascending filtration $\mbF^J$ defined by $\deg X=0$ and $\deg J(X)=1$ for all $X\in \mfg$. For each $k\geq 0$, let $\mbF_k^J$ denote the subspace of $Y(\mfg)$ spanned by elements of degree less than or equal to $k$ and denote by $\bar{X}$ and $\ol{J(X)}$ 
the images of $X$ and $J(X)$, respectively, in $\mbF^J_0$ and $\mbF^J_1/\mbF^J_0$, respectively. A proof of the following well-known result, dating back to \cite{Dr1}, was made available recently in \cite{GRW}. 
\begin{proposition}[Proposition 2.2 of \cite{GRW}]\label{P:grJ}
 The associated graded algebra $\gr\,Y(\mfg)$ is isomorphic to $U(\mfg[z])$. An isomorphism $\varphi_{J}:U(\mfg[z])\to \gr\,Y(\mfg)$ is provided by the assignment 
 \begin{equation*}
  X_\lambda z\mapsto \ol{J(X_\lambda)},\quad X_\lambda \mapsto \ol{X}_\lambda \quad \forall \; \lambda \in \Lambda. 
 \end{equation*}
\end{proposition}
We pause momentarily to comment on the relations \eqref{YJ:3} and \eqref{YJ:4}. It was pointed out in \cite{Dr1} that
\begin{enumerate}[label=(\emph{\alph*})]
 \item \label{red:1} when $\mfg\cong \mfsl_2$ the relation \eqref{YJ:3} follows from \eqref{YJ:1} together with \eqref{YJ:2}, and 
 \item \label{red:2} when $\mfg\ncong \mfsl_2$ the relation \eqref{YJ:4} follows from the relations \eqref{YJ:1}--\eqref{YJ:3}.
\end{enumerate}
One way of seeing this is to appeal to the proof of \cite[Theorem 2.6]{GRW}. A careful reading of that proof together with \cite[3(ii)]{GNW} shows that if $\mfg\ncong \mfsl_2$ then the relation \eqref{YJ:4} can be omitted and the relation \eqref{YJ:3} can even be replaced with the relation 
\begin{equation*}
 [J(h),J(h^\prime)]=\frac{1}{4}\sum_{\al,\beta\in \Delta_+}\al(h)\beta(h')[x_\al^-x_\al^+,x_\beta^-x_\beta^+]\quad \forall \; h,h^\prime\in \mfh,
\end{equation*}
where $\mfh$ denotes the Cartan subalgebra of $\mfg$, $\Delta_+$ denotes the set of positive roots of $\mfg$, and for each $\alpha\in \Delta_+$ $x_\al^{\pm}\in \mfg_{\pm \al}$ are such that $(x_\al^+,x_\al^-)=1$. If instead $\mfg\cong \mfsl_2$, then the proof of \cite[Theorem 2.6]{GRW} found in Appendix A of \textit{loc. cit.} shows that the relation \eqref{YJ:3} can be omitted and \eqref{YJ:4} can be replaced with 
\begin{equation*}
 \left[[J(e),J(f)],J(h)\right]=(fJ(e)-J(f)e)h,
\end{equation*}
where $\{e,f,h\}$ is the standard $\mfsl_2$-triple and $(\cdot,\cdot)$ has been normalized to equal the trace form.

By \cite[Theorem 2]{Dr1}, $Y(\mfg)$ is a Hopf algebra with comultiplication $\Delta$, counit $\epsilon$, and antipode $S$ given by 
\begin{equation}\label{YJ:Hopf}
\begin{gathered}
 \Delta(X)=X\otimes 1 +1\otimes X,\; \Delta(J(X))=J(X)\otimes 1+1\otimes J(X)+\tfrac{1}{2}[X\ot 1,\Omega] , \\
  \epsilon(X)=\epsilon(J(X))=0,\\
 S(X)=-X,\quad S(J(x))=-J(X)+\tfrac{1}{4}c_\mfg X,
\end{gathered}
\end{equation}
where $X$ is an arbitrary element of $\mfg$. A proof that $\Delta$ is an algebra homomorphism may be found in \cite{GNW}.

The enveloping algebra $U(\mfg[z])$ has a one parameter family of Hopf algebra automorphisms $\ol \tau_c$, indexed by $c\in \C$, which are determined by 
$\ol \tau_c: Xz^r \to X(z+c)^r$ for all $r\geq 0$ and $X\in \mfg$. The Yangian 
$Y(\mfg)$ also possesses such a family of Hopf algebra automorphisms which can be viewed as quantizations of these shift automorphisms. Explicitly, for each $c\in \C$, there is a Hopf algebra automorphism $\tau_c$ of $Y(\mfg)$ given by the assignment 
\begin{equation}
 X\mapsto X, \quad J(X)\mapsto J(X)+cX \quad \text{ for all }\; X\in \mfg.  \label{J:shift}
\end{equation}
By replacing $c\in \C$ with a formal variable $u$, we obtain an automorphism $\tau_u$ of the polynomial algebra $Y(\mfg)[u]$ or even of the formal power series algebra $Y(\mfg)(\!(u^{-1})\!)$. 
Given complex numbers $c,d\in \C$ and formal variables $u,v$, we will write $\tau_{c,d}=\tau_c\otimes \tau_d$ and $\tau_{u,v}=\tau_u\otimes \tau_v$. We will also denote by $\Delta^{\mathrm{op}}$ the opposite coproduct of $Y(\mfg)$; that is, $\Delta^{\mathrm{op}}=\sigma \circ \Delta$ where $\sigma=\sigma_{Y(\mfg),Y(\mfg)}$. The next corollary follows immediately from the definition of the antipode $S$ given in \eqref{YJ:Hopf}.
\begin{corollary}\label{C:S2}
The square of the antipode $S$ is given by  $S^2=\tau_{-\frac{1}{2}c_\mfg}$.
\end{corollary}
We are now prepared to introduce the universal $R$-matrix of $Y(\mfg)$. 
\begin{theorem}[Theorem 3 of \cite{Dr1}]\label{T:DrThm3}
There is a unique formal series $\mcR(u)=1+\sum_{k=1}^\infty \mcR_k u^{-k}\in (Y(\mfg)\otimes Y(\mfg))[\![u^{-1}]\!]$ satisfying  
\begin{gather}
 (\mathrm{id}\otimes \Delta)\mcR(u)=\mcR_{12}(u)\mcR_{13}(u), \label{co-R}\\
\tau_{0,u}\Delta^{\mathrm{op}}(Y)=\mcR(u)^{-1}(\tau_{0,u}\Delta(Y))\mcR(u) \; \text{ for all }\; Y\in Y(\mfg). \label{quasi} 
\end{gather}
The series $\mcR(u)$ is called the universal $R$-matrix of $Y(\mfg)$ and it also satisfies the quantum Yang-Baxter equation
\begin{equation}
 \mcR_{12}(u-v)\mcR_{13}(u)\mcR_{23}(v)=\mcR_{23}(v)\mcR_{13}(u)\mcR_{12}(u-v), \label{QYBE}
\end{equation}
as well as the relations 
\begin{gather}
 \mcR_{12}(u)\mcR_{21}(-u)=1,\quad \tau_{c,d}\mcR(u)=\mcR(u+d-c), \label{R-inv,shift}\\
 \mcR(u)=1+\Omega u^{-1}+\sum_{\lambda\in \Lambda}(J(X_\lambda)\otimes X_\lambda -X_\lambda\otimes J(X_\lambda))u^{-2}+\tfrac{1}{2}\Omega^2 u^{-2} + O(u^{-3}). \label{R-exp}
\end{gather}
\end{theorem}
Note that \eqref{quasi} should be viewed as a relation in $(Y(\mfg)\otimes Y(\mfg))(\!(u^{-1})\!)$ and the quantum Yang-Baxter equation \eqref{QYBE} can be interpreted as an equality in the space $(Y(\mfg)\otimes Y(\mfg)\otimes Y(\mfg))[\![v^{\pm 1},u^{\pm 1}]\!]$. 

In addition to those properties of $\mcR(u)$ listed in the above theorem, standard arguments show that 
\begin{equation}\label{S(R),eps(R)}
(\mathrm{id}\otimes S)\mcR(u)=\mcR(u)^{-1}\quad \text{ and }\quad  (\mathrm{id}\otimes \epsilon)(\mcR(u))=1.
\end{equation}
We end this section by recalling a result which concerns the uniqueness and rationality of $\mcR(u)$ when evaluated on any two finite-dimensional irreducible representations. 
Let $\rho_V$ and $\rho_W$ be finite-dimensional irreducible representations of $Y(\mfg)$ on the spaces $V$ and $W$, respectively, and set $\mcR_{V,W}(u)=(\rho_V\ot\rho_W)\mcR(-u)$. 
\begin{theorem}[Theorem 4 of \cite{Dr1} and Theorem 3.10 of \cite{GRW}]\label{T:R-rat}
 Up to multiplication by elements of $\C[\![u^{-1}]\!]$, $\mcR_{V,W}(u)$ is the unique solution $R(u)\in \End(V\ot W)[\![u^{-1}]\!]$ of the equation 
\begin{equation}
(\rho_V\ot \rho_W)(\tau_{u,v}\Delta(J(X)))R(u-v)=R(u-v)(\rho_V\ot \rho_W)(\tau_{u,v}\Delta^{\mathrm{op}}(J(X))) \quad \text{ for all }\; X\in \mfg. \label{inter}
\end{equation}
Additionally, there exists a formal series $f(u)\in 1+u^{-1}\C[\![u^{-1}]\!]$ such that $f(u)\mcR_{V,W}(u) \in \End(V\ot W)\ot \C(u)$.
\end{theorem}
The negative sign which appears in the definition of $\mcR_{V,W}(u)$ does not play an important role in this result and has been included so that, up to multiplication by a formal series, $\mcR_{\C^N,\C^N}(u)$ coincides with the $R$-matrix $R(u)$ given by \eqref{R(u)-sl} if $\mfg=\mfsl_N$ and \eqref{R(u)-sosp} if $\mfg=\mfso_N$ or $\mfsp_N$: see \cite[Proposition 3.13]{GRW}.

%
\section{The \texorpdfstring{$r$}{}-matrix presentation of the current algebra \texorpdfstring{$\mfg[z]$}{}}\label{Sec:r-matrix}
An important ingredient needed to prove the isomorphism between the Drinfeld Yangian $Y(\mfg)$ and the $RTT$-Yangian $Y_R(\mfg)$ (see Section \ref{Sec:RTT}) is a presentation of the polynomial current algebra $\mfg[z]$ which is determined by the image of the Casimir element $\Omega$, or more precisely the classical $r$-matrix of $\mfg[z]$, under a fixed representation of the Lie algebra $\mfg$. In this section we obtain such a realization of $\mfg[z]$ (see Corollary \ref{C:g[z]} and Proposition \ref{P:r-matrix}), and also for the current algebra $(\mfg\oplus \mfz_\mcI)[z]$ of a certain trivial central extension $\mfg\oplus \mfz_\mcI$ of $\mfg$ (see Proposition \ref{P:g_I[z]}). The polynomial current algebra $(\mfg\oplus \mfz_\mcI)[z]$ will play an analogous role to $\mfg[z]$ in the study of the extended Yangian $X_\mcI(\mfg)$.

%
%
\subsection{Setup}\label{ssec:setup}

Let $V$ be a finite-dimensional $\mfg$-module with associated homomorphism $\rho:\mfg\to \mfgl(V)$, set $N=\dim V$, and assume that $V$ is not isomorphic to a direct sum of $N$ copies of the trivial representation.  The following setup will be used throughout this paper, with the exception that from Subsection \ref{ssec:g_I} onwards $V$ will be assumed to be a finite-dimensional $Y(\mfg)$-module. 

As in the preliminary section, we fix a basis $\{e_1,\ldots,e_N\}$ of $V$ and let $\{E_{ij}\}_{1\leq i\,j\leq N}$ denote the usual elementary matrices with respect to this basis. 
Let $\Omega_\rho$ denote the image of $\Omega$ under $\rho\ot \rho$: 
\begin{equation*}
 \Omega_\rho=(\rho\otimes \rho)(\Omega).
\end{equation*}
Since $\mfg$ is simple and $\Ker(\rho)\subsetneq\mfg$, the homomorphism $\rho$ is injective, and hence $\{X_\lambda^\bullet=\rho(X_\lambda)\}_{\lambda\in \Lambda}$ is a linearly independent set in $\mfgl(V)$ which spans a Lie subalgebra $\rho(\mfg)$ isomorphic to $\mfg$. The Lie algebra $\mfgl(V)$ acts on itself via the adjoint action, and we may restrict this action to $\mfg\cong \rho(\mfg)$ to obtain a finite-dimensional representation of $\mfg$. We denote the resulting $\mfg$-module by $\ad_\mfg(\mfgl(V))$, and we let $\varrho$ denote the corresponding Lie algebra homomorphism: 
\begin{equation*}
 \varrho:\mfg\to \End(\mfgl(V)). 
\end{equation*}
We use the same notation when $\ad_\mfg(\mfgl(V))$ is viewed as a $U(\mfg)$-module. 

The space $\mathrm{span}\{X_\lambda^\bullet\}_{\lambda\in \Lambda}$ forms a submodule of $\ad_\mfg(\mfgl(V))$ isomorphic to the adjoint representation of $\mfg$. Accordingly, we will write 
\begin{equation*}
 \ad(\mfg)=\mathrm{span}\{X_\lambda^\bullet\}_{\lambda\in \Lambda}
\end{equation*}
when the space on the right-hand side is viewed as a $\mfg$-submodule of $\ad_\mfg(\mfgl(V))$. 

We will extend the basis $\{X_\lambda^\bullet\}_{\lambda\in \Lambda}$ of $\ad(\mfg)$ to a basis $\{X_\lambda^\bullet\}_{\lambda\in\Lambda^\bullet}$ of $\End V$ 
which respects the decomposition of $\ad_\mfg(\mfgl(V))$ into irreducible submodules. Consider the subspace of intertwiners $\mcE_\mfg$ defined by
\begin{equation*}
\mathcal{E}_\mfg=\End_\mfg V. 
\end{equation*}
This is a submodule of $\ad_\mfg(\mfgl(V))$ isomorphic to a direct sum of copies of the trivial representation $\C_\mfg$ of $\mfg$. 
As $\mcE_\mfg$ intersects with $\ad(\mfg)$ trivially, the direct sum $\ad(\mfg)\oplus \mcE_\mfg$ is also a submodule of $\ad_\mfg(\mfgl(V))$. By complete reducibility,
there is a submodule $W^\prime$ of $\ad_\mfg(\mfgl(V))$ complimentary to $\ad(\mfg)\oplus \mcE_\mfg$. Let 
\begin{equation}
W^\prime=W_1\oplus \cdots \oplus W_m \label{Wprime}
\end{equation}
be its decomposition into a direct sum of irreducible $\mfg$-submodules of $\ad_\mfg(\mfgl(V))$, and set $W=\mcE_\mfg\oplus W^\prime$. In summary, we have the $\mfg$-module decomposition
\begin{equation*}
\ad_\mfg(\mfgl(V))=\ad(\mfg)\oplus W= \ad(\mfg)\oplus\mcE_\mfg\oplus W^\prime=\ad(\mfg)\oplus  \mcE_\mfg\oplus W_1\oplus \cdots \oplus W_m. 
\end{equation*}
Note that, by definition, every trivial subrepresentation of $\ad_\mfg(\mfgl(V))$ consists of endomorphisms which commute with $\rho(\mfg)$, and hence is contained in $\mcE_\mfg$. 
In particular, this implies that $W_i\ncong \C_\mfg$ for any $1\leq i\leq m$. Let
$\mcJ$ and $\Lambda_i$, for each $1\leq i\leq m$, be indexing sets such that $\{X_\lambda^\bullet\}_{\lambda\in \mcJ}$ is a basis for $\mcE_\mfg$, and
$\{X_\lambda^\bullet\}_{\lambda\in \Lambda_i}$ is a basis for $W_i$ for each fixed $1\leq i\leq m$. We then set 
\begin{equation*}
 \Lambda^c=\mcJ\sqcup \Lambda_1\sqcup\cdots\sqcup \Lambda_m \quad \text{ and }\quad \Lambda^\bullet=\Lambda\cup \Lambda^c. 
\end{equation*}
Finally, we define a family of complex scalars $\{c_{ij}^\lambda, a_{ij}^\lambda\,:\,\lambda\in \Lambda^\bullet,\, 1\leq i,j\leq N\}$ by 
\begin{equation}
 X_\lambda^\bullet=\sum_{i,j=1}^N c_{ij}^\lambda E_{ij} \quad \text{ and }\quad E_{ij}=\sum_{\lambda \in \Lambda^\bullet} a_{ij}^\lambda X_\lambda^\bullet. \label{ChangeofBasis}
\end{equation}
%
%
\subsection{The Lie algebras \texorpdfstring{$\mfg_\mcJ$}{}, \texorpdfstring{$\mfg_\rho$}{} and their polynomial current algebras}
We now turn to giving a presentation for the enveloping algebra of $\mfg$ which is governed by $\Omega_\rho$. This naturally leads to the desired presentation of the polynomial current algebra $\mfg[z]$: see Corollary \ref{C:g[z]} and Proposition \ref{P:r-matrix}.

\subsubsection{$U_\rho(\mfg)$ and the extended enveloping algebra $U_\mcJ(\mfg)$} We begin by defining an algebra $U_\mcJ(\mfg)$ which can be viewed as an extension of $U(\mfg)$. It will be proven in Proposition \ref{P:K} that this algebra is isomorphic to the enveloping algebra of the Lie algebra $\mfg\oplus \mfz_\mcJ$, where $\mfz_\mcJ$ is a commutative Lie algebra 
of dimension $\dim \End_\mfg V$.

\begin{definition} The extended enveloping algebra $U_\mcJ(\mfg)$ is defined to be the  unital associative $\C$-algebra generated by elements $\{F_{ij}^\mcJ\}_{1\leq i,j\leq N}$ subject to the defining relation 
\begin{equation}
  [F_1^\mcJ,F_2^\mcJ]=[\Omega_\rho,  F_2^\mcJ]  \quad \text{ in } \quad (\End V)^{\ot 2} \ot U_\mcJ(\mfg), \label{wtF-br}
\end{equation}
where $F^\mcJ=\sum_{i,j=1}^N E_{ij}\otimes F_{ij}^\mcJ\in \End V \otimes U_\mcJ(\mfg)$ and $\Omega_\rho$ has been identified with $\Omega_\rho\ot 1$. 
\end{definition}

For each $\lambda \in \Lambda^\bullet$, set $X_\lambda^\mcJ=\sum_{i,j=1}^N a_{ij}^\lambda F_{ij}^\mcJ$ (see \eqref{ChangeofBasis}) so that $ F^\mcJ=\sum_{\lambda \in \Lambda^\bullet} X_\lambda^\bullet \otimes  X_\lambda^\mcJ$, and let $K=\sum_{i,j=1}^N E_{ij}\ot k_{ij}$ be the element of $\End V \ot U_\mcJ(\mfg)$ defined by 
\begin{equation*}
 K=\sum_{i,j=1}^N E_{ij}\otimes k_{ij}=\sum_{\lambda\in \Lambda^c}X_\lambda^\bullet \otimes X_\lambda^\mcJ. 
\end{equation*}
Given an arbitrary vector space $\mathbf{U}$ and $A=\sum_{i,j=1}^N E_{ij}\ot \mbu_{ij} \in \End V\ot \mathbf{U}$, define 
\begin{equation*}
 \omega(A)=\sum_{i,j=1}^N \omega(E_{ij})\ot \mbu_{ij}\in \End V\ot \mathbf{U},\quad \text{ where }\quad \omega(E_{ij})=\varrho(\omega)(E_{ij}),
\end{equation*}
and let $\nabla:\End V\ot \End V\to \End V$ denote the multiplication (or composition) map. 
\begin{lemma}\label{L:K-rels} $K$ satisfies the following properties: 
\begin{enumerate}
 \item\label{K-rel:1} The coefficients $k_{ij}$ of $K$ are central, 
 \item\label{K-rel:2} $[\Omega_\rho,K_2]=0=[\Omega_\rho,K_1]$ and $\omega(K)=0$,
 \item\label{K-rel:3} $X_\lambda^\mcJ=0$ for all $\lambda \in \Lambda^c\setminus \mcJ$. In particular, $K=\sum_{\lambda\in \mcJ}X_\lambda^\bullet\ot X_\lambda^\mcJ$.
\end{enumerate}
\end{lemma}
\begin{proof} 
Consider first \eqref{K-rel:1}. 
After setting $\msF=F^\mcJ-K\in \mathrm{ad}(\mfg)\ot U_\mcJ(\mfg)$, \eqref{wtF-br} gives
\begin{equation}\label{UJ-exp}
 [K_1,F_2^\mcJ]=[\Omega_\rho,F_2^\mcJ]-[\msF_1,F_2^\mcJ]\in \mathrm{ad}(\mfg)\otimes \End V\otimes U_\mcJ(\mfg).
\end{equation}
Since $[K_1,F_2^\mcJ]\in W\otimes \End V\otimes U_\mcJ(\mfg)$, both sides of this equality must vanish, which proves \eqref{K-rel:1}.

\noindent \textit{Proof of \eqref{K-rel:2}}. By Part \eqref{K-rel:1} and \eqref{UJ-exp}, we have 
\begin{equation}\label{UJ-exp2}
 [\Omega_\rho,K_2]=[\msF_2,\Omega_\rho]+[\msF_1,\msF_2]\in \mathrm{ad}(\mfg)\otimes\mathrm{ad}(\mfg) \otimes U_\mcJ(\mfg).
\end{equation}
As $W$ is a submodule of $\mathrm{ad}_\mfg(\mfgl(V))$, $[\Omega_\rho,K_2]\in \mathrm{ad}(\mfg)\otimes W\otimes U_\mcJ(\mfg)$. Therefore $[\Omega_\rho,K_2]=0$, and applying the permutation operator $\sigma\ot 1$ to both sides of this equality gives $[\Omega_\rho,K_1]=0$. These two relations also imply that
\begin{equation*}
0=(\nabla\ot 1)( [\Omega_\rho,K_2-K_1] )=\sum_{\lambda\in\Lambda,\mu\in \Lambda^c} \left[X_\lambda^\bullet,[X_\lambda^\bullet,X_\mu^\bullet]\right]\ot X_\mu^\mcJ=\omega(K).
\end{equation*} 
\noindent \textit{Proof of \eqref{K-rel:3}}. 
On each irreducible component $W_i$ of $W^\prime$ (see \eqref{Wprime}), $\omega$ operates as multiplication by a scalar $c_i$. Hence, from the equality $\omega(K)=0$ and the fact that $\omega(X_\mu^\bullet)=0$ for all $\mu\in \mcJ$, we obtain 
\begin{equation}\label{Cas-act}
 0=\sum_{i=1}^m c_i \left(\sum_{\mu \in \Lambda_i}X_\mu^\bullet \ot  X_\mu^\mcJ\right)
 =\sum_{\mu \in \Lambda^c \setminus \mathcal{J}} X_\mu^\bullet \ot c_\mu X_\mu^\mcJ,
\end{equation}
where in the second equality we have defined $c_\mu$, for each $\mu\in \Lambda^c\setminus \mcJ$, to be equal to $c_i$ for the unique $i\in \{1,\ldots,m\}$ such that $\mu \in \Lambda_i$. It is well known result from the classical theory of simple Lie algebras over $\C$ that the Casimir element operates as a nonzero scalar in every non-trivial finite-dimensional irreducible module. Therefore, $c_i\neq 0$ for all $1\leq i\leq m$ and  \eqref{Cas-act} implies that $ X_\mu^\mcJ=0 $ for all $\mu\in \Lambda^c\setminus \mcJ$. \qedhere
\end{proof}
The next lemma gives two equivalent definitions of $K$ and proves that there is a morphism $U(\mfg)\to U_\mcJ(\mfg)$.  
\begin{lemma}\label{L:F^J,K}
The matrices $F^\mcJ$ and $K$ satisfy the identities 
\begin{gather}
 [\Omega_\rho,F_2^\mcJ]=[F_1^\mcJ,F_2^\mcJ]=[F_1^\mcJ,\Omega_\rho], \label{sigma-sym}\\
 F^\mcJ-2c_\mfg^{-1}(\nabla\ot 1)[F_1^\mcJ,F_2^\mcJ]=K=F^\mcJ-c_\mfg^{-1} \omega(F^\mcJ). \label{wtF-sym}
\end{gather}
Moreover, the assignment $X_\lambda\mapsto -X_\lambda^\mcJ$ for all $\lambda\in \Lambda$ extends to a homomorphism $\iota_\mcJ:U(\mfg)\to U_\mcJ(\mfg)$. 
\end{lemma}
\begin{proof}
 Applying the permutation operator $\sigma\ot 1$ to $[F_1^\mcJ,F_2^\mcJ]=[\Omega_\rho, F_2^\mcJ]$ gives $-[F_1^\mcJ,F_2^\mcJ]=[\Omega_\rho, F_1^\mcJ]$. 
 This implies \eqref{sigma-sym}. 
 
 By Part \eqref{K-rel:2} of Lemma \ref{L:K-rels}, $F^\mcJ-c_\mfg^{-1}\omega(F^\mcJ)=K$. Since $(\nabla\ot 1)[F_1^\mcJ,F_2^\mcJ]=(\nabla\ot 1)[\Omega_\rho, F_2^\mcJ]$, \eqref{sigma-sym} yields
 \begin{equation*}
  (\nabla\ot 1)[F_1^\mcJ,F_2^\mcJ]=\tfrac{1}{2}(\nabla \ot 1)([\Omega_\rho, F_2^\mcJ]-[\Omega_\rho, F_1^\mcJ])=\tfrac{1}{2}\omega(F^\mcJ),
 \end{equation*}
which proves \eqref{wtF-sym}.  

 As for the second part of the lemma, we obtain from Part \eqref{K-rel:2} of Lemma \ref{L:K-rels} and \eqref{UJ-exp2} that $[\msF_1,\msF_2]=[\Omega_\rho,\msF_2]$, where $\msF=F^\mcJ-K$. Expanding in terms of the basis $\{X_\lambda^\bullet\ot X_\mu^\bullet\}_{\lambda,\mu\in\Lambda}$ of $\mathrm{ad}(\mfg)\otimes \mathrm{ad}(\mfg)$ gives 
\begin{equation*}
[X_\lambda^\mcJ,X_\mu^\mcJ]=\sum_{\gamma\in \Lambda}\al_{\lambda\gamma}^\mu X_\gamma^\mcJ=-\sum_{\gamma\in \Lambda}\al_{\lambda\mu}^\gamma X_\gamma^\mcJ\quad \forall \; \lambda,\mu\in \Lambda.
\end{equation*}
 Thus, the assignment $X_\lambda\mapsto -X_\lambda^\mcJ$, for all $ \lambda\in \Lambda$, extends to a homomorphism $\iota_\mcJ:U(\mfg)\to U_\mcJ(\mfg)$. \qedhere

\qedhere
\end{proof}
We now simultaneously define the algebra $U_\rho(\mfg)$ as a quotient of $U_\mcJ(\mfg)$ and prove that it is isomorphic to the enveloping algebra $U(\mfg)$.
\begin{proposition}\label{P:new-g}
Let $U_\rho(\mfg)$ be the quotient of $U_\mcJ(\mfg)$ by the two-sided ideal generated by the coefficients of the central matrix $K$. Equivalently, $U_\rho(\mfg)$ is the unital associative $\C$-algebra generated by elements 
 $\{F_{ij}\}_{1\leq i,j\leq N}$ subject to the defining relations 
 \begin{gather}
  [F_1,F_2]=[\Omega_\rho,F_2], \label{F-br}\\
  F=c_\mfg^{-1}\omega(F),  \label{F-sym}
 \end{gather}
 where $F=\sum_{i,j=1}^N E_{ij}\otimes F_{ij}\in \End V\otimes U_\rho(\mfg)$.

 Then $U_\rho(\mfg)$ is isomorphic to the enveloping algebra $U(\mfg)$. An isomorphism $\phi_\rho$ is given by
\begin{equation}
 \phi_\rho:U_\rho(\mfg)\to U(\mfg), \quad F\mapsto -(\rho\otimes 1)\Omega. \label{g-map}
\end{equation}
\end{proposition}
\begin{proof}
Set $\mcF=\sum_{i,j=1}^NE_{ij}\otimes \mcF_{ij}=-(\rho\otimes 1)\Omega$. By \eqref{ChangeofBasis}, the element $\mcF_{ij}=\phi_\rho(F_{ij})$ is equal to $-\sum_{\lambda\in \Lambda}c_{ij}^\lambda X_\lambda$. 

\noindent \textit{Step 1}: $\phi_\rho$ is a homomorphism of algebras. 

Recall that $[\Omega,\Delta(X)]=0$ for all $X\in \mfg$. This implies that, in $\mfg\otimes \mfg\otimes \mfg$, we have the identity $ [\Omega_{13},\Omega_{23}]=-[\Omega_{12},\Omega_{23}]$. 
Applying the homomorphism $\rho\otimes \rho\otimes 1$ to both sides of this identity, we obtain the relation 
\begin{equation*}
 [\mcF_1,\mcF_2]=[\Omega_\rho,\mcF_2] \quad \text{ in }\quad  (\End V)^{\ot 2}\otimes \mfg.
\end{equation*}
Hence, the assignment \eqref{g-map} preserves the relation \eqref{F-br}. 

Since we also have $\mcF=-\sum_{\lambda\in\Lambda} X_\lambda^\bullet \ot X_\lambda\in \ad(\mfg)\ot \mfg$, and $\omega$ acts on $\ad(\mfg)$ as multiplication by the scalar $c_\mfg$, the relation $\mcF=c_\mfg^{-1}\omega(\mcF)$ is satisfied, and thus $\phi_\rho$ is a homomorphism.  

\noindent\textit{Step 2:} $\phi_\rho$ is an isomorphism. 

For each $\lambda\in \Lambda^\bullet$, define $X_\lambda^\rho$ to be the image of $X_\lambda^\mcJ$ under the natural quotient map $q:U_\mcJ(\mfg)\onto U_\rho(\mfg)$. 
Since $q(K)=0$, $X_\lambda^\rho=0$ for all $\lambda\in \Lambda^c$. Let $\psi=q\circ \iota_\mcJ: U(\mfg)\to U_\rho(\mfg)$, where $\iota_\mcJ:U(\mfg)\to U_\mcJ(\mfg)$ is the morphism from Lemma \ref{L:F^J,K}. Then $\phi_\rho\circ \psi=\mathrm{id}_{U(\mathfrak{g})}$, and to see that $\psi\circ \phi_\rho=\mathrm{id}_{U_\rho(\mathfrak{g})}$ it suffices to note that $\{X_\lambda^\rho\}_{\lambda\in \Lambda}$ generates $U_\rho(\mfg)$, which is immediate since it is the image of the generating set $\{X_\lambda^\mcJ\}_{\lambda \in \Lambda^\bullet}$ of $U_\mcJ(\mfg)$. This proves that $\phi_\rho$ is an isomorphism with inverse $\psi$. \qedhere
\end{proof}
\begin{remark}\label{R:F-sym}
After expanding $\Omega_\rho=\sum_{i,j,k,l=1}^N \mathsf{c}_{ij}^{kl} E_{ij}\otimes E_{kl}$, we may rewrite the relations \eqref{F-br} and \eqref{F-sym} of $U_\rho(\mfg)$ more explicitly in terms of the generators $F_{ij}$. They are
 \begin{gather}
  [F_{ij},F_{kl}]=\sum_{a=1}^N\left(\mathsf{c}_{ij}^{ka}F_{al}-\mathsf{c}_{ij}^{al}F_{ka}\right) \quad \text{ and }\quad  F_{ij}=2c_\mathfrak{g}^{-1}\sum_{a=1}^N[F_{ia},F_{aj}] \quad \forall \; 1\leq i,j,k,l\leq N, \label{F-sym'}
 \end{gather}
 where to obtain the second relation we have employed that, by \eqref{wtF-sym}, $c_\mfg^{-1}\omega(F)=2c_\mfg^{-1}(\nabla\ot 1)[F_1,F_2]$.   
 \end{remark}
Let $\mfg_\rho$ be the Lie subalgebra of $\mathrm{Lie}(U_\rho(\mfg))$  generated by $\{X_\lambda^\rho\}_{\lambda\in \Lambda}$, or equivalently by $\{F_{ij}\}_{1\leq i,j\leq N}$. 
Then Proposition \ref{P:new-g} implies that \eqref{F-br} and \eqref{F-sym} are defining relations for $\mfg_\rho$ and that $\phi_\rho|_{\mfg_\rho}$ is an isomorphism of Lie algebras $\mfg_\rho\iso \mfg$. Consequently $U(\mfg_\rho)\cong U_\rho(\mfg)$, and we will henceforth exploit this fact and denote $U_\rho(\mfg)$ instead by $U(\mfg_\rho)$. 

We now return to the study of the algebra $U_\mcJ(\mfg)$. 
Define $\mathfrak{z}_\mcJ$ to be the commutative Lie algebra with basis $\{\mcK_\lambda^\mcJ\}_{\lambda \in \mathcal{J}}$, and identify the enveloping algebra $U(\mfz_\mcJ)$ with $\C[\mcK_\lambda^\mcJ\,:\,\lambda \in \mathcal{J}]$. We will denote the matrix $\sum_{\lambda\in \mathcal{J}}X_\lambda^\bullet\ot \mcK_\lambda^\mcJ\in \End V \ot \mathfrak{z}_\mcJ$ by $\mathscr{K}^\mcJ$. 
 \begin{proposition}\label{P:K}
The assignment  $F^\mcJ\mapsto F+\mathscr{K}^\mcJ$ extends to an isomorphism of algebras 
  \begin{equation}
  \phi_{\mcJ}:U_\mcJ(\mfg)\iso  \C[\mcK_\lambda^\mcJ\,:\,\lambda \in \mathcal{J}]\ot U(\mfg_\rho).\label{phi_ex}
  \end{equation}
  \end{proposition}
\begin{proof}
Since $\mathscr{K}^\mcJ\in \mcE_\mfg\ot \mfz_\mcJ$, we have $[\Omega_\rho,\mathscr{K}_2^\mcJ]=0$.
As the coefficients of $\mathscr{K}^\mcJ$ are also central and $F$ satisfies \eqref{F-br}, $F+\mathscr{K}^\mcJ$ satisfies the defining relation \eqref{wtF-br} of $U_\mcJ(\mfg)$. Thus the assignment 
$F^\mcJ\mapsto F+\mathscr{K}^\mcJ$ extends to a homomorphism $\phi_\mcJ:U_\mcJ(\mfg)\to \C[\mcK_\lambda^\mcJ\,:\,\lambda \in \mathcal{J}]\ot U(\mfg_\rho)$. 

Since the coefficients of $K$ are central, we deduce that there is an algebra homomorphism $\psi_{\mfz_\mcJ}:\C[\mcK_\lambda^\mcJ\,:\,\lambda \in \mathcal{J}]\to U_\mcJ(\mfg)$ given by $\mathscr{K}^\mcJ\mapsto K$. Let $\iota=\iota_\mcJ\circ \phi_\rho: U(\mfg_\rho)\to U_\mcJ(\mfg)$.  Since $[\iota(X),\psi_{\mfz_\mcJ}(Y)]=0$ for all $X\in U(\mfg_\rho)$ and $Y\in \C[\mcK_\lambda^\mcJ\,:\,\lambda \in \mathcal{J}]$, there is a unique homomorphism 
\begin{equation*}
\psi_\mcJ=\psi_{\mfz_\mcJ}\ot \iota:\C[\mcK_\lambda^\mcJ\,:\,\lambda \in \mathcal{J}]\ot U(\mfg_\rho)\to U_\mcJ(\mfg)
\end{equation*}
satisfying $\psi_\mcJ(\mathscr{K}^\mcJ)=K$ and $\psi_\mcJ(F)=\msF$, where we recall that $\msF=F^\mcJ-K$. As $\phi_\mcJ$ is completely determined by $\phi_{\mcJ}(K)=\mathscr{K}^\mcJ$ and $\phi_{\mcJ}(\msF)=F$, it follows immediately that $\psi_\mcJ=\phi_{\mcJ}^{-1}$. \qedhere
\end{proof}
Define $\mfg_\mcJ$ to be the Lie subalgebra of $\mathrm{Lie}(U_\mcJ(\mfg))$ generated by the set of elements $\{F_{ij}^\mcJ\}_{1\leq i,j\leq N}$. Then the restriction $\phi_{\mcJ}|_{\mfg_\mcJ}$ (see \eqref{phi_ex}) and its composition with $\mathrm{id}\ot \phi_\rho$ (see \eqref{g-map}) produce isomorphisms 
\begin{equation}
 \mfg_\mcJ \iso \mfg_\rho\oplus \mfz_\mcJ\iso \mfg\oplus \mfz_\mcJ, \label{gJ=g+zJ}
\end{equation}
and we have $U(\mfg_\mcJ)\cong U_\mcJ(\mfg)$. Accordingly, we will henceforth denote $U_\mcJ(\mfg)$ by $U(\mfg_\mcJ)$.
%
\subsubsection{The polynomial current algebras $\mfg_\rho[z]$ and $\mfg_\mcJ[z]$}
As a consequence of Proposition \ref{P:new-g} and the comments following Remark \ref{R:F-sym}, the current algebras $\mfg_\rho[z]$ and $\mfg[z]$ are isomorphic. Similarly, Proposition \ref{P:K} and the isomorphism \eqref{gJ=g+zJ} imply that $\mfg_\mcJ[z]\cong(\mfg\oplus \mfz_\mcJ)[z]$. The former identification leads to the so called $r$-matrix realization of $\mfg[z]$, as we will illustrate in this subsection. 
\begin{corollary}\label{C:g[z]}
An isomorphism $\phi_\rho^z:\mfg_\rho[z]\to \mfg[z]$ is provided by the assignment
\begin{equation*}
 \phi_\rho^z:F^{(r)}\mapsto -(\rho\otimes 1)(\Omega z^r) \quad \forall\; r\geq 0,
\end{equation*}
where, for each $r\geq 0$, $F^{(r)}= \sum_{i,j=1}^N E_{ij}\otimes F_{ij} z^r\in \End V \ot \mfg_\rho[z]$ and $\Omega z^r= \sum_{\lambda \in \Lambda} X_\lambda \otimes X_\lambda z^r\in \mfg\otimes \mfg[z]$.

In particular, $U(\mfg[z])$ is isomorphic to the unital associative $\C$-algebra generated by the family of elements 
$\{F_{ij}^{(r)}=F_{ij}z^r\,:\, 1\leq i,j\leq N,\, r\in \Z_{\geq 0}\}$ subject to the defining relations 
\begin{gather}
 [F_1^{(r)},F_2^{(s)}]=[\Omega_\rho,F_2^{(r+s)}] \quad \forall\; r,s\geq 0, \label{g[z]-R} \\
    F^{(r)}=c_\mfg^{-1}\omega(F^{(r)})  \quad \forall\; r\geq 0. \label{g[z]-sym}
\end{gather}
\end{corollary}
\begin{proof}
 The corollary follows from Proposition \ref{P:new-g}, the three sentences following Remark \ref{R:F-sym}, and the definition of the current algebra $\mfg[z]$ (see \eqref{g[z]}).
\end{proof}
\begin{remark} The relations \eqref{g[z]-R} and \eqref{g[z]-sym} are, of course, just the defining relations of $U(\mfg_\rho[z])$. Omitting the relation \eqref{g[z]-sym}
gives the definition of $U(\mfg_\mcJ[z])$. 
\end{remark}
Introduce the generating matrix 
\begin{equation*}\label{F(u)}
 F(u)=\sum_{i,j=1}^N E_{ij}\ot F_{ij}(u)\in \End V \ot (\mfg_\rho[z])[\![u^{-1}]\!],\quad \text{ where }\quad F_{ij}(u)=\sum_{r\geq 0}F_{ij}^{(r)}u^{-r-1}\in (\mfg_\rho[z])[\![u^{-1}]\!].
\end{equation*}
Using this notation, we can express the defining relations of $\mfg[z]$ (or more precisely those of $\mfg_\rho[z]$) using the classical $r$-matrix 
$\frac{\Omega}{u-v}$ associated to its standard Lie bialgebra structure.
\begin{proposition}\label{P:r-matrix}
The defining relations \eqref{g[z]-R} and \eqref{g[z]-sym} are equivalent to the relations
\begin{align}
[F_1(u),F_2(v)]&=\left[\frac{\Omega_\rho}{u-v}, F_1(u)+F_2(v)\right], \label{g[z]-R:u}\\
F(u)&=c_\mfg^{-1}\omega(F(u)). \label{g[z]-Cas}
\end{align}
The relation \eqref{g[z]-R:u} independently serves as the defining relation of $U(\mfg_\mcJ[z])$. 
\end{proposition}

\begin{proof}
It is clear that the relation \eqref{g[z]-Cas} is equivalent to \eqref{g[z]-sym}. To prove the equivalence of \eqref{g[z]-R:u} with \eqref{g[z]-R}, we will expand 
\begin{equation}
(u-v)^{-1}=\sum_{p\geq 0} v^p u^{-p-1} \in (\C[v])[\![u^{-1}]\!],  \label{exp:u-v}
\end{equation}
view \eqref{g[z]-R:u} as an equality in the space  $(\End V)^{\otimes 2}\ot U(\mfg_\rho[z])[\![v^{\pm 1},u^{-1}]\!]$, and compare the coefficient of $v^s u^{-r}$ on each side for $s\in \Z$ and $r\in \Z_{\geq 0}$. Note that \eqref{exp:u-v} is not the unique expansion of $(u-v)^{-1}$ in $\C[\![v^{\pm 1},u^{\pm 1}]\!]$, and thus there are other equivalent ways of viewing \eqref{g[z]-R:u}: see Remark \ref{R:expand}.

 Expanding  \eqref{g[z]-R:u} using \eqref{exp:u-v}, we obtain
 \begin{equation}\label{g[z]-R:exp}
  \sum_{r,s\geq 0} [F_1^{(r)},F_2^{(s)}]v^{-s-1}u^{-r-1}=\sum_{p,a,b\geq 0}\left([\Omega_\rho, F_1^{(a)}]v^p u^{-p-a-2}+[\Omega_\rho, F_2^{(b)}]v^{p-b-1}u^{-p-1}\right)
 \end{equation}
 Comparing the coefficient of $u^{-r-1}v^{-s-1}$ in both sides, for $r,s\in \Z_{\geq 0}$, we obtain \eqref{g[z]-sym}:
 \begin{equation*}
  [F_1^{(r)},F_2^{(s)}]=[\Omega_\rho,F_2^{(r+s)}] \quad \forall \; r,s\geq 0. 
 \end{equation*}
 We must also compare the coefficient of $v^s u^{-r}$ (for $r,s\in \Z_{\geq 0}$) in both sides of \eqref{g[z]-R:exp} to guarantee that this relation does not imply any additional relations which are not satisfied in $U(\mfg_\rho[z])$. If $0\leq r< 2$ or $s>r-2$, the coefficient of $u^{-r}v^s$ on both sides of \eqref{g[z]-R:u} is zero. Otherwise, we obtain 
 \begin{equation*}
  0=[\Omega_\rho, F_1^{(r-s-2)}]+[\Omega_\rho, F_2^{(r-s-2)}],
 \end{equation*}
which is also a consequence of \eqref{g[z]-R}: this can be deduced from \eqref{g[z]-R} in the same way that the relation \eqref{sigma-sym} of Lemma \ref{L:F^J,K}
was deduced from \eqref{wtF-br}. \qedhere
\end{proof}
\begin{remark}\label{R:expand}
In the proof of Proposition \ref{P:r-matrix} we have expanded the rational expression $(u-v)^{-1}$ as an element of $(\C[v])[\![u^{-1}]\!]$ and then interpreted 
\eqref{g[z]-R:u} as an equality in $(\End V)^{\otimes 2}\ot U(\mfg_\rho[z])[\![v^{\pm 1},u^{-1}]\!]$. As mentioned in the proof of the proposition, this is not the only way we could have proceeded. Working in a more general framework, \eqref{g[z]-R:u} should be viewed as an equality in $(\End V)^{\otimes 2}\ot U(\mfg_\rho[z])[\![v^{\pm 1},u^{\pm 1}]\!]$. In particular, 
$(u-v)^{-1}$ can be expanded as the formal series $-\sum_{p\geq 0}u^p v^{-p-1}$ in $(\C[u])[\![v^{-1}]\!]$, leading to an equivalent set of defining relations. 

An alternative expansion involves multiplying both sides of \eqref{g[z]-R:u} by the polynomial $u-v$ and then expanding both sides as elements of $(\End V)^{\otimes 2}\ot U(\mfg_\rho[z])[\![u^{-1},v^{-1}]\!]$: see for instance Subsection 1.1 of \cite{Mobook}.  
\end{remark}
%
%
%
\subsection{The extended Lie algebra \texorpdfstring{$\mfg_\mcI$}{} and its polynomial current algebra}\label{ssec:g_I}
In this subsection we consider an algebra $U_\mcI(\mfg)$ which is constructed from a fixed finite-dimensional $Y(\mfg)$-module. Like $U(\mfg_\mcJ)=U_\mcJ(\mfg)$, it is an extension of 
the enveloping algebra $U(\mfg_\rho)$, but the role played by $\End_\mfg V$ is instead played by $\End_{Y(\mfg)} V$. Consequently, $U_\mcI(\mfg)$ encodes certain information about the underlying $Y(\mfg)$-module structure which $U(\mfg_\mcJ)$ does not. 

Henceforth, we assume that $V$ is a finite-dimensional $Y(\mfg)$-module with corresponding homomorphism $\rho: Y(\mfg) \to \End V$. We also assume that $V$ contains a non-trivial  irreducible submodule. This hypothesis guarantees that $V$ has at least one non-trivial irreducible component when viewed as a $\mfg$-module (via restriction), and hence that we are in the situation of Subsection \ref{ssec:setup}. In particular, all the definitions and results of the previous subsection apply. 

Going forward, we will need to specialize our basis $\{X_\lambda^\bullet\}_{\lambda \in \mcJ}$ of $\mcE_\mfg=\End_\mfg(V)$. Let $\mcE\subset \mcE_\mfg$ 
denote the subspace of $Y(\mfg)$-module endomorphisms, and let $\mcE_c$ be a subspace of $\mcE_\mfg$ complimentary to $\mcE$:
\begin{equation*}
 \mcE=\End_{Y(\mfg)}V\subset \mcE_\mfg, \quad \mcE_\mfg=\mcE\oplus \mcE_c. 
\end{equation*}
We may then partition $\mcJ=\mcI\sqcup \mcI_c$ and choose $\{X_\lambda^\bullet\}_{\lambda \in \mcJ}$ in such a way that $\{X_\lambda^\bullet\}_{\lambda\in \mcI}$ is a basis of $\mcE$ and $\{X_\lambda^\bullet\}_{\lambda\in \mcI_c}$ is a basis of $\mcE_c$. 

%
\subsubsection{The extended enveloping algebra $U_\mcI(\mfg)$}
Following our convention of labeling $X^\bullet=\rho (X)$ for each $X\in \mfg$, we will write $J(X^\bullet)$ for the image of $J(X)$ in $\End V $ under $\rho $. In addition, we define a module homomorphism $J:\ad(\mfg)\to \ad_\mfg(\mfgl(V))$ by $X^\bullet \mapsto J(X^\bullet)$ for all $X\in \mfg$.
\begin{definition}
 Define $U_\mcI(\mfg)$ to be the quotient of $U(\mfg_\mcJ)$ by the two-sided ideal generated by the relation $[K_2,(1\ot J)(\Omega_\rho)]=[K_1,(J\ot 1)(\Omega_\rho)]$. That is, 
$U_\mcI(\mfg)$ is the associative unital $\C$-algebra generated by elements $\{F_{ij}^\mcI\}_{ 1\leq i,j\leq N}$ subject to the defining relations 
\begin{gather}
  [F_1^\mcI,F_2^\mcI]=[\Omega_\rho, F_2^\mcI], \label{FI-br}\\
  [K_2^\mcI,(1\ot J)(\Omega_\rho)]=[K_1^\mcI,(J\ot 1)(\Omega_\rho)] \label{KJ=JK}
\end{gather}
where $F^\mcI=\sum_{i,j=1}^N E_{ij}\otimes F_{ij}^\mcI\in \End V \otimes  U_\mcI(\mfg)$ and $K^\mcI=F^\mcI-c_\mfg^{-1}\omega(F^\mcI)$.  
\end{definition}

We now work towards establishing an analogue of Proposition \ref{P:K}. For each $\lambda \in \Lambda^\bullet$, let $X_\lambda^\mcI$ denote the image 
of $X_\lambda^\mcJ$ in $U_\mcI(\mfg)$. Explicitly, $X_\lambda^\mcI=\sum_{i,j=1}^N a_{ij}^\lambda F_{ij}^\mcI$ (see \eqref{ChangeofBasis}) and we have $F^\mcI=\sum_{\lambda \in \Lambda^\bullet} X_\lambda^\bullet \otimes X_\lambda^\mcI$. In fact, by Part \eqref{K-rel:3} of  Lemma \ref{L:K-rels}, we have
\begin{equation*}
 K^\mcI=\sum_{\lambda\in \mcJ} X_\lambda^\bullet \ot X_\lambda^\mcI \quad \text{ and } \quad F^\mcI=\sum_{\lambda\in \Lambda\cup \mcJ} X_\lambda^\bullet \ot X_\lambda^\mcI.
\end{equation*}
As a first step, we construct for each $x\in \mcE_\mfg$ a $\mfg$-module $W(x)$ which is either zero or isomorphic to $\ad(\mfg)$, but which cannot have a nonzero intersection with $\ad(\mfg)$. Fix $x\in \mcE_\mfg$ and let 
\begin{equation*}
 W(x)=\mathrm{span}\{[x,J(X_\lambda^\bullet)]\}_{\lambda \in \Lambda}.
\end{equation*}
Note that $W(x)$ is a submodule of the $\mfg$-module $\ad_\mfg(\mfgl(V))$, and that there is a module homomorphism
\begin{equation*}
 \varphi_x:\ad(\mfg) \to W(x), \; X_\lambda^\bullet \mapsto [x,J(X_\lambda^\bullet)] \quad \forall\; \lambda\in \Lambda.
\end{equation*}
This homomorphism is surjective and, by Schur's lemma, it is either an isomorphism or the zero morphism. We also have  $\mcE=\{x\in \mcE_\mfg\,:\,\varphi_x=0\}=\{x\in \mcE_\mfg\,:\,W(x)=0\}$. 
\begin{lemma}\label{L:Wx=/=ad}
There does not exist $x\in \mcE_\mfg$ such that $W(x)\cap \ad(\mfg)\neq \{0\}$. 
\end{lemma}

\begin{proof}
Suppose that $x\in \mcE_\mfg$ satisfies $W(x)\cap\ad(\mfg)\neq \{0\}$. Then $W(x)$ is irreducible and, since the same is true for $\ad(\mfg)$, we have $W(x)=\ad(\mfg)$. In particular, $\varphi_x$ must be an isomorphism, and by Schur's lemma, every module homomorphism $W(x)\to \ad(\mfg)$ is a scalar multiple of $\varphi_x^{-1}:[x,J(X_\lambda^\bullet)]\mapsto X_\lambda^\bullet$. As the identity map provides such a homomorphism, there exists $c\in \C^\times$ such that $[x,J(X_\lambda^\bullet)]=cX_\lambda^\bullet$ for all 
$\lambda\in\Lambda$. After re-normalizing $x$ if necessary, we can assume that $c=1$. Consider the linear map 
\begin{equation*}
 \ad_x:\End V\to \End V,\quad X\mapsto [x,X] \quad \forall\; X\in \End V. 
\end{equation*}
Since $\ad_x(J(X_\lambda^\bullet))=X_\lambda^\bullet$ and $\ad_x(X_\lambda^\bullet)=0$ for all $\lambda\in \Lambda$, we deduce from the fact that $\ad_x$ is a derivation that it restricts to a linear map 
\begin{equation*}
\ad_{\rho,x}:\rho(Y(\mfg))\to \rho(Y(\mfg)). 
\end{equation*}
Given a monomial $X$ in the variables $\{J(X_\lambda^\bullet),X_\gamma^\bullet\}_{\lambda,\gamma\in \Lambda}$, we denote by $\ell(X)$ the degree of this monomial with respect to 
the assignment $\deg X_\gamma^\bullet=0$ and $\deg J(X_\lambda^\bullet)=1$. 
For each $k\geq 0$, let $\mbH_k$ denote the subspace of $\rho(Y(\mfg))$ which is spanned by monomials $X$ such that $\ell(X)\leq k$, i.e. $\mbH_k=\rho (\mbF_k^J)$, where $\mbF^J=\{\mbF_k^J\}_{k\geq 0}$ is the filtration defined below Definition \ref{D:YJ}. 
We then have $\ad_{\rho,x}(\mbH_0)=0$ and $\ad_{\rho,x}(\mbH_k)\subset \mbH_{k-1}$ for all $k\geq 1$. This follows from the facts that $\ad_{\rho,x}(J(X_\lambda^\bullet))=X_\lambda^\bullet$ for all $\lambda\in \Lambda$, 
$\ad_{\rho,x}(X_\gamma^\bullet)=0$ for all $\gamma \in \Lambda$, and that $\ad_{\rho,x}$ is a derivation. We will break the remainder of our proof into two steps: 

\noindent \textit{Step 1:} There exists $k\geq 1$ such that $\ad_{\rho,x}^k=0$. 

Note that $\mbH_{k-1}\subset \Ker(\ad_{\rho,x}^k)$ for all $k\geq 1$. Indeed, since $\ad_{\rho,x}(\mbH_{k-1})\subset \mbH_{k-2}$ for all $k\geq 1$ (here $\mbH_a=\{0\}$ for all $a<0$), 
we obtain inductively that $\ad_{\rho,x}^k(\mbH_{k-1})\subset \mbH_{-1}=\{0\}$. Since $\rho(Y(\mfg))\subset\End V$ is finite-dimensional, it has a finite basis $\{B_1,\ldots,B_{\dim \rho(Y(\mfg))}\}$ consisting of monomials $B_i$ in  the variables $\{J(X_\lambda^\bullet),X_\gamma^\bullet\}_{\lambda,\gamma\in \Lambda}$. Let $\ell$ denote the finite integer $\max\{\ell(B_i)\,:\,1\leq i\leq \dim \rho(Y(\mfg))\}$. Then each $B_i$ belongs to $\mbH_\ell$ and hence so does all of $\rho(Y(\mfg))$. Since $\ad_{\rho,x}^{\ell+1}(\mbH_{\ell})=0$, $\ad_{\rho,x}^{\ell+1}$ is identically zero.

\noindent \textit{Step 2:} The image of $\ad_{\rho,x}^k$ contains $\rho(\mfg)\cong \mfg$ for every $k\geq 1$.

For each $k\geq 1$ and $k$-tuple $\alpha_1,\ldots,\alpha_{k}\in \Lambda$, set 
\begin{gather*}
 A_{\al_1,\ldots,\al_k}=[J(X_{\al_1}^\bullet),[J(X_{\al_2}^\bullet),\cdots,[J(X_{\al_{k-1}}^\bullet),J(X_{\al_k}^\bullet)]\cdots]],\\
  Y_{\al_1,\ldots,\al_k}=[X_{\al_1}^\bullet,[X_{\al_2}^\bullet,\cdots,[X_{\al_{k-1}}^\bullet,X_{\al_{k}}^\bullet]\cdots]].
\end{gather*}
If $k=1$, then it is understood that $A_{\al}=J(X_\al^\bullet)$ and $Y_{\al}=X_{\al}^\bullet$. 

\noindent\textit{Claim: } $\ad_{\rho,x}^{k}(A_{\al_1,\ldots,\al_k})= k ! Y_{\al_1,\ldots,\al_k}$ for all $k\geq 1$. 

We will prove the claim by induction on $k$. If $k=1$ then it is just the statement that $\ad_{\rho,x}(J(X_\al^\bullet))=X_\al^\bullet$. Suppose inductively that 
the claim holds whenever $k=l$, and consider $\ad_{\rho,x}^{l+1}(A_{\al_1,\ldots,\al_{l+1}})$. We have 
\begin{equation*}
 \ad_{\rho,x}^{l+1}(A_{\al_1,\ldots,\al_{l+1}})=\sum_{j=0}^{l+1}\binom{l+1}{j}\left[\ad_{\rho,x}^j(J(X_{\al_1}^\bullet)),\ad_{\rho,x}^{l+1-j}(A_{\al_2,\ldots,\al_{l+1}})\right].
\end{equation*}
Since $\ad_{\rho,x}^2(J(X_{\al_1}^\bullet))=0$ and $\ad_{\rho,x}^{l+1}(A_{\al_2,\ldots,\al_{l+1}})=0$ (since $A_{\al_2,\ldots,\al_{l+1}}\in \mbH_{l}$), the only term of the sum on the right-hand side which does not necessarily vanish corresponds to $j=1$. As $\ad_{\rho,x}(J(X_{\al_1}^\bullet))=X_{\al_1}^\bullet$ and, by induction, $\ad_{\rho,x}^{l}(A_{\al_2,\ldots,\al_{l+1}})={l!}Y_{\al_2,\ldots,\al_{l+1}}$, we have
\begin{equation*}
 \ad_{\rho,x}^{l+1}(A_{\al_1,\ldots,\al_{l+1}})=(l+1)l!\left[X_{\al_1}^\bullet,Y_{\al_2,\ldots,\al_{l+1}}\right]=(l+1)!Y_{\al_1,\ldots,\al_{l+1}}.
\end{equation*}
This completes the proof of the claim. 

To complete the proof of Step 2, it remains to note that, since $\rho(\mfg)$ is a simple Lie algebra, it is perfect and thus spanned by the collection of elements $\{Y_{\al_1,\ldots,\al_k}\}_{\al_i\in \Lambda}$ for any fixed $k\geq 1$. 

We can now finish the proof of the lemma. By Step 1, there exists $k \geq 1$ such that $\ad_{\rho,x}^k=0$. By Step 2, $\rho(\mfg) \subset \ad_{\rho,x}^k(\rho(Y(\mfg)))=\{0\}$, which is a contradiction. Therefore there cannot exist $x\in \mcE_\mfg$ such that $W(x)\cap \ad(\mfg) \neq \{0\}$. 
\end{proof}

This leads us to the following analogue of Part \eqref{K-rel:3} of Lemma \ref{L:K-rels}.  
\begin{lemma}\label{L:KI}
 We have $X_\mu^\mcI=0$ for all $\mu \in \mcI_c$. In particular, $K^\mcI=\sum_{\lambda\in \mcI} X_\lambda^\bullet \ot X_\lambda^\mcI$.
\end{lemma}
\begin{proof} 
  Since $[X_\mu^\bullet,J(X_\lambda^\bullet)]=0$ for all $\mu\in \mcI$ and $\lambda \in \Lambda$, \eqref{KJ=JK} is equivalent to 
\begin{equation}\label{J-comm}
\sum_{\lambda\in \Lambda,\mu\in \mcI_c} X_\lambda^\bullet \ot [X_\mu^\bullet,J(X_\lambda^\bullet)]\ot X_\mu^\mcI=
\sum_{\lambda\in \Lambda,\mu\in \mcI_c} [X_\mu^\bullet,J(X_\lambda^\bullet)] \ot X_\lambda^\bullet \ot X_\mu^\mcI.
\end{equation}

Let's first show that for any fixed $\lambda\in \Lambda$, $\{[X_\mu^\bullet,J(X_\lambda^\bullet)]\}_{\mu \in \mcI_c}$ is a linearly independent set. Suppose that 
\begin{equation*}
 \sum_{\mu\in \mcI_c}a_\mu [X_\mu^\bullet,J(X_\lambda^\bullet)]=0 \;\text{ for some }\; \{a_\mu\}_{\mu\in \mcI_c}\subset \C. 
\end{equation*}
Then $x=\sum_{\mu\in \mcI_c}a_\mu X_\mu^\bullet$ must belong to $\mcE$, because $\varphi_x$ cannot be an isomorphism as its kernel contains $X_\lambda^\bullet$. Since $x$ also belongs to $\mcE_c$, we must have $x=0$. The assertion then follows from the linear independence of the set $\{X_\mu^\bullet\}_{\mu\in \mcI_c}$. 

Next, we deduce that, for any fixed $\lambda\in \Lambda$, the set $\{X_\gamma^\bullet, [X_\mu^\bullet,J(X_\lambda^\bullet)]\}_{\gamma\in \Lambda,\mu \in \mcI_c}$ must also be linearly independent. Indeed, if $0\neq \sum_{\mu\in \mcI_c}a_\mu [X_\mu^\bullet,J(X_\lambda^\bullet)]\in \ad(\mfg)$, then $x=\sum_{\mu\in \mcI_c}a_\mu X_\mu^\bullet$ is such that $W(x)\cap \ad(\mfg) \neq \{0\}$. By  Lemma \ref{L:Wx=/=ad}, no such $x$ can exist, and hence we have shown that $\mathrm{span}_{\mu \in \mcI_c}\{[X_\mu^\bullet,J(X_\lambda^\bullet)]\}$ intersects trivially with $\ad(\mfg)$, from which the linear independence of  $\{X_\gamma^\bullet, [X_\mu^\bullet,J(X_\lambda^\bullet)]\}_{\gamma\in \Lambda,\mu \in \mcI_c}$ follows automatically from the previous assertion and the linear independence of $\{X_\gamma^\bullet\}_{\gamma\in \Lambda}$. 

Let $\{f_\mu\}_{\mu\in \Lambda^\bullet}\subset (\End V)^*$ denote the dual basis to $\{X_\lambda^\bullet\}_{\lambda\in \Lambda^\bullet} \subset \End V$. By the linear independence of $\{X_\gamma^\bullet, [X_\mu^\bullet,J(X_\lambda^\bullet)]\}_{\gamma\in \Lambda,\mu \in \mcI_c}$, applying $f_\lambda\ot \mathrm{id}\ot \mathrm{id}$ to both sides of \eqref{J-comm} for a fixed $\lambda\in \Lambda$ yields 
\begin{equation*}
\sum_{\mu \in \mcI_c} [X_\mu^\bullet,J(X_\lambda^\bullet)]\ot X_\mu^\mcI=0.
\end{equation*}
The linear independence of $\{[X_\mu^\bullet,J(X_\lambda^\bullet)]\}_{\mu \in \mcI_c}$ then implies $X_\mu^\mcI=0$ for all $\mu\in \mcI_c$. \qedhere
\end{proof}
We define $\mathfrak{z}_\mcI$ similarly to $\mfz_\mcJ$: it is the commutative Lie algebra with basis $\{\mcK_\lambda^\mcI\}_{\lambda \in \mcI}$. We identify its enveloping algebra with the polynomial ring $\C[\mcK_\lambda^\mcI\,:\,\lambda \in \mcI]$, and set $\mathscr{K}^\mcI=\sum_{\lambda\in \mcI}X_\lambda^\bullet\ot \mcK_\lambda^\mcI\in \End V \ot \mathfrak{z}_\mcI$. 
We are now prepared to state the analogue of Proposition \ref{P:K}. 
\begin{proposition}\label{P:gI-iso}
 The assignment  $F^\mcI \mapsto F+\mathscr{K}^\mcI$ extends to an isomorphism of algebras 
  \begin{equation}
  \phi_{\mcI}:U_\mcI(\mfg)\iso  \C[\mcK_\lambda^\mcI\,:\,\lambda \in \mcI]\ot U(\mfg_\rho).\label{phi_I}
  \end{equation}
\end{proposition}
\begin{proof}
 Let $\pi:\C[\mcK_\lambda^\mcJ\,:\,\lambda \in \mcJ]\onto \C[\mcK_\lambda^\mcI\,:\,\lambda \in \mcI]$ be the surjection given by 
 \begin{equation*}
  \pi(\mcK_\lambda^\mcJ)=\begin{cases}
                     \mcK_\lambda^\mcI \; &\text{ if }\; \lambda \in \mcI,\\
                     0 \; &\text{ if }\lambda \in \mcI_c.
                    \end{cases}
 \end{equation*}
 Consider the tensor product $\pi\ot \mathrm{id}: \C[\mcK_\lambda^\mcJ\,:\,\lambda \in \mcJ]\ot U(\mfg_\rho)\onto  \C[\mcK_\lambda^\mcI\,:\,\lambda \in \mcI]\ot U(\mfg_\rho)$. Its kernel 
 is precisely the ideal generated by $\{\mcK_\lambda^\mcJ\}_{\lambda \in \mcI_c}$, which is the image of the ideal generated by $\{X_\mu^\mcJ \}_{ \mu \in \mcI_c}$ under the isomorphism $\phi_\mcJ$ of Proposition \ref{P:K}. By Lemma \ref{L:KI} and the definition of $U_\mcI(\mfg)$, this ideal is contained in the two-sided ideal $\mathscr{I}$ of
 $U_\mcJ(\mfg)$ generated by the relation $[K_2,(1\ot J)(\Omega_\rho)]=[K_1,(J\ot 1)(\Omega_\rho)]$, hence $\Ker((\pi \ot \mathrm{id})\circ \phi_\mcJ)\subset \mathscr{I}$. Since 
 $[\mathscr{K}_2^\mcI,(1\ot J)(\Omega_\rho)]=[\mathscr{K}_1^\mcI,(J\ot 1)(\Omega_\rho)]$ trivially holds in $\C[\mcK_\lambda^\mcI\,:\,\lambda \in \mcI]\ot U(\mfg_\rho)$, we indeed have 
 the equality $\Ker((\pi \ot \mathrm{id})\circ \phi_\mcJ)=\mathscr{I}$. Thus $(\pi \ot \mathrm{id})\circ \phi_\mcJ$ induces an isomorphism 
 $\phi_{\mcI}:U_\mcI(\mfg)\iso  \C[\mcK_\lambda^\mcI\,:\,\lambda \in \mcI]\ot U(\mfg_\rho)$ which is given by $F^\mcI \mapsto F+\mathscr{K}^\mcI$.\qedhere
\end{proof}
We conclude our discussion of $U_\mcI(\mfg)$ by emphasizing that Proposition \ref{P:gI-iso} can be naturally interpreted at the level of Lie algebras. Letting $\mfg_\mcI$ denote the Lie subalgebra of $\mathrm{Lie}(U_\mcI(\mfg))$ generated by $\{F_{ij}^\mcI\}_{1\leq i,j\leq N}$, we find that $\phi_\mcI|_{\mfg_\mcI}$ and its composition with $\mathrm{id}\ot \phi_\rho$ induce isomorphisms
\begin{equation}\label{gI-g+zI}
 \mfg_\mcI \iso \mfg_\rho\oplus \mfz_\mcI\iso \mfg\oplus \mfz_\mcI,
\end{equation}
and moreover that $U(\mfg_\mcI)\cong U_\mcI(\mfg)$. With this in mind, $U_\mcI(\mfg)$ will be denoted $U(\mfg_\mcI)$ from this point on. 
%
%
\subsubsection{The extended polynomial current algebra \texorpdfstring{$\mfg_\mcI[z]$}{}} \label{ssec:PCA-ext}
By \eqref{FI-br}, \eqref{KJ=JK} and \eqref{gI-g+zI}, the enveloping algebra $U(\mfg_\mcI[z])$ is isomorphic to the unital associative $\C$-algebra generated by elements $\{  \mdF_{ij}^{(r)}=F_{ij}^\mcI z^r\,:\, 1\leq i,j\leq N,\, r\in \Z_{\geq 0}\}$ subject to the defining relations 
\begin{gather}
   [\mdF_1^{(r)}, \mdF_2^{(s)}]=[\Omega_\rho,\mdF_2^{(r+s)}] \quad \forall\; r,s\geq 0, \label{ex:g[z]-R}\\
   [\mdK_2^{(r)},(1\ot J)(\Omega_\rho)]=[\mdK_1^{(r)},(J\ot 1)(\Omega_\rho)] \quad \forall\; r\geq 0,\label{g[z]:KJ=JK}
\end{gather}
 where $\mdF^{(a)}=\sum_{i,j=1}^N E_{ij}\otimes \mdF_{ij}^{(a)} \in \End V \otimes U(\mfg_\mcI[z])$ and $\mdK^{(a)}=\mdF^{(a)}-c_\mfg^{-1} \omega(\mdF^{(a)})$ for all $a\geq 0$. 

Following \eqref{F(u)}, let us define 
\begin{equation*}
 \mdF(u)=\sum_{i,j=1}^N E_{ij} \otimes \mdF_{ij}(u)\in \End V\ot (\mfg_\mcI[z])[\![u^{-1}]\!],\quad \text{ where }\quad \mdF_{ij}(u)=\sum_{r\geq 0} \mdF_{ij}^{(r)} u^{-r-1}\in (\mfg_\mcI[z])[\![u^{-1}]\!].
\end{equation*}
Recall that, for each $\lambda \in \Lambda^\bullet$, $X_\lambda^\mcI=\sum_{i,j} a_{ij}^\lambda  F_{ij}^\mcI\in \mfg_\mcI$, where the family of scalars $\{a_{ij}^\lambda\}$ is defined in \eqref{ChangeofBasis}. To every $\lambda \in \Lambda^\bullet$ we associate the series $\mdX_\lambda (u)=\sum_{r\geq 0} \mdX_\lambda^{(r)} u^{-r-1}\in (\mfg_\mcI[z])[\![u^{-1}]\!]$, where $\mdX_\lambda^{(r)}=X_\lambda^\mcI z^r$. 

Finally, we set  $\mcK_\lambda^{(r)}=\mcK_\lambda^\mcI z^{r-1}$, so that $U(\mfz_\mcI[z])\cong \C[\mcK_\lambda^{(r)}:\lambda \in \mcI,\, r\geq 1]$, and define 
\begin{equation*}
\mathscr{K}(u)=\sum_{\lambda \in \mcI}X_\lambda^\bullet \ot \mcK_\lambda(u), \quad \text{ where }\quad \mcK_\lambda(u)=\sum_{r\geq 1} \mcK_\lambda^{(r)} u^{-r}.
\end{equation*}
We can now state the polynomial current algebra version of Proposition \ref{P:gI-iso}:
\begin{proposition}\label{P:wtphi-gen}
The assignment $\mdF(u)\mapsto  F(u)+\mathscr{K}(u)$ extends to an isomorphism of algebras 
\begin{equation}
 \phi_\mcI^z: U(\mfg_\mcI[z])\iso \C[\mathcal{K}_\lambda^{(r)}:\lambda \in \mcI,\, r\geq 1]\ot U(\mfg_\rho[z]). \label{wtphi_z}
\end{equation}
\end{proposition}
\begin{proof}
 The isomorphism $\mfg_\mcI\iso \mfg_\rho\oplus \mfz_\mcI$ furnished by Proposition \ref{P:gI-iso} (see \eqref{gI-g+zI}) extends to an isomorphism $ \mfg_\mcI[z]\iso (\mfg_\rho\oplus \mfz_\mcI)[z]\cong \mfg_\rho[z]\oplus \mfz_\mcI[z]$, which induces the desired isomorphism $\phi_\mcI^z$ between the corresponding enveloping algebras. 
\end{proof}
Setting $\mdK(u)=\sum_{r\geq 0}\mdK^{(r)}u^{-r-1}$,  we have $\phi_\mcI^z(\mdK(u))=\mathscr{K}(u)$ and $\mdK(u)=\mdF(u)-c_\mfg^{-1}\omega(\mdF(u))$. By Lemma \ref{L:KI}, $\mdK(u)$ can be equivalently defined by $\mdK(u)=\sum_{\lambda \in \mcI}X_\lambda^\bullet \ot \mdX_\lambda(u)$.

We will end this section by rewriting the defining relations of $U(\mfg_\mcI[z])$ using the classical $r$-matrix formalism, which is achieved with the use of Proposition \ref{P:r-matrix}. 
\begin{proposition}\label{P:g_I[z]}
 The defining relations \eqref{ex:g[z]-R} and \eqref{g[z]:KJ=JK} are equivalent to the relations
\begin{align}
&[\mdF_1(u),\mdF_2(v)]=\left[\frac{\Omega_\rho}{u-v}, \mdF_1(u)+\mdF_2(v)\right], \label{g_I[z]:R}\\
&[\mdK_2(u),(1\ot J)(\Omega_\rho)]=[\mdK_1(u),(J\ot 1)(\Omega_\rho)], \label{g_I[z]:Cas}
\end{align}
where  $\mdK(u)=\mdF(u)-c_\mfg^{-1}\omega(\mdF(u))$.
\end{proposition}
%
%
%
\section{The \texorpdfstring{$R$}{R}-matrix presentation of the Yangian \texorpdfstring{$Y(\mfg)$}{}}\label{Sec:RTT}
We have now reached the second and main part of this paper, where we will focus on establishing the Yangian version of the results of Section \ref{Sec:r-matrix} and studying them in more detail. In this section specifically, we define the extended Yangian $X_\mcI(\mfg)$, the $RTT$-Yangian $Y_R(\mfg)$,  and we then study some of their basic properties.

We continue to assume that $V$ is a fixed finite-dimensional $Y(\mfg)$-module with corresponding homomorphism $\rho$, and that $V$ has a non-trivial (not necessarily proper) irreducible submodule. We let $R(u)$ denote the image of the universal $R$-matrix $\mcR(-u)$ (see Theorem \ref{T:DrThm3}) under $\rho\ot \rho$:
\begin{equation*}
 R(u)=(\rho\ot \rho)\mcR(-u)\in \End(V\otimes V)[\![u^{-1}]\!]. 
\end{equation*}
We adapt all of the notation from Section \ref{Sec:r-matrix}. In particular, we fix a basis  $\{e_1,\ldots,e_N\}$ of $V$ and we let $\{E_{ij}\}_{1\leq i,j\leq N}$ denote the usual elementary matrices with respect to this basis.

%
%
\subsection{The extended Yangian \texorpdfstring{$X_\mcI(\mfg)$}{}}
In this subsection we define and study a Hopf algebra $X_\mcI(\mfg)$ larger than $Y(\mfg)$ which we will eventually prove (in Section \ref{Sec:XR}) is a filtered deformation of $U(\mfg_\mcI[z])$.

\subsubsection{Definition of the extended Yangian}

We begin with the definition of $X_\mcI(\mfg)$ as an algebra.
\begin{definition}\label{D:Y}
 The extended Yangian $X_\mcI(\mfg)$ is the unital associative $\C$-algebra generated by elements $\{t_{ij}^{(r)}\,:\,1\leq i,j\leq N,\,r\geq 1 \}$ subject to the defining $RTT$-relation
 \begin{equation}
 R(u-v)T_1(u)T_2(v)=T_2(v)T_1(u)R(u-v) \quad \text{ in } \quad (\End V)^{\ot 2} \ot X_\mcI(\mfg)[\![v^{\pm 1},u^{\pm 1}]\!]\label{RTT-V}, 
 \end{equation}
where $T(u)=\sum_{i,j=1}^N E_{ij}\otimes t_{ij}(u)$ with $t_{ij}(u)=\delta_{ij}+\sum_{r\geq 1}t_{ij}^{(r)}u^{-r}$ for all $1\leq i,j\leq N$, and $R(u-v)$ has been identified with 
$R(u-v)\ot 1$. 
\end{definition}
\begin{remark}\label{R:rational}
An equivalent definition is obtained by replacing $R(u)$ by $f(u)R(u)$ for any fixed $f(u)\in 1+u^{-1}\C[\![u^{-1}]\!]$. In particular, if $V$ is irreducible then, by Theorem \ref{T:R-rat}, $R(u)$ can be replaced with a rational $R$-matrix.   
 
 Since no explicit description of the coefficients $\mcR_k$ of $\mcR(u)$ is known, $R(u)$ cannot be computed directly by evaluating $\mcR(-u)$. In practice, $R(u)$ is obtained by instead solving the equation \eqref{inter}. By Theorem \ref{T:R-rat}, this determines $R(u)$ up to multiplication by elements of $\C[\![u^{-1}]\!]$, provided $V$ is irreducible. See for example \cite[Proposition 3.13]{GRW}. 
 \end{remark}
Note that $X_\mcI(\mfg)$ comes equipped with a natural action on the underlying $Y(\mfg)$-module $V$. Namely, there is an algebra homomorphism
\begin{equation*}
 X_\mcI(\mfg)\to \End V, \quad T(u)\mapsto R(u).
\end{equation*}

A standard argument (see \cite[Theorem 1.5.1]{Mobook} and \cite{FRT}) shows that $X_\mcI(\mfg)$ is a Hopf algebra with coproduct $\Delta_\mcI$, antipode $S_\mcI$, and counit $\epsilon_\mcI$ given by 
\begin{equation*}
 \Delta_\mcI(T(u))=T_{[1]}(u)T_{[2]}(u), \quad S_\mcI(T(u))=T(u)^{-1},\quad \epsilon_\mcI(T(u))=I, 
\end{equation*}
respectively. Expressing $\Delta_\mcI$ in terms of the generating series $t_{ij}(u)$ and the generators $t_{ij}^{(r)}$, we have
\begin{equation*}
 \Delta_\mcI(t_{ij}(u))=\sum_{a=1}^N t_{ia}(u)\ot t_{aj}(u)\quad \text{ and }\quad \Delta_\mcI(t_{ij}^{(r)})=\sum_{a=1}^N \sum_{b=0}^r t_{ia}^{(b)}\ot t_{aj}^{(r-b)},
\end{equation*}
where $t_{kl}^{(0)}=\delta_{kl}$ for all $1\leq k,l\leq N$.

\subsubsection{Automorphisms of $X_\mcI(\mfg)$}
The extended Yangian $X_\mcI(\mfg)$ has at least two important families of automorphisms. The first family we will discuss turns out to be closely tied to the Yangian $Y_R(\mfg)$, as we will make precise in Subsection \ref{ssec:Y->X}. 

Recall that $\mcE=\End_{Y(\mfg)}V\subset \End V$, and consider the tensor product $\mcE\ot u^{-1}\C[\![u^{-1}]\!]$. This space can be identified with $\prod_{\lambda\in \mcI}(u^{-1}\C[\![u^{-1}]\!])_\lambda$, i.e. 
the collection of all tuples $(f_\lambda(u))_{\lambda\in \mcI}\subset u^{-1}\C[\![u^{-1}]\!]$, the identification being given by
\begin{equation}\label{mbF-tuple}
 (f_\lambda(u))_{\lambda\in \mcI}\in \prod_{\lambda\in \mcI}(u^{-1}\C[\![u^{-1}]\!])_\lambda\mapsto \mbf^\circ(u)=\sum_{\lambda\in \mcI} X_\lambda^\bullet \ot f_\lambda(u)\in \mcE\ot u^{-1}\C[\![u^{-1}]\!].
\end{equation}
Here $(u^{-1}\C[\![u^{-1}]\!])_\lambda$ just denotes a copy of $u^{-1}\C[\![u^{-1}]\!]$ associated to $\lambda$. The following lemma shows that the extended Yangian $X_\mcI(\mfg)$ admits a family of automorphisms indexed by $\prod_{\lambda\in \mcI}(u^{-1}\C[\![u^{-1}]\!])_\lambda$.
\begin{lemma}\label{L:auto}
 Let $(f_\lambda(u))_{\lambda\in \mcI}\in \prod_{\lambda\in \mcI}(u^{-1}\C[\![u^{-1}]\!])_\lambda$ and set $\mbf(u)=I+\mbf^\circ(u)$. Then the assignment 
 \begin{equation}
  m_{\mbf}:T(u)\mapsto \mbf(u)T(u)   \label{aut:m_F}
 \end{equation}
 extends to an automorphism $m_{\mbf}$ of $X_\mcI(\mfg)$.
\end{lemma}
\begin{proof}
 Using that $\mbf(u)\in \mcE\ot \C[\![u^{-1}]\!]$ and $R(u)\in (\rho(Y(\mfg))\ot\rho(Y(\mfg)))[\![u^{-1}]\!]$, we can conclude that $\mbf(u)$ satisfies the defining $RTT$-relation of $X_\mcI(\mfg)$. Indeed, by definition $\mcE$ is the centralizer of $\rho(Y(\mfg))$ in $\End V$, which implies $ R(u-v)\mbf_a(u)=\mbf_a(u)R(u-v)$ for $a\in \{1,2\}$. Moreover, $[\mbf_1(u),\mbf_2(v)]=0$, from which the assertion follows easily. 
 
 Applying this observation in conjunction with $[\mbf_1(u),T_2(v)]=0=[\mbf_2(v),T_1(u)]$, we deduce that $m_{\mbf}$ extends to an algebra endomorphism of $X_\mcI(\mfg)$. The invertibility of $m_{\mbf}$ follows from the invertibility of $\mbf(u)$ as an element $\mcE[\![u^{-1}]\!]$. \qedhere
\end{proof}
The second family of automorphisms is indexed by the complex numbers. For each $c\in \C$, the assignment 
\begin{equation}\label{tau_c:R}
T(u)\mapsto T(u-c)
\end{equation}
extends to an automorphism of $X_\mcI(\mfg)$. These automorphisms are closely related to the automorphisms $\tau_c$ of $Y(\mfg)$ defined in \eqref{J:shift}.

\subsubsection{The associated graded algebra $\gr X_\mcI(\mfg)$} 
By \eqref{R-exp}, the $R$-matrix $R(u)$ admits an expansion 
\begin{equation}\label{R(u)-eval}
R(u)=I+\sum_{k\geq 1}R^{(k)}u^{-k}=I-\Omega_\rho u^{-1}+\left((J\ot 1-1\ot J)(\Omega_\rho)+\tfrac{1}{2}\Omega_\rho^2\right)u^{-2}+\sum_{k\geq 3} R^{(k)} u^{-r}  
\end{equation}
with $R^{(k)}=(-1)^k(\rho\ot \rho)(\mcR_k)$ for each $k\geq 1$. Setting $T^\circ(u)=T(u)-I$, the defining relation \eqref{RTT-V} can be rewritten as
\begin{align}\label{eq1}
\begin{split}
  [T_1^\circ(u),T_2^\circ(v)]&=\frac{1}{u-v}\left([\Omega_\rho,T_1^\circ(u)]+[\Omega_\rho,T_2^\circ(v)]+\Omega_\rho T_1^\circ(u)T_2^\circ(v)-T_2^\circ(v)T_1^\circ(u)\Omega_\rho\right)\\
                          &+\sum_{k\geq 2}\frac{1}{(u-v)^k}\left([T_2^\circ(v),R^{(k)}]+[T_1^\circ(u), R^{(k)}]+T_2^\circ(v)T_1^\circ(u)R^{(k)}-R^{(k)}T_1^\circ(u)T_2^\circ(v) \right),
\end{split}
\end{align}
where $\Omega_\rho$ and $R^{(k)}$ have been identified with $\Omega_\rho\ot 1$ and $R^{(k)}\ot 1$, respectively. 

The degree assignment $\deg t_{ij}^{(r)}=r-1$ for all $1\leq i,j\leq N$ and $r\geq 1$ equips $X_\mcI(\mfg)$ with the structure of a filtered algebra.
 Let $\mathbf{F}_k(X_\mcI(\mfg))$ (or $\mbF_k^\mcI$ for brevity) denote the subspace spanned by elements of degree less than or equal to $k$, and set $\bar{t}_{ij}^{(r)}$ to be the image of $t_{ij}^{(r)}$ in $\mbF_{r-1}^\mcI/\mbF_{r-2}^\mcI\subset \gr X_\mcI(\mfg)$. 
\begin{proposition}\label{P:X-cur}
The assignment 
 \begin{equation}
  \varphi_\mcI: \mdF_{ij}^{(r-1)}\mapsto \bar{t}_{ij}^{(r)} \quad \forall \; 1\leq i,j\leq N,\; r\geq 1 \label{varphi_I}
 \end{equation}
extends to a surjective morphism of algebras $\varphi_\mcI:U(\mfg_\mcI[z])\onto \gr X_\mcI(\mfg)$. 
\end{proposition}
\begin{proof}
Let $\mathds{T}(u)=\sum_{k\geq 1} \mathds{T}^{(k)}u^{-k}$, where $\mathds{T}^{(k)}=\sum_{i,j=1}^N E_{ij}\ot \bar{t}_{ij}^{(k)}$.  
 
\noindent \textit{Step 1}: The relation $[\mdT_1(u),\mdT_2(v)]=\left[\frac{\Omega_\rho}{u-v},\mdT_1(u)+\mdT_2(v)\right]$ is satisfied. 

For each $k>0$, we expand $(u-v)^{-k}$ as an element of $(\C[v])[\![u^{-1}]\!]$: 
\begin{equation}\label{exp:1}
 (u-v)^{-k}=\sum_{s\geq 0} \binom{k+s-1}{s}v^{s} u^{-s-k}. 
\end{equation}
Note the following simple fact: if $\mathds{A}(u,v)=\sum_{a,b\geq 1}\mathds{A}_{a,b}u^{-a}v^{-b}$ with $\mathds{A}_{a,b}\in (\End V)^{\ot 2}\ot \mbF_{a+b-c}^\mcI$, then
\begin{equation}\label{obs}
\frac{1}{(u-v)^k}\mathds{A}(u,v)=\sum_{a\in \Z_{\geq k+1},b\in \Z} \mathds{B}_{a,b} u^{-a}v^{-b} \quad \text{ with }\; \mathds{B}_{a,b}\in (\End V)^{\ot 2}\ot \mbF_{a+b-c-k}^\mcI\quad \forall \; a,b\geq 0,
\end{equation}
 where $\mbF_{-l}^\mcI=\{0\}$ for all $l\in \mathbb{N}$. Here $c$ is assumed to be a fixed positive integer depending on $\mathds{A}(u,v)$. 

For each $l\geq 0$, set 
\begin{equation*}
\mbF_{l}(u,v)=(\End V)^{\ot 2}\ot \prod_{a\in \Z_{\geq 0},b\in \Z} \mbF^\mcI_{a+b-l} u^{-a}v^{-b}\subset (\End V)^{\ot 2}\ot X_\mcI(\mfg)[\![v^{\pm 1},u^{-1}]\!],
\end{equation*}
and note that $\mbF_l(u,v)/\mbF_{l+1}(u,v)$ can be naturally identified with
\begin{equation*}
(\End V)^{\ot 2}\ot \prod_{a\in \Z_{\geq 0},b\in \Z} (\gr_{a+b-l}X_\mcI(\mfg) ) u^{-a}v^{-b}\subset (\End V)^{\otimes 2}\ot (\gr X_\mcI(\mfg))[\![v^{\pm 1},u^{-1}]\!],
\end{equation*}
where $\gr_k X_\mcI(\mfg)$ denotes the $k$-th graded component of $\gr X_\mcI(\mfg)$, which is understood to equal zero if $k<0$. 

We will simultaneously show both sides of \eqref{eq1} belong to $\mbF_2(u,v)$ and compute their images in the quotient $\mbF_2(u,v)/\mbF_3(u,v)$. By the above observation this yields an identity in $(\End V)^{\otimes 2}\ot (\gr X_\mcI(\mfg))[\![v^{\pm 1},u^{-1}]\!]$.

If $\mathds{A}(u,v)=\Omega_\rho T_1^\circ(u)T_2^\circ(v)$ or $\mathds{A}(u,v)=T_2^\circ(v)T_1^\circ(u)\Omega_\rho$, then the integer $c$ (see \eqref{obs}) is equal to $2$ and hence 
$(u-v)^{-1}\mathds{A}(u,v)\equiv 0 \mod \mbF_3(u,v)$. 

If instead $\mathds{A}(u,v)$ is equal to one of the terms that appears within the parentheses on the second line of the right-hand side of \eqref{eq1} (i.e. a term involving $R^{(k)}$ with $k\geq 2$), then 
$c=1$ or $2$ but $k\geq 2$. Therefore the observation \eqref{obs} yields that $(u-v)^{-k}\mathds{A}(u,v)\equiv 0 \mod \mbF_3(u,v)$.

Since both $[T_1^\circ(u),T_2^\circ(v)]$ and $[\frac{\Omega_\rho}{u-v},T_1^\circ(u)+T_2^\circ(v)]$ belong to $\mbF_{2}(u,v)$ with images 
$[\mdT_1(u),\mdT_2(v)]$ and $[\frac{\Omega_\rho}{u-v},\mdT_1(u)+\mdT_2(v)]$ in $\mbF_{2}(u,v)/\mbF_3(u,v)$, respectively, we obtain the relation 
\begin{equation*}
 [\mdT_1(u),\mdT_2(v)]=\left[\frac{\Omega_\rho}{u-v},\mdT_1(u)+\mdT_2(v)\right]. 
\end{equation*}
Note that Step 1 implies that there is a surjective algebra homomorphism $U(\mfg_\mcJ[z])\onto \gr X_\mcI(\mfg)$. To verify that it factors through 
$U(\mfg_\mcI[z])$ we must show that the assignment $\varphi_\mcI$ preserves the relation \eqref{g_I[z]:Cas}. In order to state this more precisely we define, for each $\lambda \in \Lambda^\bullet$ and $k\geq 1$, $\bar{t}_\lambda^{(k)}=\sum_{i,j=1}^N a_{ij}^\lambda \bar{t}_{ij}^{(k)}$ (see \eqref{ChangeofBasis}). Then the statement that $\varphi_\mcI$ preserves 
\eqref{g_I[z]:Cas} is equivalent to the statement that, for each $k\geq 1$, $\mdD^{(k)}=\sum_{\lambda\in \mcJ} X_\lambda^\bullet \ot \bar{t}_\lambda^{(k)}$ satisfies
\begin{equation}
[\mdD^{(k)}_2,(1\ot J)(\Omega_\rho)]=[\mdD^{(k)}_1,(J\ot 1)(\Omega_\rho)]. \label{I-rel:gr} 
\end{equation}

\noindent \textit{Step 2: } the relation \eqref{I-rel:gr} is satisfied for every $k\geq 1$. 

We will divide this step of the proof into a few smaller steps. 

\noindent \textit{Step 2.1: } The relation 
\begin{equation}
 [\mdT_2^{(k)},(J\ot 1-1\ot J)(\Omega_\rho)]=-[\mdT_1^{(k)},(J\ot 1-1\ot J)(\Omega_\rho)] \label{TJ=JT}
\end{equation}
holds in $\gr X_\mcI(\mfg)$ for all $k\geq 1$.

We will prove \eqref{TJ=JT} by expanding \eqref{eq1} in two different ways. First,  we compute for each $k\geq 1$ the $v^0u^{-k-2}$ coefficient of both sides of \eqref{eq1} modulo $\mbF^\mcI_{k-2}$, using the expansion \eqref{exp:1}. Using \eqref{obs}, it is not difficult to deduce that no term on the right-hand side of \eqref{eq1} involving $R^{(k)}$ with $k\geq 3$ makes a contribution, and the same is true for the terms $T_2^\circ(v)T_1^\circ(u)R^{(2)}$ and $R^{(2)}T_1^\circ(u)T_2^\circ(v)$. As the coefficient of $v^0 u^{-k-2}$ 
in $[T_1^\circ(u),T_2^\circ(v)]$ is zero, we are left with the equivalence 
\begin{align}
\begin{split}\label{u-exp}
 0\equiv &[\Omega_\rho,T_1^{(k+1)}]+[\Omega_\rho,T_2^{(k+1)}]\\
                        &+\sum_{a=1}^{k}(\Omega_\rho T_1^{(k+1-a)}T_2^{(a)}- T_2^{(a)}T_1^{(k+1-a)}\Omega_\rho)+[T_1^{(k)},R^{(2)}]+(k+1)[T_2^{(k)},R^{(2)}]\mod \mbF^\mcI_{k-2}.
\end{split}
\end{align}
Next, we compute the $u^0v^{-k-2}$ coefficient of both sides of \eqref{eq1} modulo $\mbF^\mcI_{k-2}$ after expanding $(u-v)^{-r}$ as an element of $(\C[u])[\![v^{-1}]\!]$
and viewing \eqref{eq1} as an equality in $(\End V)^{\ot 2} \ot X_\mcI(\mfg)[\![u^{\pm 1}, v^{-1}]\!]$. Using the symmetry and skew-symmetry between $u$ and $v$ in the relation \eqref{eq1}, we deduce from \eqref{u-exp} the equivalence
\begin{align}
\begin{split}\label{v-exp}
 0\equiv &-[\Omega_\rho,T_1^{(k+1)}]-[\Omega_\rho,T_2^{(k+1)}]\\
                        &-\sum_{b=1}^{k}(\Omega_\rho T_1^{(b)}T_2^{(k+1-b)}- T_2^{(k+1-b)}T_1^{(b)}\Omega_\rho)+(k+1)[T_1^{(k)},R^{(2)}]+[T_2^{(k)},R^{(2)}]\mod \mbF^\mcI_{k-2}.
\end{split}                        
\end{align}
Adding \eqref{u-exp} and \eqref{v-exp} and dividing by $k+2$, we obtain 
\begin{equation*}
 [T_2^{(k)},R^{(2)}]\equiv -[T_1^{(k)},R^{(2)}] \mod \mbF^\mcI_{k-2}\implies [\mdT_2^{(k)},R^{(2)}]= -[\mdT_1^{(k)},R^{(2)}] \; \text{ in }\; \gr X_\mcI(\mfg). 
\end{equation*}
Recall from \eqref{R(u)-eval} that $R^{(2)}=(J\ot 1-1\ot J)(\Omega_{\rho})+\tfrac{1}{2}\Omega_{\rho}^2$. Substituting this into the above equality gives
\begin{equation}\label{eq4}
 [\mdT_2^{(k)},(J\ot 1-1\ot J)(\Omega_\rho)]+\tfrac{1}{2}[\mdT_2^{(k)},\Omega_\rho^2]= -[\mdT_1^{(k)},(J\ot 1-1\ot J)(\Omega_\rho)]-\tfrac{1}{2}[\mdT_1^{(k)},\Omega_\rho^2]. 
\end{equation}
Since $\mdT^{(k)}$ is a homomorphic image of $F^\mcJ z^{k-1}\in U(\mfg_\mcJ[z])$, Lemma \ref{L:F^J,K} yields $[\mdT_2^{(k)},\Omega_\rho]=-[\mdT_1^{(k)},\Omega_\rho]$, from which the identity
\begin{equation*}
 \tfrac{1}{2}[\mdT_2^{(k)},\Omega_\rho^2]=-\tfrac{1}{2}[\mdT_1^{(k)},\Omega_\rho^2]
\end{equation*}
follows directly. Therefore the relation \eqref{eq4} implies the relation \eqref{TJ=JT}. 

\noindent \textit{Step 2.2: } We have 
\begin{equation}\label{DJ-JD}
 [\mdD_2^{(k)},(J\ot 1-1\ot J)(\Omega_\rho)]=-[\mdD_1^{(k)},(J\ot 1-1\ot J)(\Omega_\rho)].
\end{equation}
for each $k\geq 1$. 

For each $k\geq 1$, set $\mathds{L}^{(k)}=\mdT^{(k)}-\mdD^{(k)}$, so that $\mdL^{(k)}=\sum_{\lambda\in\Lambda}X_\lambda^\bullet \ot \bar{t}_\lambda^{(k)}$. 
Using that $J:\ad(\mfg)\to \ad_\mfg(\mfgl(V))$ is a morphism of $\mfg$-modules, it is straightforward to derive from the relation $[\Omega_\rho,\mdL_2^{(k)}]=-[\Omega_\rho,\mdL_1^{(k)}]$ that 
\begin{equation}\label{LJ-JL}
 [\mdL_2^{(k)},(J\ot 1-1\ot J)(\Omega_\rho)]=-[\mdL_1^{(k)},(J\ot 1-1\ot J)(\Omega_\rho)].
\end{equation}
Subtracting \eqref{LJ-JL} from \eqref{TJ=JT} yields \eqref{DJ-JD}.

Since $\mdD^{(k)}$ is a homomorphic image of $\sum_{i,j=1}^N E_{ij}\ot K_{ij}z^{k-1}\in \End V\ot U(\mfg_\mcJ[z])$, Lemma \ref{L:K-rels} implies that $ [\Omega_\rho,\mdD^{(k)}_2]=0=[\Omega_\rho,\mdD_1^{(k)}]$.
We thus also have 
\begin{equation*}
 [\mdD_2^{(k)},(J\ot 1)(\Omega_\rho)]=0=[\mdD_1^{(k)},(1\ot J)(\Omega_\rho)].
\end{equation*}
Subtracting this identity from \eqref{DJ-JD}  leaves us with the equality 
\eqref{I-rel:gr}.

By Step 1, Step 2 and Proposition \ref{P:g_I[z]}, the assignment \eqref{varphi_I} extends to an epimorphism $\varphi_\mcI: U(\mfg_\mcI[z])\onto \gr X_\mcI(\mfg)$. \qedhere
\end{proof}

We will prove that $\varphi_\mcI$ is in fact an isomorphism, but this will be delayed until Subsection \ref{ssec:X-str}.  
%
%
\subsection{The \texorpdfstring{$RTT$}--Yangian \texorpdfstring{$Y_R(\mfg)$}{}}
Our present goal is to give an exposition of $Y_R(\mfg)$ analogous to that given for $X_\mcI(\mfg)$ in the previous subsection. 

\subsubsection{Definition of the \texorpdfstring{$RTT$}--Yangian }
Let us begin with the definition of the Yangian $Y_R(\mfg)$: 
\begin{definition}
The $RTT$-Yangian $Y_R(\mfg)$ is the quotient of $X_\mcI(\mfg)$ by the two-sided ideal generated by the elements $z_{ij}^{(r)}$, for $1\leq i,j\leq N$ and $r\geq 1$, defined by 
\begin{equation}
 \mcZ(u)=\sum_{i,j=1}^N E_{ij}\ot z_{ij}(u)=S_\mcI^2(T(u))T(u+\tfrac{1}{2}c_\mathfrak{g})^{-1}, \quad  \label{Y-sym}
\end{equation}
where $z_{ij}(u)=\delta_{ij}+\sum_{r\geq 1} z_{ij}^{(r)}u^{-r}$ for each pair of indices $1\leq i,j\leq N$. 
\end{definition}
The ideal of $X_\mcI(\mfg)$ generated by $\{z_{ij}^{(r)}\,:\, 1\leq i,j\leq N,\, r\geq 1\}$ will be conveniently denoted by  $(\mcZ(u)-I)$. 
Note that it is not obvious that this ideal is a Hopf ideal, and hence that $Y_R(\mfg)$ inherits the structure of a Hopf algebra from $X_\mcI(\mfg)$. This will, however, be a consequence of Lemma \ref{L:X->YJ} and Theorem \ref{T:YR->YJ}, which will be proven in the next section. 

We will denote the images of $t_{ij}^{(r)}, t_{ij}(u)$ and $T(u)$ in $Y_R(\mfg)$ by $\tau_{ij}^{(r)}, \tau_{ij}(u)$ and $\mathcal{T}(u)$, respectively. 

For each $c\in \C$ the automorphism \eqref{tau_c:R} factors through the Yangian $Y_R(\mfg)$ yielding an automorphism given by the assignment $\mcT(u)\mapsto \mcT(u-c)$. We will prove in Subsection \ref{ssec:Y->X} that each automorphism $m_{\mbf}$ of $X_\mcI(\mfg)$ (see Lemma \ref{L:auto}) also induces an automorphism of $Y_R(\mfg)$, but that these turn out to all be equal to the identity map. This fact will be used to give an equivalent characterization of $Y_R(\mfg)$.

\subsubsection{The associated graded algebra $\gr Y_R(\mfg)$} 
The $RTT$-Yangian $Y_R(\mfg)$ inherits an algebra filtration from $X_\mcI(\mfg)$ via the quotient filtration; this is equivalent to assigning $\deg \tau_{ij}^{(r)}=r-1$. Let $\bar\tau_{ij}^{(r)}$ denote the image of $\tau_{ij}^{(r)}$ in $\mbF_{r-1}(Y_R(\mfg))/\mbF_{r-2}(Y_R(\mfg))=\gr_{r-1}Y_R(\mfg)$. 
\begin{proposition} \label{P:Y-cur}
  The assignment 
 \begin{equation*}
 \varphi: F_{ij}^{(r-1)}\mapsto \bar{\tau}_{ij}^{(r)} \quad \forall\; 1\leq i,j\leq N,\; r\geq 1
 \end{equation*}
extends to a surjective algebra morphism $\varphi:U(\mfg_\rho[z])\onto\gr Y_R(\mfg)$.
\end{proposition}
\begin{proof}  We will take a slightly more explicit route than taken in the proof of Proposition \ref{P:X-cur} and work directly with the generators $\bar{\tau}_{ij}^{(r)}$ of 
$\gr Y_R(\mfg)$. 
 By \eqref{F-sym'}, Corollary \ref{C:g[z]} and Proposition \ref{P:X-cur} it suffices to show that 
\begin{equation}
  \tfrac{1}{2}\bar{\tau}_{ij}^{(r)}= c_\mfg^{-1}\sum_{a=1}^N[\bar{\tau}_{ia}^{(r)},\bar{\tau}_{aj}^{(1)}] \quad \forall \; r\geq 1. \label{eq:U->Y}
\end{equation}
In $Y_R(\mfg)$ we have, by \eqref{Y-sym}, the relation $\mcT(u+\tfrac{1}{2}c_\mfg)=S_\mcI^2(\mcT(u))$ where $S_\mcI^2(\mcT(u))$ is understood to equal the image of $S^2_\mcI(T(u))$ in $Y_R(\mfg)$ under the natural quotient map. Since 
\begin{equation*}
\left(u+\tfrac{1}{2}c_\mfg\right)^{-k}=\sum_{s\geq 0}\binom{k+s-1}{s}\left(-\tfrac{1}{2}c_\mfg\right)^s u^{-s-k}\quad \forall \; k\geq 1,
\end{equation*}
the $u^{-r-1}$ coefficient of $\tau_{ij}(u+\tfrac{1}{2}c_\mfg)$ is equal to 
\begin{equation}
\tau_{ij}^{(r+1)}-\tfrac{r}{2}c_\mfg \tau_{ij}^{(r)} \mod \mbF_{r-2}(Y_R(\mfg)). \label{eq:s-check}
\end{equation}
Let $\wh T^{(r)}=\sum_{i,j=1}^N E_{ij}\otimes \wh t_{ij}^{(r)}$ denote the $u^{-r}$ coefficient of $T(u)^{-1}$. In particular, $\wh T^{(r)}$ can be determined inductively from the relation $\wh T^{(r)}=-\sum_{b=1}^r T^{(b)}\wh T^{(r-b)}=-\sum_{b=1}^r \sum_{i,j=1}^N E_{ij}\otimes \left(\sum_{a=1}^Nt_{ia}^{(b)} \wh t_{aj}^{(r-b)}\right)$. By definition of the antipode $S_\mcI$, we thus have 
\begin{equation*}
  S_\mcI^2(t_{ij}^{(r+1)})=-S_\mcI(\sum_{b=1}^{r+1}\sum_{a=1}^Nt_{ia}^{(b)} \wh t_{aj}^{(r+1-b)})=-\sum_{b=1}^{r+1}\sum_{a=1}^N S_\mcI(\wh t_{aj}^{(r+1-b)}) \wh t_{ia}^{(b)}.
\end{equation*} 
Expanding the right-hand side and using that $S_\mcI$ is a filtration preserving map with $S_\mcI(t_{kl}^{(s)})=\wh t_{kl}^{(s)}\equiv -t_{kl}^{(s)} \mod \mbF^\mcI_{s-2}$ for each $s\geq 1$, we obtain
\begin{align*}
 S_\mcI^2(t_{ij}^{(r+1)})&=\sum_{b=1}^{r+1}\sum_{a=1}^N S_\mcI(\wh t_{aj}^{(r+1-b)})(t_{ia}^{(b)}+\sum_{d=1}^{b-1}\sum_{c=1}^N t_{ic}^{(d)}\wh t_{ca}^{(b-d)})\\
                    &\equiv t_{ij}^{(r+1)}-\sum_{b=1}^{r}\sum_{a=1}^N \wh t_{aj}^{(r+1-b)}t_{ia}^{(b)}+\sum_{d=1}^{r}\sum_{c=1}^N t_{ic}^{(d)}\wh t_{cj}^{(r+1-d)} \mod \mbF^\mcI_{r-2}\\
                    &\equiv t_{ij}^{(r+1)}+\sum_{b=1}^{r}\sum_{a=1}^N [t_{aj}^{(r+1-b)},t_{ia}^{(b)}] \mod \mbF^\mcI_{r-2}.
\end{align*}
Combining this with the relation $[\mdT_1^{(r)},\mdT_2^{(s)}]=[\Omega_\rho,\mdT_2^{(r+s)}]$ of $\gr X_\mcI(\mfg)$ (which holds by Proposition \ref{P:X-cur}),
we arrive at the relation 
\begin{equation*}
 S_\mcI^2(t_{ij}^{(r+1)})\equiv t_{ij}^{(r+1)}+r\sum_{a=1}^N [t_{aj}^{(1)},t_{ia}^{(r)}] \mod \mbF^\mcI_{r-2}.
\end{equation*}
As the same relation must hold in $Y_R(\mfg)/\mbF_{r-2}(Y_R(\mfg))$ with each generator $t_{kl}^{(s)}$ replaced by $\tau_{kl}^{(s)}$, equating the resulting expression with \eqref{eq:s-check} and subtracting $\tau_{ij}^{(r+1)}$ from both sides gives \eqref{eq:U->Y}. \qedhere
\end{proof}
We conclude this section by noting a simple, but rather useful, corollary of Proposition \ref{P:Y-cur}. 
\begin{corollary}\label{C:mingen}
 The algebra $Y_R(\mfg)$ is generated by the elements $\tau_{ij}^{(r)}$ with $1\leq i,j\leq N$ and $1\leq r\leq 2$. 
\end{corollary}
\begin{proof} 
 Since $U(\mfg_\rho[z])$ is generated by $\{F_{ij}^{(0)},F_{ij}^{(1)}\}_{1\leq i,j\leq N}$ and $\vphi:U(\mfg_\rho[z])\to \gr Y_R(\mfg)$ is surjective, 
 the associated graded algebra $\gr Y_R(\mfg)$ is generated by $\{\bar \tau_{ij}^{(1)},\bar \tau_{ij}^{(2)}\}_{1\leq i,j\leq N}$. If $r>2$, then 
 we may write $\bar \tau_{ij}^{(r)}$ as a homogeneous polynomial $Q$ in the variables $\{\bar\tau_{kl}^{(1)},\bar\tau_{kl}^{(2)}\}_{1\leq i,j\leq N}$ of degree $r-1$. Let $P$ be 
 the polynomial in $\{\tau_{kl}^{(1)},\tau_{kl}^{(2)}\}_{1\leq i,j\leq N}$  obtained from $Q$ by replacing $\bar\tau_{kl}^{(s)}$ with $\tau_{kl}^{(s)}$ for $s=1,2$ and $1\leq k,l\leq N$. Then $P\in \mbF_{r-1}(Y_R(\mfg))$ and $\tau_{ij}^{(r)}-P\in \mbF_{r-2}(Y_R(\mfg))$. The result thus follows by a straightforward induction on $r\geq 1$. 
\end{proof}
%
\section{Equivalence of the two definitions of the Yangian}\label{Sec:YR->YJ}
In this section we prove that, irrespective of the choice of $V$, we always have $Y_R(\mfg)\cong Y(\mfg)$.  In the process we prove that the surjection $\varphi:U(\mfg_\rho[z])\onto \gr Y_R(\mfg)$ from Proposition \ref{P:Y-cur} is an isomorphism, yielding a Poincar\'{e}-Birkhoff-Witt theorem for $Y_R(\mfg)$: see Theorem \ref{T:PBW}. This in turn implies that the center of $Y_R(\mfg)$ is trivial, as will be explained in Corollary \ref{C:ZY=1}.

The first step in proving the equivalence of the two Yangians is the construction of a surjective 
Hopf algebra homomorphism $X_\mcI(\mfg)\onto Y(\mfg)$, and this is the content of the next lemma.
\begin{lemma}\label{L:X->YJ}
 The assignment 
 \begin{equation}
  \wt \Phi: T(u) \to (\rho\otimes 1)(\mcR(-u))
 \end{equation}
 extends to a surjective homomorphism of Hopf algebras $\wt \Phi:X_\mcI(\mfg)\onto Y(\mfg)$.
\end{lemma}
\begin{proof} 
The lemma follows from the same kind of arguments as used to prove \cite[Theorem 3.16]{GRW}.
 By \eqref{QYBE}, $\mcR(u)$ satisfies 
 \begin{equation*}
 \mcR_{12}(v-u)\mcR_{13}(-u)\mcR_{23}(-v)=\mcR_{23}(-v)\mcR_{13}(-u)\mcR_{12}(v-u).
 \end{equation*}
 Applying the homomorphism
 $\rho\otimes \rho\otimes 1$ to both sides of this relation we obtain that $\wt \Phi(T(u))$ satisfies the defining $RTT$-relation \eqref{RTT-V}. Therefore,
 $\wt \Phi$ extends to a homomorphism  $\wt \Phi:X_\mcI(\mfg)\to Y(\mfg)$. By \eqref{R-exp}, 
 \begin{equation}\label{R(-u)}
  \mcR(-u)=1-\Omega u^{-1}+\sum_{\lambda\in \Lambda}(J(X_\lambda)\otimes X_\lambda -X_\lambda\otimes J(X_\lambda))u^{-2}+\tfrac{1}{2}\Omega^2 u^{-2} + O(u^{-3}).
 \end{equation}
After applying $\rho\otimes 1$ to both sides, we obtain that
\begin{equation}
 \wt \Phi(t_{ij}^{(1)})=\mcF_{ij} \quad \text{ and }\quad \wt \Phi(t_{ij}^{(2)})\equiv J(\mcF_{ij}) \mod \mbF_0^J\; \text{ for all }\; 1\leq i,j\leq N, \label{wtPhi-exp}
\end{equation}
where we recall that the elements $\mcF_{ij}\in \mfg$, which were defined in the proof of Proposition \ref{P:new-g}, are determined by $\sum_{i,j=1}^N E_{ij}\ot \mcF_{ij}=-(\rho\ot 1)\Omega$. 
 
Since $\mbF_0^J=U(\mfg)$ is generated by $\{\mcF_{ij}\}_{1\leq i,j\leq N}$, this shows that $\wt\Phi$ is surjective. The proof that $\wt \Phi$ is a coalgebra morphism commuting with the antipodes of $X_\mcI(\mfg)$ and $Y(\mfg)$ follows from the relations \eqref{co-R} and \eqref{S(R),eps(R)}: see the proof of \cite[Theorem 3.16]{GRW}. 
\end{proof}
We are now prepared to prove that $Y_R(\mfg)$ and $Y(\mfg)$ are isomorphic. 
\begin{theorem}\label{T:YR->YJ}
 The homomorphism $\wt \Phi$ factors through the quotient algebra $Y_R(\mfg)=X_\mcI(\mfg)/(\mcZ(u)-I)$ to yield an isomorphism of algebras 
 $\Phi:Y_R(\mfg)\iso Y(\mfg)$ which sends $\mcT(u)$ to $(\rho\otimes 1)(\mcR(-u))$. 
\end{theorem}
\begin{proof}
  By Corollary \ref{C:S2} the relation $S^2=\tau_{-\frac{1}{2}c_\mfg}$ is satisfied in $Y(\mfg)$ and by the second identity of \eqref{R-inv,shift} we have 
  $(1\otimes \tau_{-\frac{1}{2}c_\mfg})(\mcR(-u))=\mcR(-u-\frac{1}{2}c_\mfg)$.
  This shows that $(1\otimes S^2)(\mcR(-u))\mcR(-u-\tfrac{1}{2}c_\mfg)^{-1}=1$. Applying $\rho\otimes 1$ to both sides of this equality and using that $\wt\Phi$ is a morphism of Hopf algebras, we arrive at the relation 
\begin{equation*}
\wt \Phi(\mcZ(u))=\wt \Phi(S_\mcI^2(T(u)))\wt \Phi(T(u+\tfrac{1}{2}c_\mfg))^{-1}=I. 
\end{equation*}
This proves that $(\mcZ(u)-I)\subset \Ker\, \wt \Phi$ and hence that $\wt \Phi$ factors through $Y_R(\mfg)$ to yield an algebra epimorphism $\Phi:Y_R(\mfg)\onto Y(\mfg)$ determined by $\mcT(u)\mapsto(\rho\otimes 1)(\mcR(-u))$.

By \eqref{wtPhi-exp}, $\Phi(\tau_{ij}^{(1)})=\mcF_{ij}$ for all $1\leq i,j\leq N$ and $\Phi(\tau_{ij}^{(2)})\equiv J(\mcF_{ij}) \mod U(\mfg)$ for all $1\leq i,j\leq N$.  Since, by Corollary \ref{C:mingen}, $Y_R(\mfg)$ is generated by $\{\tau_{ij}^{(1)},\tau_{ij}^{(2)}\}_{1\leq i,j\leq N}$, this shows that $\Phi$ is a filtered homomorphism. To conclude that $\Phi$ is an isomorphism, it is enough to show that the associated graded morphism $\gr \Phi:\gr Y_R(\mfg)\to \gr Y(\mfg)$ is an isomorphism. 

Set $\varphi_\bullet=(\phi_\rho^z)^{-1}\circ \varphi_J^{-1}\circ \gr \Phi: \gr Y_R(\mfg)\to U(\mfg_\rho[z])$, where $\varphi_J$ is the isomorphism $U(\mfg[z])\iso \gr Y(\mfg)$ of Proposition \ref{P:grJ} and $\phi_\rho^z:U(\mfg_\rho[z])\iso U(\mfg[z])$ is the isomorphism of Corollary \ref{C:g[z]}.
This morphism sends $\bar{\tau}_{ij}^{(r)}$ to 
$F_{ij}^{(r-1)}=F_{ij}z^{r-1}$ for all $1\leq r\leq 2$ and $1\leq i,j\leq N$. Consider the composition $\varphi_\bullet\circ \varphi$ where $\varphi:U(\mfg_\rho[z])\onto \gr Y_R(\mfg)$ is the epimorphism of Proposition \ref{P:Y-cur}. This composition 
sends $F_{ij}^{(r)}$ to $F_{ij}^{(r)}$ for $r=0,1$ and hence is equal to the identity morphism. Therefore $\gr \Phi$, and thus $\Phi$, is an isomorphism. 
\end{proof}
In particular, we have shown that the ideal $(\mcZ(u)-I)$ is the kernel of the Hopf algebra morphism $\wt \Phi$, and hence is a Hopf ideal. The Yangian $Y_R(\mfg)$ thus inherits from $X_\mcI(\mfg)$ the unique Hopf algebra structure  such that $\Phi$ becomes an isomorphism of Hopf algebras. Explicitly, it has coproduct $\Delta_R$, antipode $S_R$, and counit $\eps_R$ given by 
\begin{equation}\label{Hopf-Y_R}
 \Delta_R(\mcT(u))=\mcT_{[1]}(u)\mcT_{[2]}(u),\quad S_R(\mcT(u))=\mcT(u)^{-1},\quad \eps_R(\mcT(u))=I.
\end{equation}

As was noted in Remark \ref{R:rational}, the coefficients $\mcR_k$ of the universal $R$-matrix $\mcR(u)$ have not been explicitly written down, and consequently the elements $\Phi(\tau_{ij}^{(r)})$ do not in general admit an explicit description. Nonetheless, such a description does exist for the images of the elements $\{\tau_{ij}^{(1)},\tau_{ij}^{(2)}\}_{1\leq i,j\leq N}$ which, by Corollary \ref{C:mingen}, do generate $Y_R(\mfg)$. Since the $J$-presentation $Y(\mfg)$ of the Yangian is defined only in terms of degree one and degree zero generators, it is perhaps more natural to rephrase this observation by stating that $\Phi^{-1}$ can be concretely described, which is the purpose of the next corollary. 
\begin{corollary}
 For each $1\leq i,j\leq N$ let $\{b_{kl}^{(ij)}\}_{1\leq k,l\leq N}\subset \C$ be defined by $(J\ot 1)(\mcF)=\sum_{i,j=1}^N E_{ij}\ot (\sum_{k,l=1}^N b_{kl}^{(ij)} \mcF_{kl})$. Then $\Phi^{-1}$ is determined on the generating set $\{\mcF_{ij},J(\mcF_{ij})\}_{1\leq i,j\leq N}$ of $Y(\mfg)$ by
\begin{equation}\label{Phi-inv}
  \mcF_{ij}\mapsto \tau_{ij}^{(1)},\quad J(\mcF_{ij})\mapsto\tau_{ij}^{(2)}-\tfrac{1}{2}\sum_{a=1}^N \tau_{ia}^{(1)}\tau_{aj}^{(1)}+\sum_{k,l=1}^N b_{kl}^{(ij)}\tau_{kl}^{(1)}
  \quad \forall \; 1\leq i,j\leq N. 
\end{equation}
\end{corollary}
\begin{proof}
 For each $r\geq 1$, set $\mcT^{(r)}=\sum_{i,j=1}^N E_{ij}\ot \tau_{ij}^{(r)}$ and define $J(\mcF)=\sum_{i,j=1}^N E_{ij}\ot J(\mcF_{ij})\in \End V\ot Y(\mfg)$. Then, using the expansion \eqref{R(-u)} we find that $\Phi(\mcT^{(1)})=\mcF$ and $\Phi(\mcT^{(2)})=J(\mcF)-(J\ot 1)(\mcF)+\tfrac{1}{2}\mcF^2$. 
Thus, $\Phi^{-1}(J(\mcF))=\mcT^{(2)}-\tfrac{1}{2}(\mcT^{(1)})^2+(J\ot 1)(\mcT^{(1)})$, which implies \eqref{Phi-inv}. \qedhere
\end{proof}
\begin{remark}
 When $\mfg$ is a symplectic or orthogonal Lie algebra and $V$ is its vector representation, it was proven in Proposition 3.19 of \cite{GRW} directly that the assignment 
 \eqref{Phi-inv} extends to an isomorphism $Y(\mfg)\iso Y_R(\mfg)$. In that case, and more generally in any case where $\rho(J(X))=0$ for all $X\in \mfg$, the term involving the coefficients $b_{kl}^{(ij)}$ in \eqref{Phi-inv} vanishes and we have $J(\mcF_{ij})\mapsto\tau_{ij}^{(2)}-\tfrac{1}{2}\sum_{a=1}^N \tau_{ia}^{(1)}\tau_{aj}^{(1)}$. 
\end{remark}
In the process of proving Theorem \ref{T:YR->YJ} we have also shown that the homomorphism $\varphi$ of Proposition \ref{P:Y-cur} is injective. We thus
obtain the following Poincar\'{e}-Birkhoff-Witt type theorem for  $Y_R(\mfg)$: 
\begin{theorem}\label{T:PBW}
 The surjective homomorphism $\varphi:U(\mfg_\rho[z])\onto\gr Y_R(\mfg)$ of Proposition \ref{P:Y-cur}, which is given by $F_{ij}^{(r-1)}\to \bar{\tau}_{ij}^{(r)}$,  is an isomorphism of algebras. Consequently, the assignment 
 \begin{equation*}
  F_{ij}\mapsto \tau_{ij}^{(1)} \quad \forall \;1\leq i,j\leq N
 \end{equation*}
defines an embedding $U(\mfg_\rho)\into Y_R(\mfg)$. 
\end{theorem}
The above theorem can be employed to obtain a complete description of the center of $Y_R(\mfg)$:
\begin{corollary}\label{C:ZY=1}
 The center of $Y_R(\mfg)$ is equal to $\C\cdot 1$. 
\end{corollary}
\begin{proof}
 The center of the universal enveloping algebra $U(\mfg_\rho[z])\cong U(\mfg[z])$ is known to be trivial: see for instance \cite[Lemma 1.7.4]{Mobook}. As a consequence of Theorem \ref{T:PBW}, the same must be true for the associated graded algebra $\gr Y_R(\mfg)$, and thus the Yangian $Y_R(\mfg)$. See also \cite[Theorem 1.12]{Ol}, \cite[Theorem 1.7.5]{Mobook}, and \cite[Corollary 3.9]{AMR} for the version of this result corresponding to the case where $\mfg$ is equal to $\mfsl_N, \mfso_N$, or $\mfsp_N$ and $V=\C^N$, which is proven in the exact same way.
\end{proof}
%
%
\section{Structure of the extended Yangian}\label{Sec:XR}
Using the results of the previous section one can extract a fair amount of information about the extended Yangian $X_\mcI(\mfg)$, and in fact prove several results which are known to hold when $\mfg$ is a classical Lie algebra and $V$ is its vector representation. Making this explicit is the main goal of the current section.

%
%
\subsection{The tensor product decomposition, the center, and the PBW theorem}\label{ssec:X-str}
In this subsection we will prove that $X_\mcI(\mfg)$ is isomorphic to the tensor product of a polynomial algebra in countably many variables with the Yangian $Y_R(\mfg)$. This will allow us to deduce a Poincar\'{e}-Birkhoff-Witt type theorem for $X_\mcI(\mfg)$ and also to obtain a complete description of its center.

For brevity we shall denote the polynomial algebra $\C[\my_\lambda^{(r)}\,:\,\lambda\in\mcI, \, r\geq 1]$ simply by $\C[\my_\lambda^{(r)}]_{\lambda,r}$.
\begin{definition}
We define the auxiliary algebra $\mathscr{X}_\mcI(\mfg)$ to be the tensor product of $\C[\my_\lambda^{(r)}]_{\lambda,r}$ with $Y_R(\mfg)$:
\begin{equation*}
 \mathscr{X}_\mcI(\mfg)=\C[\my_\lambda^{(r)}]_{\lambda,r}\otimes Y_R(\mfg). 
\end{equation*}
\end{definition}
Our present goal is to prove the deformed version of Proposition \ref{P:wtphi-gen}. Namely, we will prove that $X_\mcI(\mfg)$ and $\mathscr{X}_\mcI(\mfg)$ are isomorphic algebras. Define $\mathscr{Y}(u)\in \mcE\ot (\C[\my_\lambda^{(r)}]_{\lambda,r})[\![u^{-1}]\!]$ by 
\begin{equation*}
 \mathscr{Y}(u)=I+\sum_{\lambda\in \mcI} X_\lambda^\bullet \ot \mathscr{y}_\lambda(u),\quad\text{ where }\quad \my_\lambda(u)=\sum_{r\geq 1} \my_\lambda^{(r)}u^{-r}.
\end{equation*}
It will also be convenient to expand $\mathscr{Y}(u)=\sum_{i,j=1}^N E_{ij} \ot y_{ij}(u)$ with $\my_{ij}(u)=\delta_{ij}+\sum_{\lambda\in \mcI} c_{ij}^\lambda \my_\lambda(u)$ for all $1\leq i,j\leq N$: see \eqref{ChangeofBasis}.  

Set $\mathscr{T}(u)=\mathscr{Y}(u)\mcT(u)\in \End V\ot \mathscr{X}_\mcI(\mfg)[\![u^{-1}]\!]$, and denote by $\mathscr{t}_{ij}(u)=\delta_{ij}+\sum_{r\geq 1}\mathscr{t}_{ij}^{(r)}u^{-r}$  the $(i,j)$-th entry of $\mathscr{T}(u)$ (that is, $\mT(u)=\sum_{i,j=1}^N E_{ij}\ot \mt_{ij}(u)$). We then have 
\begin{equation}
 \mathscr{t}_{ij}^{(r)}=\tau_{ij}^{(r)}+\my_{ij}^{(r)}+\sum_{a=1}^N \sum_{c=1}^{r-1} y_{ia}^{(c)}\tau_{aj}^{(r-c)}\quad \forall \;1\leq i,j\leq N,\, r\geq 1. \label{tij-exp}
\end{equation}
The degree assignment $\deg \my_\lambda^{(r)}=r-1$ for all $\lambda \in \mcI$ and $r\geq 1$ defines a grading on the polynomial algebra $\C[\my_\lambda^{(r)}]_{\lambda,r}$. Let $\mbG_k$ denote the subspace spanned by monomials of degree equal to $k$ and denote the direct sum $\oplus_{i=0}^k \mbG_k$ by $\mbH_k$. In particular, 
we have $\my_{ij}^{(r)}\in \mbG_{r-1}$ for all $r\geq 1$ and $1\leq i,j\leq N$. After equipping $\mathscr{X}_\mcI(\mfg)$
with the tensor product filtration defined by 
\begin{equation*}
\mathbf{F}_r(\mathscr{X}_\mcI(\mfg))=\sum_{k+l=r} \mbH_k\otimes \mbF_l(Y_R(\mfg))=\bigoplus_{a=0}^r \mbG_a\ot \mbF_{r-a}(Y_R(\mfg)), 
\end{equation*}
it becomes a filtered algebra with $\gr \mathscr{X}_\mcI(\mfg)\cong\C[\my_\lambda^{(r)}]_{\lambda,r}\ot \gr Y_R(\mfg)$. It is immediate from \eqref{tij-exp} that the following relations are satisfied in $\gr \mathscr{X}_\mcI(\mfg)$:
\begin{equation*}
\bar{\mathscr{t}}_{ij}^{(r)}=\bar\tau_{ij}^{(r)}+\bar \my_{ij}^{(r)},\quad \forall \; 1\leq i,j\leq N,\, r\geq 1.
\end{equation*}
Here $\bar{\mathscr{t}}_{ij}^{(r)}$ and $ \bar \my_{ij}^{(r)}$ denote the images of $\mathscr{t}_{ij}^{(r)}$ and $\my_{ij}^{(r)}$, respectively, in $\mbF_{r-1}(\mX_\mcI(\mfg))/\mbF_{r-2}(\mX_\mcI(\mfg))\subset \gr \mX_\mcI(\mfg)$.

It follows from Theorem \ref{T:PBW} that the assignment $(\bar \my_{ij}^{(r)},\bar{\tau}_{ij}^{(r)})\mapsto (\my_{ij}^{(r)}, F_{ij}^{(r-1)})$ for all $r\geq 1$ and $1\leq i,j\leq N$  extends to an isomorphism $\gr \mathscr{X}_\mcI(\mfg) \iso\C[\my_\lambda^{(r)}]_{\lambda,r}\otimes U(\mfg_\rho[z])$. Composing with the inverse of the isomorphism $\phi_\mcI^z$ of Proposition \ref{P:wtphi-gen} (after identifying $\mathcal{K}_\lambda^{(r)}$ with $\my_\lambda^{(r)}$) yields an isomorphism 
\begin{equation}
 \varphi_{\mathscr{X}}:\gr \mathscr{X}_\mcI(\mfg)\iso U(\mfg_\mcI[z]), \quad \bar{\mathscr{t}}_{ij}^{(r)}\mapsto \mathds{F}_{ij}^{(r-1)}=F_{ij}^\mcI z^{r-1}\quad \forall \; 1\leq i,j\leq N,\, r\geq 1.\label{mcX->gex[z]}
\end{equation}
\begin{remark} In Step 2 of the proof  of Proposition \ref{P:X-cur}, it was useful to expand $T(u)$ with respect to the basis $\{X_\lambda^\bullet\}_{\lambda\in \Lambda^\bullet}$ of $\End V$. It is also sometimes more natural to expand $\mathscr{T}(u)$ and $\mcT(u)$ in this way.  Setting
 $\mathscr{t}_\lambda^{(r)}=\sum_{i,j=1}^N a_{ij}^\lambda \mathscr{t}_{ij}^{(r)}$ and $\tau_\lambda^{(r)}=\sum_{i,j=1}^N a_{ij}^\lambda \tau_{ij}^{(r)}$ for each $\lambda\in \Lambda^\bullet$ and $r\geq 1$, we obtain 
\begin{equation*}
 \mathscr{T}(u)=I+\sum_{\lambda\in \Lambda^\bullet}X_\lambda^\bullet \ot \mathscr{t}_\lambda(u)\quad \text{ and }\quad \mcT(u)=I+\sum_{\lambda\in \Lambda^\bullet}X_\lambda^\bullet \ot \tau_\lambda(u),
\end{equation*}
where $(\mathscr{t}_\lambda(u),\tau_\lambda(u))=(\sum_{r\geq 1}\mathscr{t}_\lambda^{(r)}u^{-r}, \sum_{r\geq 1}\tau_\lambda^{(r)}u^{-r})$ for all  $\lambda\in\Lambda^\bullet$.
We then have 
\begin{equation*}
\bar{\mathscr{t}}_\lambda^{(r)}=\begin{cases}
                                 \bar\tau_\lambda^{(r)} & \text{ if }\lambda\in \Lambda,\\
                                 \bar{\my}_\lambda^{(r)} & \text{ if }\lambda\in \mcI,\\
                                 0 & \text{ otherwise}, 
                                \end{cases} \label{mscT-filt}
\end{equation*}
in $\gr \mX_\mcI(\mfg)$, and the isomorphism $\varphi_\mX$ from \eqref{mcX->gex[z]}  is also determined by  $\bar \tau_{\lambda}^{(r)}\mapsto \mathds{X}_\lambda^{(r-1)}$ for all 
$\lambda \in \Lambda$ and $\bar \my_{\lambda}^{(r)}\mapsto \mathds{X}_\lambda^{(r-1)}$ for all $\lambda \in \mcI$: see Subsection \ref{ssec:PCA-ext}. 
\end{remark}
The next theorem is the first main result of this section, and, as previously suggested, it may be viewed as the Yangian analogue of Proposition \ref{P:wtphi-gen}. 
\begin{theorem}\label{T:XR->CxYR}
 The assignment $T(u)\mapsto \mathscr{T}(u)$ extends uniquely to yield an algebra isomorphism 
\begin{equation*}
\Phi_{\mcI}: X_\mcI(\mfg)\iso \mathscr{X}_\mcI(\mfg)=\C[\my_\lambda^{(r)}]_{\lambda,r}\otimes Y_R(\mfg).
\end{equation*}
\end{theorem}
\begin{proof}
Since $\mathscr{Y}(u)\in \mcE\ot (\C[\my_\lambda^{(r)}]_{\lambda,r})[\![u^{-1}]\!]$  and $\mcT(u)$ satisfies the $RTT$-relation \eqref{RTT-V}, the same argument as used to prove Lemma \ref{L:auto} shows that $\mathscr{T}(u)=\mathscr{Y}(u)\mcT(u)$ also satisfies \eqref{RTT-V}. Therefore, $\Phi_{\mcI}: T(u)\mapsto \mathscr{T}(u)$ extends uniquely to an algebra homomorphism $X_\mcI(\mfg)\to \mathscr{X}_\mcI(\mfg)$. By \eqref{tij-exp}, $\Phi_{\mcI}$ is filtration preserving. To prove that $\Phi_{\mcI}$ is an isomorphism, we will follow the same argument as employed to prove Theorem \ref{T:YR->YJ} and show that the associated graded morphism $\gr \Phi_{\mcI}$ is an isomorphism. 

The composition $\gr \Phi_{\mcI}\circ \varphi_\mcI$, where $\varphi_\mcI:U(\mfg_\mcI[z])\onto \gr X_\mcI(\mfg)$ is the epimorphism of Proposition \ref{P:X-cur}, sends $ \mathds{F}_{ij}^{(r-1)}$ to $\bar{\mathscr{t}}_{ij}^{(r)}$  for all $r\geq 1$ and $1\leq i,j\leq N$. Composing with the isomorphism $\varphi_{\mathscr{X}}:\gr \mathscr{X}_\mcI(\mfg)\iso U(\mfg_\mcI[z])$ defined in \eqref{mcX->gex[z]} therefore gives the identity map $\mathrm{id}_{U(\mfg_\mcI[z])}$. 
This implies that $\gr \Phi_{\mcI}$ is indeed an isomorphism, and the same must be true of $\Phi_{\mcI}$.
\end{proof}
Our next goal is to use Theorem \ref{T:XR->CxYR} to obtain a complete description of the center of $X_\mcI(\mfg)$, and to prove a Poincar\'{e}-Birkhoff-Witt theorem for $X_\mcI(\mfg)$. 
We will need a few preliminary lemmas, the first being a consequence of Theorem \ref{T:YR->YJ}.  
\begin{lemma}\label{L:YT=TY}
The generating matrix $\mcT(u)$ belongs to $\rho(Y(\mfg))\ot Y_R(\mfg)[\![u^{-1}]\!]\subset \End V\ot  Y_R(\mfg)[\![u^{-1}]\!]$. Consequently, 
\begin{equation*}
 \mY(u)\mcT(u)=\mcT(u)\mY(u)\quad \text{ and } \quad \mY(u)\mT(u)=\mT(u)\mY(u) \quad \text{ in }\quad \End V \ot \mX_\mcI(\mfg). 
\end{equation*}
\end{lemma}
\begin{proof}
Since $\mcR(u)\in (Y(\mfg)\ot Y(\mfg))[\![u^{-1}]\!]$, Theorem \ref{T:YR->YJ} implies the first part of the Lemma. As $\mY(u)\in \mcE\ot (\C[\my_\lambda^{(r)}]_{\lambda,r})[\![u^{-1}]\!]$ and $\mcE$ is the centralizer of $\rho(Y(\mfg))$ in $\End V$,  $[\mY(u),\mcT(u)]=0=[\mY(u),\mT(u)]$. \qedhere
\end{proof}
Next, define  $\mcY(u)$ to be the preimage of $\mY(u)$ under $\Phi_\mcI$:
\begin{equation*}
 \mcY(u)=I+\sum_{\lambda\in\mcI} X_\lambda^\bullet \ot y_\lambda(u)=\Phi_\mcI^{-1}(\mathscr{Y}(u))\in \mcE\ot X_\mcI(\mfg)[\![u^{-1}]\!],
\end{equation*}
and write $y_\lambda(u)=\sum_{r\geq 1} y_\lambda^{(r)}u^{-r}$. As was the case for $\mathscr{Y}(u)$, we shall also make use of the expansion of $\mcY(u)$ with respect to the basis of elementary matrices $\{E_{ij}\}_{1\leq i,j\leq N}$. That is, we may write 
\begin{equation*}
 \mcY(u)=\sum_{i,j=1}^N E_{ij}\ot y_{ij}(u)\quad \text{ with }\quad  y_{ij}(u)=\delta_{ij}+\sum_{\lambda\in \mcI} c_{ij}^\lambda y_\lambda(u)\quad \forall \; 1\leq i,j\leq N.
\end{equation*}
For each $\lambda\in \mcI$ (resp. $1\leq i,j\leq N$) and $r\geq 1$, the element $y_\lambda^{(r)}$ (resp. $y_{ij}^{(r)}$) belongs to $\mbF^\mcI_{r-1}$, and we will denote by $\bar{y}_\lambda^{(r)}$ (resp. $\bar{y}_{ij}^{(r)}$) its image in the quotient $\mbF^\mcI_{r-1}/\mbF^\mcI_{r-2}=\gr_{r-1}X_\mcI(\mfg)$. 
\begin{lemma}\label{L:Z-Y} The following statements hold: 
\begin{enumerate}
 \item \label{Z-Y:1} $\mathcal{Z}(u)=\mcY(u)\mcY(u+\tfrac{1}{2}c_\mfg)^{-1}\in \mcE \ot X_\mcI(\mfg)[\![u^{-1}]\!]$, 
 \item \label{Z-Y:2}  $z_{ij}^{(r+1)}\in \mbF^\mcI_{r-1}$ for all $1\leq i,j\leq N$ and $r\geq 0$ (where $\mbF^\mcI_{-1}=\{0\}$), 
 \item  \label{Z-Y:3} $\bar{z}_{ij}^{(r+1)}=\tfrac{r}{2}c_\mfg \bar{y}_{ij}^{(r)} \quad \forall \; 1\leq i,j\leq N$ and $r\geq 0$, where $\bar z_{ij}^{(r+1)}$ denotes the image of $z_{ij}^{(r+1)}$ in $\gr_{r-1}X_\mcI(\mfg)$. 
\end{enumerate}

\end{lemma}
\begin{proof} Consider \eqref{Z-Y:1}.
 Since $\mY(u)$ is an invertible element of $(\mcE\ot \C[\my_\lambda^{(r)}]_{\lambda,r})[\![u^{-1}]\!]\cong \mcE\ot (\C[\my_\lambda^{(r)}]_{\lambda,r})[\![u^{-1}]\!]$, we obtain an automorphism $S_\mY$ of $\C[\my_\lambda^{(r)}]_{\lambda,r}$ which is determined by $\mY(u)\mapsto \mY(u)^{-1}$. Consider the tensor product $S_\mX=S_\mY\otimes S_R$, where we recall from \eqref{Hopf-Y_R} that $S_R$ is the antipode of $Y_R(\mfg)$, and it is given by $\mcT(u)\mapsto \mcT(u)^{-1}$. Then $S_\mX$ is the anti-automorphism 
 of the algebra $\mathscr{X}_\mcI(\mfg)=\C[\my_\lambda^{(r)}]_{\lambda,r}\otimes Y_R(\mfg)$ completely determined by 
 \begin{equation*}
 S_\mX(\mathscr{T}(u))=\mY(u)^{-1}\mcT(u)^{-1}=\mcT(u)^{-1}\mY(u)^{-1}=\mathscr{T}(u)^{-1},
 \end{equation*}
 where in the second equality we have appealed to Lemma \ref{L:YT=TY}. Consequently, $S_\mX\circ \Phi_{\mcI}=\Phi_{\mcI}\circ S_\mcI$. Since 
 $\Phi:Y_R(\mfg)\to Y(\mfg)$ is a Hopf algebra morphism and $(1\ot S^2)\mcR(-u)=\mcR(-u-\tfrac{1}{2}c_\mfg)$, we have  $S_R^2(\mcT(u))=\mcT(u+\tfrac{1}{2}c_\mfg)$. Therefore,
 \begin{equation*}
  \Phi_{\mcI}(\mcZ(u))=\Phi_{\mcI}(S_\mcI^2(T(u))T(u+\tfrac{1}{2}c_\mfg)^{-1})=S_\mX^2(\mathscr{T}(u))\mathscr{T}(u+\tfrac{1}{2}c_\mfg)^{-1}=S_\mX^2(\mY(u))\mY(u+\tfrac{1}{2}c_\mfg)^{-1}.
 \end{equation*}
 Since $S_\mX$ becomes an automorphism when restricted to $\C[\my_\lambda^{(r)}]_{\lambda,r}$ (namely $S_\mY$) and $S_\mX(\mY(u))=\mY(u)^{-1}$, we have $S_\mX^2(\mY(u))=\mY(u)$, and we may thus conclude that $\Phi_{\mcI}(\mcZ(u))=\mY(u)\mY(u+\tfrac{1}{2}c_\mfg)^{-1}$, and hence that $\mathcal{Z}(u)=\mcY(u)\mcY(u+\tfrac{1}{2}c_\mfg)^{-1}$. Since $\mcE=\End_{Y(\mfg)}V$ is an algebra, $\mathcal{Z}(u)$ also belongs to  $\mcE \ot X_\mcI(\mfg)[\![u^{-1}]\!]$. This observation concludes the proof of \eqref{Z-Y:1}. 
 
\noindent \textit{Proof of \eqref{Z-Y:2}.}
The $(i,j)$-th entry of the $u^{-r-1}$ coefficient of $\mcY(u+\tfrac{1}{2}c_\mfg)^{-1}$ is equal to
$-y_{ij}^{(r+1)} \mod \mbF^\mcI_{r-1}$. It is a straightforward consequence of this fact that the $u^{-r-1}$ coefficient of the $(i,j)$-th entry of $\mcY(u)\mcY(u+\tfrac{1}{2}c_\mfg)^{-1}$, which is equal to $z_{ij}^{(r+1)}$, is contained in $\mbF^\mcI_{r-1}$.

\noindent \textit{Proof of \eqref{Z-Y:3}.}
The argument we give is similar to the proof of Proposition \ref{P:Y-cur}. By \eqref{Z-Y:1}, we have 
\begin{equation*}
\mathcal{Z}(u)\mcY(u+\tfrac{1}{2}c_\mfg)=\mcY(u).   
\end{equation*}
 Taking the $(i,j)$-th coefficient of both sides yields $\sum_{a=1}^N z_{ia}(u)y_{aj}(u+\tfrac{1}{2}c_\mfg)=y_{ij}(u)$. Writing $y_{aj}(u+\tfrac{1}{2}c_\mfg)=\sum_{r\geq 0} y_{aj}^{\circ (r)} u^{-r}$, we have $y_{aj}^{\circ (r)}\in \mbF^\mcI_{r-1}$ for each $r\geq 0$ and 
 \begin{equation}
  y_{ij}(u)=\sum_{a=1}^N z_{ia}(u)y_{aj}(u+\tfrac{1}{2}c_\mfg)=y_{ij}(u+\tfrac{1}{2}c_\mfg)+z_{ij}(u)+\sum_{a=1}^N\sum_{k,s\geq 1} z_{ia}^{(k)}y_{aj}^{\circ (s)}u^{-k-s} \label{eq:y-z}
 \end{equation}
By \eqref{Z-Y:2}, $z_{ia}^{(k)}y_{aj}^{\circ (s)}\in \mbF^\mcI_{k+s-3}$. Thus, the coefficient of $u^{-r-1}$ in the summation on the right-hand side of the above equality is contained in $\mbF^\mcI_{r-2}$. On the other hand, the same argument as used to establish \eqref{eq:s-check} allows us to deduce that the  $u^{-r-1}$ coefficient of $y_{ij}(u+\tfrac{1}{2}c_\mfg)$ is equivalent to  $y_{ij}^{(r+1)}-\tfrac{r}{2}c_\mfg y_{ij}^{(r)}$ modulo $\mbF^\mcI_{r-2}$. Thus, \eqref{eq:y-z} implies that 
\begin{equation*}
 y_{ij}^{(r+1)}\equiv y_{ij}^{(r+1)}+z_{ij}^{(r+1)}-\tfrac{r}{2}c_\mfg y_{ij}^{(r)} \mod \mbF^\mcI_{r-2},
\end{equation*}
and hence that $\bar{z}_{ij}^{(r+1)}=\tfrac{r}{2}c_\mfg \bar{y}_{ij}^{(r)}$ for all $1\leq i,j\leq N,\, r\geq 0$. 
\end{proof}
For each $\lambda\in \Lambda^\bullet$, set $z_\lambda(u)=\sum_{r\geq 1} z_\lambda^{(r)} u^{-r}$ with $z_\lambda^{(r)}=\sum_{i,j=1}^N a_{ij}^\lambda z_{ij}^{(r)}$. Then, by Part \eqref{Z-Y:1} of Lemma \ref{L:Z-Y}, 
\begin{equation*}
 \mcZ(u)=I+\sum_{\lambda\in \Lambda^\bullet} X_\lambda^\bullet \ot z_\lambda(u)=I+\sum_{\lambda\in \mcI}X_\lambda^\bullet \ot z_\lambda(u).
\end{equation*}

The following Proposition gives a complete description of the center of $X_\mcI(\mfg)$ in terms of the coefficients $z_\lambda^{(r)}$ of $\mcZ(u)$. 
\begin{proposition}\label{P:Z(u)}
Let $ZX_\mcI(\mfg)$ denote the center of $X_\mcI(\mfg)$. 
The set of elements $\{y_\lambda^{(r)}\}_{\lambda\in \mcI, r\geq 1}$ is algebraically independent and generates $ZX_\mcI(\mfg)$, and the same is true of the set $\{z_\lambda^{(r)}\}_{\lambda\in \mcI, r\geq 2}$. Consequently, 
\begin{equation*}
  \C[y_\lambda^{(r)}\,:\,\lambda \in \mcI,\, r\geq 1]\cong ZX_\mcI(\mfg)\cong \C[z_\lambda^{(r)}\,:\,\lambda \in \mcI,\, r\geq 2].
\end{equation*}
\end{proposition}
\begin{proof}
 By Corollary \ref{C:ZY=1}, the center of $\mX_\mcI(\mfg)$ is equal to the polynomial algebra $\C[\my_\lambda^{(r)}]_{\lambda,r}$. Since the isomorphism 
 $\Phi_\mcI$ of Theorem \ref{T:XR->CxYR} satisfies $\Phi_\mcI(y_\lambda^{(r)})=\my_\lambda^{(r)}$ for all $\lambda\in \mcI$ and $r\geq 1$, the set $\{y_\lambda^{(r)}\}_{\lambda\in\mcI, r\geq 1}$ must be an algebraically independent set which generates the center of $X_\mcI(\mfg)$. In particular,
 $ZX_\mcI(\mfg)\cong \C[y_\lambda^{(r)}\,:\,\lambda \in \mcI,\, r\geq 1]$. 
 
 Since the coefficients $\{z_\lambda^{(r)}\}_{\lambda\in \mcI,r\geq 2}$ are central, the assignment  $y_\lambda^{(r)}\mapsto z_\lambda^{(r+1)}$, for all $\lambda\in \mcI$ and $r\geq 1$, extends to an algebra endomorphism
 \begin{equation*}
 \varphi_{y,z}:ZX_\mcI(\mfg)\cong \C[y_\lambda^{(r)}\,:\,\lambda \in \mcI,\, r\geq 1]\to ZX_\mcI(\mfg).
 \end{equation*}
By Part \eqref{Z-Y:2} of Lemma \ref{L:Z-Y}, $\varphi_{y,z}$ is a filtered morphism, and by Part \eqref{Z-Y:3} of Lemma \ref{L:Z-Y} the associated graded morphism 
$\gr \varphi_{y,z}$ is just the rescaling automorphism of $\C[y_\lambda^{(r)}\,:\,\lambda \in \mcI,\, r\geq 1]$ which sends $y_\lambda^{(r)}$ to $2 (r c_\mfg)^{-1}y_\lambda^{(r)}$ for each $\lambda\in \mcI$ and $r\geq 1$. Thus $\varphi_{y,z}$ is an automorphism of $ZX_\mcI(\mfg)$ and hence $\{z_\lambda^{(r)}\}_{\lambda \in \mcI, r\geq 2}$ is an algebraically independent set which generates the center $ZX_\mcI(\mfg)$ of $X_\mcI(\mfg)$. \qedhere
\end{proof}
Using Theorem \ref{T:XR->CxYR} or, more accurately, its proof,  we obtain the following Poincar\'{e}-Birkhoff-Witt theorem for $X_\mcI(\mfg)$: 
\begin{theorem}\label{T:X-PBW}
 The surjective homomorphism $\varphi_\mcI:U(\mfg_\mcI[z])\to \gr X_\mcI(\mfg)$, $\mathds{F}_{ij}^{r-1}\mapsto \bar{t}_{ij}^{(r)}$, of Proposition \ref{P:X-cur} is an isomorphism of algebras. As a consequence, the assignment 
 \begin{equation}
 F_{ij}^\mcI\mapsto t_{ij}^{(1)} \quad \forall \; 1\leq i,j\leq N \label{g_I->X_I}
 \end{equation}
 defines an embedding $U(\mfg_\mcI)\into X_\mcI(\mfg)$, while the assignment
 \begin{equation}
 F_{ij}\mapsto t_{ij}^{(1)}-2c_\mfg^{-1}z_{ij}^{(2)} \quad \forall \; 1\leq i,j\leq N \label{g_rho->X_I}
 \end{equation}
 defines an embedding $U(\mfg_\rho)\into X_\mcI(\mfg)$.  
\end{theorem}
\begin{proof}
 The injectivity of $\varphi_\mcI$ was proven in the course of the proof of Theorem \ref{T:XR->CxYR}, and that \eqref{g_I->X_I} defines an embedding follows immediately. 
 
 As for the last statement of the theorem, consider 
 the embedding $\iota_R:Y_R(\mfg)\into X_\mcI(\mfg)$, $\mcT(u)\mapsto \mcY(u)^{-1}T(u)$. It sends $\tau_{ij}^{(1)}$ to $t_{ij}^{(1)}-y_{ij}^{(1)}$ for all $1\leq i,j\leq N$. 
 Composing with the embedding $U(\mfg_\rho)\into Y_R(\mfg)$, $F_{ij}\mapsto \tau_{ij}^{(1)}$ of Theorem \ref{T:PBW}, we obtain an injection 
 $U(\mfg_\rho)\into X_\mcI(\mfg)$ which is given by $F_{ij}\mapsto t_{ij}^{(1)}-y_{ij}^{(1)}$ for all $1\leq i,j\leq N$. The proof that this coincides with \eqref{g_rho->X_I} is completed by noting that, by Part \eqref{Z-Y:3} of Lemma \ref{L:Z-Y}, we have $y_{ij}^{(1)}=2 c_\mfg^{-1} z_{ij}^{(2)}$ for all $1\leq i,j\leq N$. 
\end{proof}
%
%
%
\subsection{The Yangian as a  Hopf subalgebra of the extended Yangian}\label{ssec:Y->X}
By Theorem \ref{T:XR->CxYR}, $Y_R(\mfg)$ may also be identified as a subalgebra of $X_\mcI(\mfg)$ via the embedding 
\begin{equation*}
\iota_{R}:Y_R(\mfg)\into X_\mcI(\mfg), \quad \mcT(u)\mapsto \mcY(u)^{-1}T(u),
\end{equation*}
which played a role in the proof of Theorem \ref{T:X-PBW}. In this subsection we study $Y_R(\mfg)$ from this viewpoint, our  main goals being to show that 
$\iota_R$ is a Hopf algebra morphism, to study the behaviour of the center under the coproduct $\Delta_\mcI$, and to show that $Y_R(\mfg)$ can in fact be realized as a fixed point subalgebra of $X_\mcI(\mfg)$. 

In order to distinguish between the identifications of $Y_R(\mfg)$ as a quotient and as a subalgebra of $X_\mcI(\mfg)$, we shall denote by $\wt Y_R(\mfg)\subset X_\mcI(\mfg)$ the isomorphic copy of $Y_R(\mfg)$ obtained from the embedding $\iota_{R}$. We also set $\wt \mcT(u)=\mcY(u)^{-1}T(u)=\sum_{i,j}E_{ij}\otimes \wt \tau_{ij}(u)$. 

The first and main step in showing that $\iota_R$ is a morphism of Hopf algebras is to study the behaviour of $\mcY(u)$ under the coproduct, counit, and antipode of $X_\mcI(\mfg)$. This is the purpose of the next lemma. 
\begin{lemma}\label{L:Y-Hopf}
 The central matrix $\mcY(u)$ satisfies 
 \begin{equation*}
  \Delta_\mcI(\mcY(u))=\mcY_{[1]}(u)\mcY_{[2]}(u), \quad S_\mcI(\mcY(u))=\mcY(u)^{-1},\quad \eps_\mcI(\mcY(u))=I.  
 \end{equation*}
\end{lemma}
\begin{proof}
We have already demonstrated in the course of the proof of Lemma \ref{L:Z-Y} that $S_\mcI(\mcY(u))=\mcY(u)^{-1}$. More precisely, we showed that $S_\mX(\mY(u))=\mY(u)^{-1}$ where 
$S_\mX$ is the anti-automorphism of $\mX_\mcI(\mfg)$ determined by $S_\mX(\mT(u))=\mT(u)^{-1}$. Since $S_\mX=\Phi_\mcI\circ S_\mcI\circ \Phi_\mcI^{-1}$, this implies that 
$S_\mcI(\mcY(u))=\mcY(u)^{-1}$. 

The Hopf algebra axioms dictate that $\eps_\mcI \circ S_\mcI=\eps_\mcI$, and hence $\eps_\mcI(\mcY(u))=\eps_\mcI(\mcY(u))^{-1}$. The equality $\eps_\mcI(\mcY(u) \mcY(u)^{-1})=I$ then implies that $\eps_\mcI(\mcY(u))^2=I$. Since the identity matrix $I$ is the unique square root of itself belong to $I+u^{-1}(\End V)[\![u^{-1}]\!]$, we can conclude that $\eps_\mcI(\mcY(u))=I$. 

It remains to see that $\Delta_\mcI(\mcY(u))=\mcY_{[1]}(u)\mcY_{[2]}(u)$. Let $\Delta_\mathscr{Y}$ be the algebra morphism $\C[\my_\lambda^{(r)}]_{\lambda,r}\to \C[\my_\lambda^{(r)}]_{\lambda,r}\ot \C[\my_\lambda^{(r)}]_{\lambda,r}$ determined by 
\begin{equation*}
\Delta_\mathscr{Y}(\mY(u))=\mY_{[1]}(u)\mY_{[2]}(u)\in \mcE\ot (\C[\my_\lambda^{(r)}]_{\lambda,r}\ot \C[\my_\lambda^{(r)}]_{\lambda,r})[\![u^{-1}]\!].
\end{equation*}
We then obtain 
an algebra morphism $\Delta_{\mX}=\sigma_{23}\circ (\Delta_\mY\ot \Delta_R):\mX_\mcI(\mfg)\to \mX_\mcI(\mfg)\ot \mX_\mcI(\mfg)$, where $\sigma_{23}=\mathrm{id}_{\C[\my_\lambda^{(r)}]_{\lambda,r}}\ot \sigma\ot \mathrm{id}_{Y_R(\mfg)}$ and $\sigma: Y_R(\mfg)\ot \C[\my_\lambda^{(r)}]_{\lambda,r}\to\C[\my_\lambda^{(r)}]_{\lambda,r}\ot Y_R(\mfg)$ is the flip map. By definition,  
\begin{equation*}
 \Delta_\mX(\mT(u))=\mY_{[1]}(u)\mY_{[2]}(u) \mcT_{[1]}(u)\mcT_{[2]}(u)\in \End V\ot (\mX_\mcI(\mfg)\ot \mX_\mcI(\mfg))[\![u^{-1}]\!].
\end{equation*}
Since $\mY_{[2]}(u)$ commutes with $\mcT_{[1]}(u)$, we can rewrite this as
\begin{equation*}
 \Delta_\mX(\mT(u))=\mY_{[1]}(u)\mcT_{[1]}(u)\mY_{[2]}(u) \mcT_{[2]}(u)=\mT_{[1]}(u)\mT_{[2]}(u),
\end{equation*}
and hence $(\Phi_\mcI\ot \Phi_\mcI)\circ \Delta_\mcI=\Delta_\mX\circ \Phi_\mcI$. This implies that $\Delta_\mcI=(\Phi_\mcI^{-1}\ot \Phi_\mcI^{-1})\circ \Delta_\mX \circ \Phi_\mcI$, and consequently 
\begin{equation*}
 \Delta_\mcI(\mcY(u))=(\Phi_\mcI^{-1}\ot \Phi_\mcI^{-1})(\mY_{[1]}(u)\mY_{[2]}(u))=\mcY_{[1]}(u)\mcY_{[2]}(u). \qedhere
\end{equation*}
\end{proof}

The above lemma leads us to the first main result of this subsection. Let $\eps_\mY$ be the homomorphism $\C[\my_\lambda^{(r)}]_{\lambda,r}\to \C$, $\mY(u)\mapsto I$, and recall that $\Phi_\mcI:X_\mcI(\mfg)\to \mX_\mcI(\mfg)$ is the algebra isomorphism of Theorem \ref{T:XR->CxYR}.
\begin{proposition}\label{P:Hopf}
 $\C[\my_\lambda^{(r)}]_{\lambda,r}$ is a Hopf algebra with coproduct $\Delta_\mY$, counit $\eps_\mY$ and antipode $S_\mY$, and if 
 $\mX_\mcI(\mfg)$ is equipped with the standard tensor product of Hopf algebras structure,  $\Phi_\mcI:X_\mcI(\mfg)\to \mX_\mcI(\mfg)$ becomes an isomorphism of Hopf algebras. In particular,  The embedding $\iota_R:Y_R(\mfg)\into X_\mcI(\mfg)$ is a morphism of Hopf algebras. 
\end{proposition}
\begin{proof}
 $\mX_\mcI(\mfg)$ becomes a Hopf algebra, and $\Phi_\mcI$ a Hopf algebra isomorphism, after being equipped with coproduct $(\Phi_\mcI\ot \Phi_\mcI)\circ \Delta_\mcI\circ \Phi_\mcI^{-1}$ (which, by Lemma \ref{L:Y-Hopf}, is $\Delta_\mX$), 
 counit $\eps_\mcI \circ \Phi_\mcI^{-1}$ (which, by Lemma \ref{L:Y-Hopf}, is $\eps_\mX$), and antipode $\Phi_\mcI \circ S_\mcI \circ \Phi_\mcI^{-1}$ (which, by Lemma \ref{L:Y-Hopf}, is $S_\mX$). Since the tuple $(\Delta_\mY,\eps_\mY,S_\mY)$ coincides with $({\Delta_\mX}|_{\C[\my_\lambda^{(r)}]_{\lambda,r}}, {\eps_\mX}|_{\C[\my_\lambda^{(r)}]_{\lambda,r}}, {S_\mX}|_{\C[\my_\lambda^{(r)}]_{\lambda,r}})$, 
 it endows $\C[\my_\lambda^{(r)}]_{\lambda,r}$ with the structure of a Hopf algebra. 
 
 Since $\Delta_\mX=\sigma_{23}\circ (\Delta_\mY\ot \Delta_R)$, $\eps_\mX=\eta\circ (\eps_\mY\ot \eps_R)$ (where $\eta:\C\ot \C\to \C$ is the natural isomorphism), and 
 $S_\mX=S_\mY\ot S_R$, the Hopf algebra structure on $\mX_\mcI(\mfg)$ induced from $X_\mcI(\mfg)$ via $\Phi_\mcI$ coincides with the Hopf algebra structure obtained via the standard tensor product of Hopf algebras construction. 
\end{proof}
Before moving onto the last main result of this subsection, we note the following corollary of Lemma \ref{L:Y-Hopf}.
\begin{corollary}\label{C:Hopf-Z}
The central matrix $\mcZ(u)$ satisfies 
\begin{equation*}
 \Delta_\mcI(\mcZ(u))=\mcY_{[1]}(u)\mcZ_{[2]}(u)\mcY_{[1]}(u+\tfrac{1}{2}c_\mfg)^{-1},\quad S_\mcI(\mcZ(u))=\mcY(u)^{-1}\mcY(u+\tfrac{1}{2}c_\mfg),\quad \eps_\mcI(\mcZ(u))=I.
\end{equation*}
\end{corollary}
\begin{proof}
 By Lemma \ref{L:Z-Y}, $\mcZ(u)=\mcY(u)\mcY(u+\tfrac{1}{2}c_\mfg)^{-1}$. Therefore, by Lemma \ref{L:Y-Hopf}, we have
 \begin{equation*}
  \Delta_\mcI(\mcZ(u))=\mcY_{[1]}(u)\mcY_{[2]}(u)\mcY_{[2]}(u+\tfrac{1}{2}c_\mfg)^{-1}Y_{[1]}(u+\tfrac{1}{2}c_\mfg)^{-1}=\mcY_{[1]}(u)\mcZ_{[2]}(u)\mcY_{[1]}(u+\tfrac{1}{2}c_\mfg)^{-1}.
 \end{equation*}
 Similarly, $\eps_\mcI(\mcZ(u))=\eps(\mcY(u))\eps(\mcY(u+\tfrac{1}{2}c_\mfg))^{-1}=I$. Lastly, since the restriction of $S_\mcI$ to the center $ZX_\mcI(\mfg)$ is an automorphism, $S_\mcI(\mcZ(u))=S_\mcI(\mcY(u))S_\mcI(\mcY(u+\tfrac{1}{2}c_\mfg))^{-1}=\mcY(u)^{-1}\mcY(u+\tfrac{1}{2}c_\mfg)$. \qedhere 
\end{proof}
Recall that, by Lemma \ref{L:auto}, for each $(f_\lambda(u))_{\lambda\in \mcI}\in \prod_{\lambda\in \mcI}(u^{-1}\C[\![u^{-1}]\!])_\lambda$ there is an automorphsim $ m_\mbf$ of $X_\mcI(\mfg)$ determined by the assignment \eqref{aut:m_F}. The next theorem proves that $\wt Y_R(\mfg)$ can be realized as a fixed point subalgebra of $X_\mcI(\mfg)$. 
\begin{theorem}\label{T:fixed-pt}
 The Yangian $\wt Y_R(\mfg)$ is equal to the subalgebra of $X_\mcI(\mfg)$ fixed by all automorphisms $m_\mbf$:
 \begin{equation}\label{Y-m_F}
  \wt Y_R(\mfg)=\left\{Y\in X_\mcI(\mfg)\,:\, m_\mbf(Y)=Y\quad \forall \; (f_\lambda(u))_{\lambda\in \mcI}\in \prod_{\lambda\in \mcI}(u^{-1}\C[\![u^{-1}]\!])_\lambda\right\}.
 \end{equation}
\end{theorem}
\begin{proof}
 Recall from \eqref{mbF-tuple}  that $(f_\lambda(u))_{\lambda\in \mcI}$ is identified with the matrix $\mbf^\circ(u)=\sum_{\lambda\in \mcI} X_\lambda^\bullet \ot f_\lambda(u)$, and that, by \eqref{aut:m_F}, $m_\mbf(T(u))=\mbf(u) T(u)$, where $\mbf(u)=I+\mbf^\circ(u)$. Let us denote the right-hand side of \eqref{Y-m_F} by $X_\mcI(\mfg)^{m_{\mbf}}$. 
 
For each $(f_\lambda(u))_{\lambda\in \mcI}\in \prod_{\lambda\in \mcI}(u^{-1}\C[\![u^{-1}]\!])_\lambda$, the assignment $\mY(u)\mapsto \mbf(u)\mY(u)$ extends to an automorphism $m_{\mbf}^\mY$ of $\C[\my_\lambda^{(r)}]_{\lambda,r}$. Consider the automorphism $m_{\mbf}^\mX=m_{\mbf}^\mY\ot \mathrm{id}$ of $\mX_\mcI(\mfg)$. It satisfies 
\begin{equation*}
 m_\mbf^\mX(\mT(u))=m_\mbf^\mX(\mY(u))m_\mbf^\mX(\mcT(u))=\mbf(u)\mT(u),
\end{equation*}
and thus $m_\mbf^\mX\circ \Phi_\mcI=\Phi_\mcI\circ m_\mbf$. It follows that $m_\mbf(\mcY(u))=\mbf(u)\mcY(u)$ for every tuple  $(f_\lambda(u))_{\lambda\in \mcI}$.
Therefore, for each element $(f_\lambda(u))_{\lambda\in \mcI}\in \prod_{\lambda\in \mcI}(u^{-1}\C[\![u^{-1}]\!])_\lambda$,
\begin{equation*}
m_\mbf(\wt \mcT(u))=m_\mbf(\mcY(u))^{-1} m_\mbf(T(u))=\mcY(u)^{-1}\mbf(u)^{-1} \mbf(u) T(u)=\mcY(u)^{-1}T(u)=\wt \mcT(u).
\end{equation*}
This proves that $\wt Y_R(\mfg)\subset X_\mcI(\mfg)^{m_\mbf}$. 

To obtain the reverse inclusion, we employ similar techniques as used to prove \cite[Theorem 3.1]{AMR}. Suppose towards a contradiction that there is $X\in X_\mcI(\mfg)^{m_\mbf}\setminus \wt Y_R(\mfg)$. By Theorem \ref{T:XR->CxYR} we may write $X$ as a polynomial in the variables $\{y_\lambda^{(r)}\}_{\lambda\in \mcI,r\geq 1}$ with coefficients in $\wt Y_R(\mfg)$. This polynomial is non-constant by assumption. Only finitely many variables can appear in this polynomial, so there is $m\geq 1$  such that $X$ depends only on the variables $\{y_\lambda^{(r)}\}_{\lambda\in \mcI,r=1,\ldots,m}$. We take $m$ to be minimal with this property, and we fix $\mu \in \mcI$ such that 
$X$ depends on $y_\mu^{(m)}$. 

Let $X=\sum_{a\geq 0} X_a (y_\mu^{(m)})^a$ be the expansion of $X$ as a polynomial in the single variable $y_\mu^{(m)}$ and set $P(y_\mu^{(m)})=\sum_{a\geq 1} X_a (y_\mu^{(m)})^a$. The polynomial $P(y_\mu^{(m)})$ has degree at least $1$, as otherwise $X$ would not depend on $y_\mu^{(m)}$. For each $w\in \C$, define $\mbf_w^\circ(u)=(f_\lambda(u))_{\lambda\in \mcI}$ by
\begin{equation*}
 f_\lambda(u)=\begin{cases}
               0 & \text{ if } \; \lambda\neq \mu ,\\
               wu^{-m} & \text{ if } \; \lambda=\mu.
              \end{cases}
\end{equation*}
As a matrix in $\mcE\ot u^{-1}\C[\![u^{-1}]\!]$, $\mbf_w^\circ(u)=X_\mu^\bullet \ot wu^{-m}$. Note that 
\begin{equation*}
\mbf_w(u)\mcY(u)=(I+X_\mu^\bullet \ot w u^{-m})\mcY(u)=\mcY(u)+X_\mu^\bullet \ot wu^{-m}+(X_\mu^\bullet \ot wu^{-m})\mcY^\circ(u),
\end{equation*}
where $\mcY^\circ(u)=\mcY(u)-I$. This implies that, for $1\leq r\leq m$ and $\lambda\in \mcI$, the image of $y_\lambda^{(r)}$ under $m_{\mbf_w}$ is given by
\begin{equation*}
m_{\mbf_w}(y_\lambda^{(r)})=\begin{cases}
                             y_\lambda^{(r)} \; & \text{ if }\; (\lambda,r)\neq (\mu,m),\\
                             y_\mu^{(m)}+w \; & \text{ if }\; (\lambda,r)=(\mu,m).
                            \end{cases}
\end{equation*}
Consequently, $X=m_{\mbf_w}(X)=X_0+m_{\mbf_w}(P(y_\mu^{(m)}))=X_0+P(y_\mu^{(m)}+w)$ for each $w\in \C$. Here $P(y_\mu^{(m)}+w)$ is the polynomial obtained from $P(y_\mu^{(m)})$ by substituting $y_\mu^{(m)}\mapsto y_\mu^{(m)}+w$. This allows us to deduce that 
\begin{equation}\label{P-van}
 P(y_\mu^{(m)})=P(y_\mu^{(m)}+w) \quad \forall \; w\in \C. 
\end{equation}
For each $w\in \C$, let $\mathrm{ev}_w$ be the algebra endomorphism of $\C[y_\lambda^{(r)}]_{\lambda,r}$ given by $y_\lambda^{(r)}\mapsto y_\lambda^{(r)}$ for all $(\lambda,r)\neq (\mu,m)$ and $y_\mu^{(m)}\mapsto -w$. Note that
$\cap_{w\in \C} \Ker(\mathrm{ev}_w)=\{0\}$. We can extend $\mathrm{ev}_w$ to obtain an endomorphism $\mathrm{ev}_w^\mX$ of $\mX_\mcI(\mfg)$ by setting 
$\mathrm{ev}_w^\mX=\mathrm{ev}_w\ot \mathrm{id}$. We then have $\Ker(\mathrm{ev}_w^\mX)=\Ker(\mathrm{ev}_w)\ot Y_R(\mfg)$ and $\cap_{w\in \C}\Ker(\mathrm{ev}_w^\mX)=\{0\}$.  

The equality \eqref{P-van} implies that $\mathrm{ev}_w^\mX(\Phi_\mcI(P(y_\mu^{(m)})))=0$ for all $w\in \C$. This shows that $\Phi_\mcI(P(y_\mu^{(m)}))=0$, and thus that $P(y_\mu^{(m)})=0$. This contradicts the fact that $P(y_\mu^{(m)})$ is a non-constant polynomial of degree at least $1$. Thus no such $X$ can exist, and we may conclude that 
$\wt Y_R(\mfg)=X_\mcI(\mfg)^{m_\mbf}$. \qedhere
\end{proof}

%
%
\section{Drinfeld's theorem and classical Lie algebras}\label{Sec:Drin-Vec}

When $V$ is assumed to be irreducible, one can recover from the results of Sections \ref{Sec:RTT}, \ref{Sec:YR->YJ} and \ref{Sec:XR} a proof of \cite[Theorem 6]{Dr1}. Our first task is to formalize this statement: this will be accomplished in Subsection \ref{ssec:Drin}. We will conclude in Subsection \ref{ssec:class} by explaining how many of the results of this paper reduce to, and have been motivated by, results which are known to hold when $V$ is the vector representation of a classical Lie algebra $\mfg$. 
%
%
\subsection{Drinfeld's theorem and the irreducibility assumption}\label{ssec:Drin}
We now restrict our attention to the setting where the underlying $Y(\mfg)$-module $V$ is irreducible. As has been explained in Remark \ref{R:rational}, this situation has additional practical value, since, at least in principle, $R(u)$ can be computed by solving the equation \eqref{inter} and, after a suitable re-normalization, is equal to a rational $R$-matrix. 

Since $V$ is irreducible, Schur's lemma implies that $\mcE=\End_{Y(\mfg)}V=\C\cdot I$.
In particular, the indexing set $\mcI$ contains a single element, say $\varsigma$, and the basis element $X_\varsigma^\bullet$ of $\mcE$ can be chosen to equal the identity matrix $I$.
With this in mind, we shall henceforth denote $X_\mcI(\mfg)$ simply by $X(\mfg)$ whenever $V$ is assumed to be irreducible.

Set $z(u)=1+\sum_{r\geq 2}z_r u^{-r}=1+z_\varsigma(u)$ and $y(u)=1+\sum_{r\geq 1}y_r u^{-r}=1+y_\varsigma(u)$. The observation made in the previous paragraph implies the first part of the following result.
\begin{corollary}\label{C:z,y}
The matrices $\mcZ(u)$ and $\mcY(u)$ are equal to $z(u)\cdot I$ and $y(u)\cdot I$, respectively. In particular, $z(u)$ is uniquely determined by the relation 
\begin{equation*}
S^2_\mcI(T(u))T(u+\tfrac{1}{2}c_\mfg)^{-1}=z(u)\cdot I= T(u+\tfrac{1}{2}c_\mfg)^{-1}S^2_\mcI(T(u)). \label{z(u)}
\end{equation*}
\end{corollary}
\begin{proof}
 The relation $z(u)\cdot I=S^2_\mcI(T(u))T(u+\tfrac{1}{2}c_\mfg)^{-1}$ is immediate from \eqref{Y-sym}. This relation, together with the centrality of $z(u)$, implies that 
 $z(u)T(u+\tfrac{1}{2}c_\mfg)=T(u+\tfrac{1}{2}c_\mfg)z(u)=S^2_\mcI(T(u))$, and hence that  $z(u)\cdot I= T(u+\tfrac{1}{2}c_\mfg)^{-1}S^2_\mcI(T(u))$. 
\end{proof}

These simplifications allow us to write down a proof of the following theorem, whose first two parts are precisely the statement of \cite[Theorem 6]{Dr1}. 
\begin{theorem}[Theorem 6 of \cite{Dr1}]\label{T:Drin}  The following three statements are satisfied:
 \begin{enumerate}
  \item \label{Dr:1}There is an epimorphism of Hopf algebras $\wt \Phi: X(\mfg)\onto Y(\mfg)$ such that 
  \begin{equation*}
  \wt \Phi(T(u))=(\rho\ot1)(\mcR(-u)).
  \end{equation*}
  \item \label{Dr:2} There is a series $c(u)=1+\sum_{r\geq 1} c_r u^{-r}$, whose coefficients $\{c_r\}_{r\geq 1}$ are central and  generate $\Ker\, \wt \Phi$ as an ideal, which satisfies
  \begin{equation*}
    \Delta_\mcI(c(u))=c(u)\otimes c(u). \label{grouplike}
  \end{equation*}
  \item \label{Dr:3} The coefficients of $c(u)$ generate the center of $X(\mfg)$, which is a polynomial algebra in countably many variables. 
 \end{enumerate}
\end{theorem}
\begin{proof}
 The first statement is precisely Lemma \ref{L:X->YJ}, which we have seen holds even when $V$ is not irreducible. Let us turn to \eqref{Dr:2}. There are two natural candidates for 
 the series $c(u)$, the first being $z(u)$ and the second being $y(u)$, and both satisfy the desired properties. If $c(u)=z(u)$, then by Corollaries \ref{C:Hopf-Z} and \ref{C:z,y} we have 
 \begin{equation*}
  \Delta_\mcI(z(u))=(y(u)\ot 1)(1\ot z(u))(y(u+\tfrac{1}{2}c_\mfg)^{-1}\ot 1)=z(u)\ot z(u), 
 \end{equation*}
 while Theorem \ref{T:YR->YJ} gives $\Ker\, \wt\Phi=(z(u)-1)$. If instead $c(u)=y(u)$, then it is immediate from Lemma \ref{L:Y-Hopf} and Corollary \ref{C:z,y} that $c(u)$ satisfies the grouplike property \eqref{grouplike}. As the ideal $(y(u)-1)$ generated by the coefficients $\{y_r\}_{r\geq 1}$ is equal to $(z(u)-1)$, we also have $\Ker\, \wt\Phi=(y(u)-1)$. 
 
 As for part \eqref{Dr:3}, Proposition \ref{P:Z(u)} and Corollary \ref{C:z,y} guarantee that both $\{z_r\}_{r\geq 2}$ and $\{y_r\}_{r\geq 1}$ are algebraically independent sets which generate $ZX(\mfg)$.
\end{proof}
\begin{remark}\label{R:Drin}
 More generally, when $V$ is not assumed to be irreducible, we have shown that $C(u)=\mcY(u)\in I+\mcE\ot u^{-1}X_\mcI(\mfg)[\![u^{-1}]\!]$ has central coefficients which generate 
 the ideal $(\mcZ(u)-I)=\Ker\, \wt \Phi$, and moreover that $C(u)$ satisfies  $\Delta(C(u))=C_{[1]}(u)C_{[2]}(u)$. This should be viewed as a generalization of \eqref{Dr:2}, and the statement that the coefficients $y_\lambda^{(r)}$ of $C(u)$ are algebraically independent generators of $ZX_\mcI(\mfg)$ (see Proposition \ref{P:Z(u)}) should be viewed as a generalization of \eqref{Dr:3}. 
\end{remark}
In the proof of Theorem \ref{T:Drin} we have observed that the series $z(u)$ is grouplike. It is thus also the case that $S_\mcI(z(u))=z(u)^{-1}$ (as can also be seen from Corollary \ref{C:Hopf-Z}). The next corollary summarizes these results. 
\begin{corollary}
When $V$ is irreducible the formulas of Corollary \ref{C:Hopf-Z} reduce to 
\begin{equation*}
\Delta_\mcI(z(u))=z(u)\ot z(u),\quad S_\mcI(z(u))=z(u)^{-1},\quad \eps_\mcI(z(u))=1.
\end{equation*}
\end{corollary}
We conclude this subsection by noting that, since $\mcE=\C\cdot I$, every automorphism $m_\mbf$ (see \eqref{aut:m_F}) takes the form $T(u)\mapsto f(u)T(u)$ for a series $f(u)\in 1+u^{-1}\C[\![u^{-1}]\!]$ uniquely determined by $\mbf(u)=I\ot f(u)$. With this in mind, we will denote $m_\mbf$ by $m_f$ for the remainder of this paper. 

%
%
\subsection{The vector representation of the Yangian of a classical Lie algebra}\label{ssec:class}
We now narrow our focus to the case where $\mfg$ is a Lie algebra of classical type and $V$ is specialized to its vector representation, our goal being to briefly highlight results in the literature which have motivated some of the results of this paper, with emphasis on the results of Section \ref{Sec:XR}.

We remark that these specializations fall into the slightly more general framework in which the representation $V$ of $Y(\mfg)$ is irreducible as a $\mfg$-module. Considering only such modules leads to fairly significant simplifications. For instance,  $\mfg_\mcI$ always coincides with $\mfg_\mcJ$ and hence Subsection \ref{ssec:g_I} and Step 2 of the proof of Proposition \ref{P:X-cur} are no longer needed. 
There are, however, examples where $R(u)$ has been computed when $V$ is not irreducible as a $\mfg$-module: see \cite{ChPr1}. 
%
%
%
\subsubsection{The special linear Lie algebra $\mfsl_N$}
Fix $N\geq 2$, let $\{e_1,\ldots,e_N\}$ denote the standard basis of $\C^N$, and view $\mfg=\mfsl_N$ as the space of traceless $N\times N$ matrices. Fixing the invariant form $(\cdot,\cdot)$ to be the trace form, 
we have $\Omega_\rho=P-\frac{1}{N} I$ and $c_\mfg=2N$, where $P=\sum_{i,j=1}^N E_{ij}\ot E_{ji}$ is the permutation operator $\sigma$ on $\C^N\ot \C^N$. Additionally, we have  $\mfg_\mcJ=\mfg_\mcI\cong \mfgl_N$. 

It is well known that the $\mfsl_N$-module $\C^N$ admits a $Y(\mfsl_N)$-module structure defined  by allowing $J(X)$, for each $X\in \mfsl_N$, to act as the 
zero operator: see for instance Example 1 of \cite{Dr1}. In this case $(\rho\ot \rho)(\mcR(-u))$ is, up to multiplication by a formal series in $u^{-1}$, equal to Yang's $R$-matrix 
\begin{equation} \label{R(u)-sl}
 R(u)=I-Pu^{-1}, 
\end{equation}
as can be deduced by directly solving the equation \eqref{inter} with $V=W=\C^N$. The associated extended Yangian $X(\mfsl_N)$ is usually denoted $Y(\mfgl_N)$ in the literature, and has been studied extensively.  In what follows we do not attempt to provide a full account of the history behind each result, but instead refer the reader to the appropriate results in the monograph \cite{Mobook} where a detailed bibliography is given. 

The central series $y(u)$ and $z(u)$ (adapting the notation from Corollary \ref{C:z,y}) both admit rather concrete descriptions. The series $y(u)$ is equal to the 
series $\wt{d}(u)$ which has appeared in the proof of \cite[Theorem 1.8.2]{Mobook}: it is the unique central series in $1+u^{-1}ZX(\mfsl_N)[\![u^{-1}]\!]$ such that 
\begin{equation*}
 \wt{d}(u)\wt{d}(u-1)\cdots \wt{d}(u-N+1)=\mathrm{qdet} T(u),
\end{equation*}
where $\mathrm{qdet} T(u)$ is the quantum determinant of the generating matrix $T(u)$: see Definition 1.6.5 of \cite{Mobook}. By \cite[Proposition 1.6.6]{Mobook}, it is given by 
\begin{equation*}
 \mathrm{qdet} T(u)=\sum_{\pi\in \mfS_N} \mathrm{sign}(\pi)\cdot t_{\pi(1),1}(u)\cdots t_{\pi(N),N}(u-N+1). 
\end{equation*}
The series $z(u)$ is related to the series
\begin{equation*}
 \mathsf{z}(u)=\frac{\mathrm{qdet} T(u-1)}{\mathrm{qdet}T(u)},
\end{equation*}
which was defined in \cite[(1.68)]{Mobook}, by $z(u)=\mathsf{z}(u+N)$, as can be seen using \cite[Theorem 1.9.9]{Mobook}.  
The relation $z(u)=1$ is equivalent to $\mathrm{qdet} T(u)=1$, as was pointed out in the original statement of \cite[Theorem 6]{Dr1}.

Theorem \ref{T:XR->CxYR} reduces to the statements of Theorems 1.7.5 and 1.8.2 of \cite{Mobook}, and Proposition \ref{P:Z(u)} follows from these same results together with \cite[Corollary 1.9.7]{Mobook}. The Poincar\'{e}-Birkhoff-Witt theorem for $X(\mfsl_N)$ (Theorem \ref{T:X-PBW} with $(\mfg,V)=(\mfsl_N,\C^N)$) is given in \cite[Theorem 1.4.1]{Mobook}. 

The description of $Y_R(\mfsl_N)$ as the subalgebra of $X(\mfsl_N)$ consisting of all elements stable under all automorphisms of the form $m_f$, which is provided by Theorem \ref{T:fixed-pt}, was actually taken as the definition of $Y_R(\mfsl_N)$ in \cite{Mobook}. It was then proven in Corollary 1.8.3 of \cite{Mobook} that $Y_R(\mfsl_N)$ could be equivalently characterized as in Definition \ref{D:Y}. According to \cite[Bibliographical notes 1.8]{Mobook}, the description of $Y_R(\mfsl_N)$ using the automorphisms $m_f$ is originally due to Drinfeld, as is the more general fact that $Y_R(\mfsl_N)$ can be realized as a subalgebra of $X(\mfsl_N)$: see 
Theorem 1.13 of \cite{Ol}.

%
%
\subsubsection{The orthogonal and symplectic Lie algebras $\mfso_N$ and $\mfsp_{2n}$} 
Still assuming $N\geq 2$, let $n\in \mathbb{N}$ be defined by $N=2n$ (if $N$ is even) 
and $N=2n+1$ (if $N$ is odd). We now assume that $\mfg_N=\mfg$ is either equal to $\mfso_N$ or $\mfsp_{N}$, where $N$ is necessarily even in the latter case. It is convenient to relabel 
the standard basis of $\C^N$ using the indexing set $\mcI_N=\{-n,\ldots,-1,(0),1,\ldots,n\}$, where 
$(0)=0$ if $N=2n+1$ and should be omitted otherwise. That is, we denote the standard basis of $\C^N$ by 
$\{e_{-n},\ldots,e_{-1}, (e_0), e_1,\ldots,e_N\}$. Let $t: \End \C^N \to \End \C^N$ denote the transposition determined  by
\begin{equation*}
 (E_{ij})^t=\theta_{ij}E_{-j,-i} \quad \text{ where } \quad \theta_{ij}=\begin{cases}
                                                                         1 & \; \text{ if }\; \mfg_N=\mfso_N,\\
                                                                         \mathrm{sign}(i)\mathrm{sign}(j)& \; \text{ if }\; \mfg_N=\mfsp_N.
                                                                    \end{cases}
\end{equation*}
The Lie algebra $\mfg_N$ can then be realized as the Lie subalgebra of $\mfgl_N$ spanned by the elements $F_{ij}=E_{ij}-(E_{ij})^t$, and as this notation suggests the corresponding presentation is consistent with that provided by Proposition \ref{P:new-g}. We refer the reader to (2.4) and (2.5) of \cite{AMR} for an explicit description of the defining relations. 

Letting $(\cdot,\cdot)$ be equal to one half of the trace form, we have $\Omega_\rho=P-Q$ and $c_\mfg=4\ka$, where 
\begin{equation*}
 P=\sum_{i,j\in \mcI_N} E_{ij}\ot E_{ji}, \quad Q=P^{t_2}=\sum_{i,j\in \mcI_N} \theta_{ij} E_{ij} \ot E_{-i,-j},\quad \text{ and } \quad \ka=\begin{cases}
                                                                                                                                          N/2-1 & \; \text{ if }\; \mfg_N=\mfso_N,\\
                                                                                                                                          n+1  & \; \text{ if }\; \mfg_N=\mfsp_N.
                                                                                                                                         \end{cases}
\end{equation*}
As in the $\mfg=\mfsl_N$ case, it is well known that the vector representation $\C^N$ of $\mfg_N$ extends to a representation of $Y(\mfg_N)$ by setting $\rho(J(X))=0$ for all $X\in \mfg_N$. For an explicit proof see \cite[Proposition 3.1]{GRW}. The $R$-matrix $(\rho\ot \rho)(\mcR(-u))$ can be computed from \eqref{inter} and is equal to
\begin{equation}\label{R(u)-sosp}
 R(u)=I-Pu^{-1}+Q(u-\ka)^{-1},
\end{equation}
up to multiplication by an invertible element of $\C[\![u^{-1}]\!]$. This has certainly been known for a long time (see \cite{KS2} and \cite[Example 2]{Dr1}), but for a complete proof 
we refer the reader to Proposition 3.13 of the recent paper \cite{GRW}.
The $RTT$-Yangian $Y_R(\mfg_N)$ and the extended Yangian $X(\mfg_N)$ have not been studied to the same extent as their $\mfsl_N$ analogues, although there has been an increase in efforts over the last fifteen years \cite{AACFR,AMR,MoMuWalg,MoMuWalg2,GRW,JLM}.  

It was proven in \cite{AACFR} (see also \cite[(2.26)]{AMR}) that there is a central series $\mathsf{z}(u)\in 1+u^{-1}ZX(\mfg_N)[\![u^{-1}]\!]$ determined by 
\begin{equation*}
 \mathsf{z}(u)\cdot I=T^t(u+\ka)T(u)=T(u)T^t(u+\ka), \quad \text{ where }\quad T^t(u)=\sum_{i,j\in \mcI_N} (E_{ij})^t\ot t_{ij}(u).
\end{equation*}
By comparing (2.31) of \cite{AMR} with the relation $S^2_\mcI(T(u))=z(u)T(u+2\ka)$ of Corollary \ref{C:z,y}, we can conclude that 
\begin{equation*}
 z(u)=\frac{\mathsf{z}(u)}{\mathsf{z}(u+\ka)}. 
\end{equation*}
Conversely $y(u)$ is equal to the central series $\mathsf{y}(u)$ defined in \cite[Theorem 3.1]{AMR}: it is uniquely determined by $y(u)y(u+\ka)=\mathsf{z}(u)$. It was also noted in the statement of \cite[Theorem 6]{Dr1} that, when $(\mfg,V)=(\mfso_N,\C^N$), the coefficients of $\mathsf{z}(u)-1$ generate the kernel of the epimorphism $\wt \Phi$ from Lemma \ref{L:X->YJ} as an ideal.

Theorem \ref{T:XR->CxYR} with $(\mfg,V)=(\mfg_N,\C^N)$ is precisely Theorem 3.1 of \cite{AMR}, while Corollary \ref{P:Z(u)} is deduced from that same theorem of \cite{AMR} together with \cite[Corollary 3.9]{AMR}. The Poincar\'{e}-Birkhoff-Witt theorem for $X(\mfg_N)$ when $V=\C^N$ was stated and proven in Corollary 3.10 of \cite{AMR}: see also \cite[Theorem 3.6]{AMR}, which is exactly Theorem \ref{T:PBW} in the particular case being discussed.  

Just as was the case for $\mfg=\mfsl_N$ with $V=\C^N$, the authors of \cite{AMR} first defined $Y_R(\mfg_N)$ as the fixed point subalgebra of $X(\mfg_N)$ under all automorphisms $m_f$, and then  in  \cite[Corollary 3.2]{AMR} proved that it could be equivalently defined as a quotient of $X(\mfg_N)$.
%
%
 
\end{document}